\documentclass[11pt]{amsart}
\usepackage[utf8]{inputenc} 
\usepackage{latexsym,amsmath,amsfonts,amsthm,amssymb,color}
\usepackage{fullpage}
\usepackage{tikz} 
\usepackage{bigints} 
\newcommand{\red}[1]{{%\color{red}
    #1}}

\colorlet{symbols}{black} % Change this to change color of symbols
\tikzset{
	dot/.style={circle,fill=symbols,draw=symbols,inner sep=0pt,minimum size=0.4pt},
	basic/.style={draw=symbols},
	>=stealth,
	}

\input shuffle

\newtheorem{theorem}{Theorem}
\newtheorem{lemma}[theorem]{Lemma}
\newtheorem{proposition}[theorem]{Proposition}
\newtheorem{cor}[theorem]{Corollary}
\newtheorem{definition}[theorem]{Definition}
\newtheorem{remark}[theorem]{Remark}

\newcommand{\R}{{\mathbb R }}
\newcommand{\N}{{\mathbb N }}

\newcommand{\C}{{\mathbb C }}
\newcommand{\Z}{{\mathbb Z }}

\renewcommand{\Z}{{\mathbb Z}}
\newcommand{\step}{{\mathfrak S}}

\newcommand{\Mu}{M}

\DeclareMathOperator{\real}{\mathrm{Re}}
\DeclareMathOperator{\im}{\mathrm{Im}}
\DeclareMathOperator{\dist}{\mathrm{dist}}

\DeclareMathOperator{\trace}{\mathrm{tr}}

\newcommand{\newsection}[1]{\section{#1}\setcounter{theorem}{0}
 \setcounter{equation}{0}\par\noindent}
\renewcommand{\theequation}{\arabic{section}.\arabic{equation}}
\renewcommand{\thesubsection}{{\arabic{section}.\arabic{subsection}}}
\renewcommand{\thetheorem}{\arabic{section}.\arabic{theorem}}

\begin{document}
\title{Conserved energies for the cubic NLS in 1-d. }

\author{Herbert Koch}
\address{Mathematisches Institut   \\ Universit\"at Bonn }

\author{ Daniel Tataru}
\address {Department of Mathematics \\
  University of California, Berkeley}

\begin{abstract}
  We consider the cubic Nonlinear Schr\"odinger Equation (NLS) as well as the
  modified Korteweg-de Vries Equation (mKdV) in one space dimension.
  We prove that for each $s>-\frac12$ there exists a conserved energy
  which is equivalent to the $H^s$ norm of the solution.  For the
  Korteweg-de Vries Equation (KdV) there is a similar conserved energy for
  every $s\ge -1$. 

%No context for phrase below

% We bound the transmission coefficients and provide
%   formulas for the parts of degree 2,4 and 6  for  all the
%   energies for NLS and mKdV, as well as those of degree 2,3 and 4  for KdV.
\end{abstract}
\maketitle

\maketitle 

\tableofcontents

\newsection{Introduction}
We consider the (de)focusing cubic Nonlinear Schr\"odinger  equation (NLS) 
\begin{equation}
i u_t + u_{xx} \pm 2 u |u|^2 = 0, \qquad u(0) = u_0,
\label{nls}\end{equation}
and the complex (de)focusing modified Korteweg-de Vries equation (mKdV) 
\begin{equation}
u_t + u_{xxx} \pm 2  (|u|^2u)_x  = 0, \qquad u(0) = u_0,
\label{mkdv}\end{equation}
with real or complex solutions in one space dimension, where the ``$+$'' sign corresponds
to the focusing problems.
The NLS equation \eqref{nls} 
is invariant with respect to the scaling
\[
u(x,t) \to \lambda u(\lambda x,\lambda^2 t),
\]
and mKdV \eqref{mkdv} is invariant with respect to 
\[
u(x,t) \to \lambda u(\lambda x,\lambda^3 t). 
\]
The initial datum for the two problems scales in the same way,
\begin{equation} \label{scaling}
u_0(x) \to \lambda u_0(\lambda x),
\end{equation}
 and so does the Sobolev space $\dot H^{-\frac12}$, which one may view as the
critical Sobolev space. 

The reason we consider the two flows simultaneously is that they are commuting 
Hamiltonian flows. They are in effect part of an infinite family of commuting 
Hamiltonian flows with respect to the symplectic form
\[
\omega(u,v) = \im \int u \bar v \ dx.
\]
Each of these Hamiltonians yield joint conservation laws for all of 
these flows. The first  several energies are as follows (see \eqref{Hj2}, Lemma \ref{lHj4} and Lemma \ref{lHj6} where formulas for the terms of degree 2,4, and 6 in $u$ are given):

\begin{equation} \label{eq:energies} 
\begin{split} 
H_0=&    \int |u|^2 dx, \\
 H_1=&   \frac1i \int u  \partial_x
              \bar u dx ,
\\ H_2=&  \int |u_x|^2 + |u|^4 dx, 
\\
H_3 = & i \int \ u_x \partial_x \overline{u}_x + 3  |u|^2 u \overline{u}_x dx,   
 \\
H_4 =&  \int  |u_{xx}|^2 + \frac32 |(u^2)_x|^2+   ||u|^2_x|^2   + 2 |u|^6 dx.
\end{split}
\end{equation}

The even ones are even with respect to complex conjugation and have a
positive definite principal part, and we will refer to them as
energies. The odd ones are odd under the replacement of $u$ by its
complex conjugate, and we will refer to them as momenta.  With respect
to the symplectic form above these commuting Hamiltonians generate
flows as follows: $H_0$ generates the phase shifts, $H_1$ generates
the group of translations, $H_2$ the NLS flow, $H_3$ the \red{m}KdV flow,
etc.

In this article we prove that we can extend this countable family of conservation laws
to a continuous family. For  all $s > -\frac12$ we construct conserved energies 
$E_s$ associated to $H^s$ solutions for our equations.  Our construction of these energies
relies heavily on the scattering transform associated to these problems, which
requires some extensive preliminaries. For the reader's benefit we will now state 
a preliminary form of our main result, which makes no reference to the scattering transform.
A more complete version will be presented at the end of the next section, after 
a substantial review of the scattering transform; see Theorem~\ref{t:main+}.

\begin{theorem}\label{energies} 
  There exists $\delta > 0$ so that for each $s > -\frac12$ and both for the focusing
  and defocusing case there exists an energy functional
\[
E_s :  \{ u \in H^s, \  \| u\|_{l^2 DU^2} \leq \delta \}   \to \R^+
\]
with the following properties:
\begin{enumerate}
\item $E_s$ is conserved along the NLS and mKdV flow.
\item  $E_s$ agrees with the linear $H^s$ energy up to quartic terms,   
\[
\left| E_s(u) - \|u\|_{H^s}^2 \right| \lesssim  \Vert u \Vert_{l^2 DU^2}^2 
\Vert u \Vert_{H^s}^2. 
\]
\item The    map 
\[
 \{ u \in H^\sigma, \  \| u\|_{l^2 DU^2} \leq \delta \}         \times  (-\frac12,\sigma] \ni (u,s) \to  E_s(u)  
\]
is  analytic in $u\in H^\sigma$,  analytic in $s$ for $s < \sigma$ and continuous in $s$ at $s = \sigma$.
\item The  functionals $E_s$ interpolate the energies $H_{2j}$ in the sense that  
\[ E_n(u) = \sum_{j=1}^n \binom{n}{j} H_{2j}(u). \label{identification} \] 
\end{enumerate}
\end{theorem}

The Banach space $l^2 DU^2 = L^2 +DU^2$ is the inhomogeneous version
of the $DU^2$ space, and is first introduced in Section \ref{s:t2j}
and in greater detail in Appendix \ref{a:uv}. It is a replacement for
the unusable scaling critical space $H^{-\frac12}$ and satisfies
\begin{equation}  
\Vert u \Vert_{l^2 DU^2} \lesssim \Vert u \Vert_{B^{-\frac12}_{2,1}} 
\lesssim \Vert u \Vert_{H^{s}}, \qquad s >-\frac12.
\end{equation} 

We note that  for half-integers
the energies $E_{k+\frac12}$ are not directly associated to the momenta.
For the sake of completeness, in Section~\ref{s:further} we discuss also
the construction of generalized momenta $P_s$, so that for half integers
$P_{k+\frac12}$ is a linear combination of the odd Hamiltonians 
$H_1,H_3,\cdots,H_{2k+1}$. 
 
To further clarify the assertions in the theorem, we note that for
simplicity we establish the energy conservation result for regular
initial data. By the local well-posedness theory, this extends to all
$H^s$ data above the (current) Sobolev local well-posedness threshold,
which is $s \geq 0$ for NLS, respectively $s \geq \frac14$ for
mKdV. If $s$ is below these thresholds, then the energy conservation
property holds for all data at the threshold, i.e. for $L^2$ data for
NLS, respectively $H^\frac14$ data for mKdV.  It is not known whether
the two problems are well-posed below these thresholds and above the
scaling; however, it is known that local uniformly continuous
dependence fails, see \cite{MR2376575}.
One key consequence of our result is that if the initial datum is in
$H^s$ then the solutions remain bounded in $H^s$ globally in
time in a uniform fashion:

\begin{cor}\label{c:large-data}
Let $s > -\frac12$, $R > 0$ and $u_0$ be an initial datum for either NLS or mKdV so that 
\[
\| u_0\|_{H^s} \leq R
\]
Then the corresponding solution $u$  satisfies the global bound
\[
\| u(t)\|_{H^s} \lesssim F(R,s) : = \left\{ \begin{array}{ll} R+ R^{1+2s} & s \geq 0 \cr
R+ R^{\frac{1+4s}{1+2s}} &  s < 0 \end{array} \right.
\]
\end{cor}

This follows directly from the above theorem  if $R \ll 1$. For larger 
$R$  it still follows from the theorem, but only after applying the scaling 
\eqref{scaling}. Here one needs to make the choice   $\lambda = cR^{-2}$, with $c \ll 1$ if 
$s \geq 0$, respectively $\lambda = cR^{-\frac{2}{1+2s}}$ if $s < 0$.

This work is a natural continuation of earlier work of the
authors~\cite{MR2995102}, \cite{MR2353092} and of Christ, Colliander
and Tao~\cite{MR2376575}, where apriori $H^s$ bounds for NLS are
obtained for a restricted range $-\frac14 \leq s < 0$, without
explicit use of integrable structures, but also without providing
conserved energies.  The same ideas were later implemented for mKdV by
the second author together with Christ and Holmer ~\cite{MR3058496}
for $s > - \frac18$, as well as for KdV by Liu ~\cite{MR3292346} for
$s > - \frac45$.  Together with Buckmaster the first author has proven
uniform in time $H^{-1}$ bounds for the KdV equation
\cite{MR3400442} using the Miura map but no inverse scattering
techniques.  

In contrast, Theorem~\ref{energies} relies on inverse
scattering methods via the AKNS formalism, see for instance \cite{MR0450815,
  MR1462745}. Likely our ideas here also extend to other completely
integrable systems. As an illustration of this, in Section~\ref{s:kdv} 
we state the similar result for the KdV problem; we also outline its proof,
which conveniently uses exactly the same algebraic structure as the 
NLS/mKdV equations.

  In the periodic case Kappeler, Schaad  and Topalov 
\cite{MR2454610, MR2267286}    and
Kappeler and Grebert \cite{MR3203027} have proved related results for the KdV equation, for mKdV and for the defocusing
NLS relying on complex algebraic geometry. Also in the KdV case, 
there is independent work of Killip-Visan-Zhang~\cite{KVZ} in this direction in the range $-1 \leq s < 1$,
both  in the periodic case and on the real line.

The structure of the paper is as
follows. Sections~\ref{inverse}-\ref{s:proof} contain the proof of our
main result, beginning with an outline of the scattering transform in
Section~\ref{inverse}, which concludes with a more complete form of
our main result, see Theorem~\ref{t:main+}.  This is followed by an
analysis of the terms in the formal series for the conserved energies
in the subsequent sections, and a final summation argument in
Section~\ref{s:proof}.  It also contains a further discussion of our
main results, touching on issues such as conserved momenta and frequency
envelope bounds.  Section~\ref{s:further} provides detailed formulas
for Taylor expansion of $E_s$ of degree $4$ for all $s$, and of degree
$6$ in $u$ for the $H_j$.  Section~\ref{s:kdv} discusses the
corresponding results for the KdV equation. All these equations are
closely related. They are part of hierarchies of integrable equations
with a Lax operator as essential building block. For defocusing NLS
and modified KdV the Lax operator is
\[  \mathcal{L}= i\left( \begin{matrix} \partial_x & -u \\ \bar u & -\partial_x \end{matrix} \right), \]   
in the focusing case it is 
\[  \mathcal{L}= i\left( \begin{matrix} \partial_x & -u \\ -\bar u & -\partial_x \end{matrix} \right), \]   
and for the KdV equation we have 
\[  \mathcal{L}= i\left( \begin{matrix} \partial_x & u \\ 1 & -\partial_x \end{matrix} \right).  \]

There are also three appendices to the paper, all self-contained,
establishing some results which are useful here but may also be of
independent interest.   The first appendix is devoted to a Hopf algebra structure
arising in connection with the multilinear integrals appearing in the
expansion for our conserved energies. We need it for the proof of
Theorem \ref{primitive}.  This drastically simplifies the analysis for
large $s$. The connection of Theorem \ref{primitive} to Hopf algebras
was pointed out to us by Martin Hairer, which we gratefully
acknowledge.  The second appendix is concerned with $U^p$ and
$V^p$ spaces and their properties, which are heavily used in all our
estimates. The third appendix collects results on integrable hierarchies.

\newsection{An outline of the construction}
\label{inverse}

\subsection{ An overview of the scattering transform}
Here we recall some basic facts about the inverse scattering transform
for NLS and mKdV, which can be found in many places and in particular in Faddeev and Takhtajan \cite{MR905674} .  Both the NLS evolution \eqref{nls} and the mKdV
evolution \eqref{mkdv} are completely integrable, so we have at our
disposal the inverse scattering transform conjugating the nonlinear
flow to the corresponding linear flow.  To describe their Lax pairs
 we consider the system
\begin{equation}\begin{split}
\psi_x = &\left(\begin{matrix} -iz & u \\ \bar u  &iz \end{matrix} 
\right) \psi   \\
\psi_t = & i\left(\begin{matrix}-[2z^2 +|u|^2] & -2izu+u_x \\
-2 i z\bar u - \bar u_x  & 2z^2 +|u|^2\end{matrix}\right)\psi   
\end{split} \label{laxpairnls}
\end{equation}
 where $z$ is a complex parameter. The defocusing
NLS equation arises as compatibility condition for the system
\eqref{laxpairnls}: For fixed $z$ there exist two unique solutions
$\psi_1$, $\psi_2$ to \eqref{laxpairnls} with $\psi_1(0,0)=(1,0)$ and
$\psi_2(0,0)=(0,1)$ if and only if $u$ satisfies the NLS  equation.  The above is  often  referred 
to in the literature as the Lax pair for NLS.

If instead we want the canonical form $\mathcal L,\mathcal P$ with 
\[
\mathcal L_t = [\mathcal P,\mathcal L]
\]
then we should view the first equation above as $\mathcal L \psi = z \psi$
where
\[
\mathcal L = i \left(\begin{matrix} \partial_x & - u \\ \bar u  & - \partial_x
 \end{matrix} 
\right) 
\]
and $\mathcal P$ is given by the second matrix in \eqref{laxpairnls} where $z$ has been eliminated 
using the relations $L \psi = z \psi$,
\[
\begin{split}
\mathcal P = & \ 2i \left(\begin{matrix}- 1 &  0 \\
0 & 1 \end{matrix}\right) \mathcal L^2 +
2 \left(\begin{matrix} 0 & u  \\
\bar u   & 0 \end{matrix}\right) \mathcal L + 
i \left(\begin{matrix}-|u|^2 & u_x \\
-\bar u_x  & |u|^2\end{matrix}\right)
\\[1mm]
= & \ - 2i \left(\begin{matrix} -\partial_x^2 + |u|^2  &   u_x \\ - \bar u_x & \partial_x^2 - |u|^2   \end{matrix}\right) + 2 i
\left(\begin{matrix} |u|^2  & -u \partial_x  \\
\bar u \partial_x   & - |u|^2 \end{matrix}\right)  + 
i \left(\begin{matrix}-|u|^2 & u_x \\
-\bar u_x  & |u|^2\end{matrix}\right)
\\[1mm]
= & \ 
i \left(\begin{matrix} 2 \partial_x^2 - |u|^2  &  -  u \partial_x -   \partial_x u \\[1mm]  \bar u \partial_x + \partial_x \bar u & -2 \partial_x^2 + |u|^2   \end{matrix}\right)
\end{split}
\]
It is interesting to note that this operator is formally skew adjoint and it generates a unitary evolution operator $U(t)$ so that 
\[  \mathcal{L}(u(t)) =  U(t) \mathcal{L}(u(0)) U^*(t). \] 
This is equivalent to the pair of Kappeler and Grebert \cite{MR3203027}. 
The same applies for the defocusing mKdV problem with respect to the system
\begin{equation}\begin{split}
\psi_x = &\left(\begin{matrix} -iz & u \\ \bar u  &iz \end{matrix} 
\right) \psi  \\
\psi_t = & i\left(\begin{matrix}-[4z^3 +2z  |u|^2]+ i(u_x \bar u -u \bar u_x) & 
-4iz^2u +2zu_x + i( u_{xx}-2|u|^2u) \\[1mm]
-4iz^2 \bar u -2z\bar u_x + i(\bar u_{xx} - 2|u|^2\bar u)  & 4z^3 +2z|u|^2-i( u_x \bar u -u \bar u_x)    \end{matrix}\right)\psi  
\end{split} \label{laxpairmkdv}
\end{equation}
The operator $\mathcal P$ of the corresponding Lax pair can also be easily computed 
using the same algorithm as above.

The Lax pairs associated to the focusing equations are similar up to
some sign twists resp. the replacement of $\bar u$ by $-\bar
u$. Precisely, for the focusing NLS problem we have the system
\begin{equation}\begin{split}
\psi_x = &\left(\begin{matrix} -iz & u \\ - \bar u  &iz \end{matrix} 
\right) \psi   \\[2mm]
\psi_t = & i\left(\begin{matrix}-[2z^2 - |u|^2] & -2izu+ u_x \\
2iz\bar u + \bar u_x  & 2z^2 - |u|^2\end{matrix}\right)\psi   
\end{split} \label{laxpairnls-f}
\end{equation}
while for the focusing mKdV we have the related system
\begin{equation}\begin{split}
\psi_x = &\left(\begin{matrix} -iz & u \\ - \bar u  &iz \end{matrix} 
\right) \psi  \\
\psi_t = & i\left(\begin{matrix}-[4z^3 -2z |u|^2]-i(u_x\bar u  - u \bar u_x)    & -4i z^2 u +2zu_x+ i(u_{xx}+2|u|^2u)  \\[1mm]
 4i z^2 \bar u +2z \bar u_x+ i(\bar u_{xx} +2|u|^2)\bar u  & 4z^3 -2z|u|^2+ i(u_x\bar u  - u \bar u_x)     \end{matrix}\right)\psi  
\end{split} \label{laxpairmkdv-f}
\end{equation}
Much of this formalism here and below can be found in the seminal paper by Ablowitz, Kaup, Newell and Segur \cite{MR0450815}.  We collect the calculations leading to the Lax pairs above in Appendix \ref{appNLS}.  
The scattering transform associated to both the defocusing NLS and the
defocusing mKdV is defined via the first equation of
\eqref{laxpairnls} resp. \eqref{laxpairmkdv} which we write as linear
system
\begin{equation}\label{scatter}
\left\{ 
\begin{array}{l}
\dfrac{d\psi_1}{dx} = -i z \psi_1 + u \psi_2 
\\[5mm]
\dfrac{d\psi_2}{dx} =  i z \psi_2 + \bar u \psi_1
\end{array}
\right.
\end{equation}
The last $+\bar u $ turns into $-\bar u $ for the focusing case.
In the defocusing case, the scattering data for this problem is
obtained for $z=\xi$, real, by considering the relation between the
asymptotics for $\psi$ at $\pm \infty$.  Precisely, one considers the
Jost solutions $\psi_l$ and $\psi_r$ with asymptotics
\[
\psi_l (\xi,x,t) = \left( \begin{array}{c}    e^{-i\xi x}  \cr 
0 \end{array}\right) + o(1) \ \ \ \text{ as $x \to -\infty$},
\quad \psi_l (\xi,x,t) = \left( \begin{array}{c}   T^{-1}(\xi)  e^{-i\xi x}  \cr 
R(\xi) T^{-1}(\xi)  e^{i\xi x} \end{array}\right) + o(1) \ \ \ \text{ as $x \to \infty$},
\] 
respectively
\[
\psi_r (\xi,x,t) = \left( \begin{array}{c}  L(\xi)  T^{-1}(\xi)  e^{-i\xi x}  \cr 
T^{-1}(\xi)  e^{i\xi x} \end{array}\right) + o(1) \ \ \ \text{ as $x \to -\infty$},
\quad \psi_l (\xi,x,t) = \left( \begin{array}{c}   0  \cr 
e^{i\xi x} \end{array}\right) + o(1) \ \ \ \text{ as $x \to \infty$}.
\] 
These are viewed as initial value problems with datum at $-\infty$,
respectively $+\infty$ \red{and are easy to define for Schwartz functions $u$}.

 We note that the $T$'s in the two solutions
$\psi_l$ and $\psi_r$ are the same since the Wronskian of the two
solutions is constant:
\[ \det ( \psi_l ,\psi_r) \to T^{-1}(\xi) \qquad \text{ for } x \to \pm \infty. \] 
The quantity $|\psi_1|^2 - |\psi_2|^2$ is also conserved,
which shows that on the real line we have
\[
|T| \leq 1, \qquad |T|^2 = 1- |R|^2 = 1-|L|^2  
\]
 Further, we have the symmetry $(\psi_1,\psi_2) \to (\bar
\psi_2,\bar \psi_1)$ which via the Wronskian leads to
\[
L \bar T = - \bar R T
\]

More generally for any $z$ in the \red{open}  upper half plane there exist the Jost 
solutions 
\[
\psi_l (z,x,t) = \left( \begin{array}{c}    e^{-iz x}  \cr 
0 \end{array}\right)+ o(1) e^{\im z x}  \ \ \ \text{ as $x \to -\infty$},
\] 
\[ 
 \psi_l (z,x,t) = \left( \begin{array}{c}   T^{-1}(z)  e^{-iz x}  \cr 
0 \end{array}\right) + o(1)e^{\im zx}  \ \ \ \text{ as $x \to \infty$},
\] 
This provides a holomorphic extension of $T^{-1}$ to the upper half-space.
In the defocusing case it also provides a holomorphic extension of $T$ to the upper half space
with $|T| \leq 1$.  To see that this extension is holomorphic, one needs to show that there are no
$z$ for which $T^{-1}(z) = 0$.  Indeed, a straightforward ODE analysis
shows that such solutions $\psi$ would have to decay exponentially at
both ends,
\[
 \psi= \left( \begin{matrix}  e^{-izx} \\ 0 \end{matrix} \right) (1+ o(1)) 
\text{ as } x \to -\infty,          \qquad   
  \psi= \left( \begin{matrix}  0 \\ e^{izx}  \end{matrix} \right) (d+ o(1))  \text{ as } x\to \infty  
\] 
Then one can view $z$ and $\psi$ as an eigenvalue/eigenfunction for the problem
\[
\mathcal{L} \psi = z \psi, \qquad \mathcal{L} = \left( \begin{array}{cc}   i \partial_x & -i u  \cr 
i \bar u &  -i \partial_x
 \end{array}\right) 
\]
In the defocusing case $\mathcal{L}$ is self-adjoint, so no such
eigenfunctions exist outside the continuous spectrum $\R$ of
$\mathcal{L}$.  
To see that this extension satisfies $|T| \leq 1$ we begin with the simpler case $u \in L^1$,
when one can easily show that in the upper half space we have
\[
\lim_{|z| \to \infty} T(z) = 1 
\]
and then the property  $|T| \leq 1$ on the real line extends to the upper half-space by the maximum 
principle resp. a version of Liouville's theorem. For more general $u \in l^2 DU^2$ this property extends by density.

In the previous discussion we neglected the time dependence. The evolution of 
$T$, $L$ and $R$ can by obtained  as a  consequence of the existence of the Lax pair by considering the $t$ depends near $\pm \infty$. It follows that
as $u$ evolves along the \eqref{nls} flow, the functions $L,R,T$
evolve according to
\[
T_t = 0, \qquad L_t = - 4i \xi^2 L, \qquad R_t = 4i\xi^2 R 
\]
and if $u$ evolves according to \eqref{mkdv}  
\[
T_t = 0, \qquad L_t = - 8i \xi^3 L, \qquad R_t = 8i\xi^3 R.
\]

The scattering map for $u$ is given by 
\[
u \to R,
 \]
and the map $u \to R$ conjugates the NLS flow \eqref{nls} to the
(Fourier transform of) the linear Schroedinger equation, and
simultaneously the mKdV flow to the linear Airy flow. Reconstructing
$u$ from $R$ requires solving a Riemann-Hilbert problem, see
\cite{MR1207209} for this approach for the  mKdV
equation.

A key difference between real and nonreal $z$ is that for real $\xi$,
one essentially needs $u\in L^1$ in order to define the scattering
data $L(\xi)$ and $T(\xi)$ in a pointwise fashion. This restricts the
use of the inverse scattering transform to localized, rather than
$L^2$ data. On the other hand, for $z$ in the open upper half space it
suffices to have some $L^2$ type bound on $u$ in order to define
$T(z)$.  We greatly exploit this property, as our fractional conserved
energies are all defined in terms of only the values of $T$ away from
the real axis.

The situation is more complicated in the focusing case, due to the
fact that $T$ is no longer holomorphic in the upper half space, which
is closely related to the fact that the corresponding Lax operator
$\mathcal L$ is no longer self-adjoint. 
Now the quantity $|\psi_1|^2 +|\psi_2|^2$ is conserved, which shows that
\[ 
|T|^2 - |R|^2 = 1 
\]
on the real line. The function $T^{-1}$ still has a holomorphic extension to the 
upper half-plane, so as above we get the relation $|T| \geq 1$ in the upper half-plane.
But now  $T$ may have poles in the upper half-space, which 
 correspond to nonreal eigenvalues of $\mathcal L$.  Arguing either by the 
holomorphy of $T^{-1}$ or by standard Fredholm theory, the poles of $T$ 
must be isolated in the open upper half space, though they can accumulate 
on the real line. At this point the understanding of the spectrum of the
AKNS operator is limited.  Zhou \cite{MR1008796} has
constructed an example of a Schwartz potential with infinitely many
eigenvalues, showing that the situation is much more complex than for Schr\"odinger operators.

For a datum $u$ for which $T$ is holomorphic in the upper half-space, the
analysis is similar to the defocusing case. If instead $T$ is merely
meromorphic, then the scattering data involves not only the function
$R$ on the real line, but also at least the singular part of the
Laurent series of $T$ at the poles.  However, this still does not fully
describe the problem, as by the results of Zhou \cite{MR1008796},  $T$ may have poles in
the upper half space accumulating at the real axis even for Schwartz
functions $u$.

There is, however, one redeeming feature: All such poles are localized
in a strip near the real axis if $u \in L^2$, and more generally in a
polynomial neighbourhood of the real line
\[
\{z:0 \leq \im z \lesssim_{\|u\|_{H^s}}
(1+|\real z|)^{-2s}\}
\]
  if $u \in H^s$ with $-1/2 < s < 0$.
In the limiting case $s = -1/2$, smallness of $u$ in $l^2 DU^2$
guarantees the localization of the poles in   
\[
\{0 \leq \im z \ll
(1+|\real z|)
\}.
\]
  This is what allows us to still construct the fractional Sobolev
conservation laws even in the focusing case.

\subsection{The transmission coefficient in the upper 
half-plane and conservation laws} 

Our construction of fractional Sobolev conserved quantities relies essentially on the 
fact that the transmission coefficient $T$ is preserved along both the NLS and mKdV 
flows. In principle this gives us immediate access to infinitely many conservation
laws, but the question is whether one can relate (some of) them nicely to the standard scale 
of Sobolev spaces.  In particular, for any function  $\eta : \R\to \R$  the expression
\[
- \frac1\pi \int_{\R}  \eta(\xi) \ln |T(-\xi/2)|\,   d\xi 
\]
is formally conserved. It is useful to consider at first special 
functions $\eta$.   In the defocusing case, for instance, trace
formulas (which basically involve a change of  contour of integration) show that the (real) conserved quantities, which we may  take  as definition of $H_k$ in the defocusing case
\begin{equation}  
H_k=-\frac{(-1)^k}\pi  \int_{\R}  \xi^k  \ln |T(-\xi/2)| \, d\xi = -\frac1\pi \int_{\R} \xi^k \real \ln  T(\xi/2)\, d\xi \label{energy}   
\end{equation} 
can be explicitly expressed in terms of $u$ and its derivatives if $k$ is an nonnegative integer. 

More precisely, if $u$ is a Schwartz function then $\ln |T|$ is a Schwartz
function on the real line, and has a  Taylor expansion at $\infty$ 
\begin{equation}\label{expand} 
\ln T(z) \approx  -i   \sum_{j=0}^{\infty}  H_{j}(2z)^{-j-1}
\end{equation}
and the coefficients can be recursively expressed in terms of $u$ and its derivatives, see \cite{MR905674}.  This provides easy access to the
conserved energies \eqref{eq:energies}. \red{We refer to \cite{MR905674} for this expansion. It will also be a byproduct of our constructions below. The $H_k$ are the Hamiltonians of the KdV hierarchy. The leading order term of $H_{2k}$ is $\int |\partial_x^k u|^2 dx$.} 

  In the focusing case the expansion \eqref{expand} is used with a sign twist 
\begin{equation}\label{expand-f} 
\ln T(z) \approx  i   \sum_{j=0}^{\infty}  H_{j}(2z)^{-j-1}
\end{equation}
to define the conserved energies $H_{j}$.   However, $\ln T$ may have poles in the upper half plane, and the right hand side  in the
formula \eqref{energy} above has to be modified to account for the residues at the
poles. Precisely, if the poles of $T$ in the upper half plane are located at $z_j$ with multiplicities $m_j$
then the counterpart of the relation \eqref{energy} is
\begin{equation}  
H_k=  \frac1\pi \int_{\R}  \xi^k \real  \ln T(\xi/2)d\xi +   \sum_j  \frac{1}{k+1} m_j \im (2z_j)^{k+1}.  
\label{energy-f}   
\end{equation} 
This is clear if $T$ has finitely many poles away from the real line, but can also be justified in general by 
interpreting the trace of $-\ln|T|$ on the real line as a nonnegative measure.
Given the above discussion, a natural candidate for a fractional Sobolev conservation law 
would be obtained by choosing any (real) function $\eta$ so that 
\[
\eta(\xi) \approx (1 + \xi^2)^s
\]
However, there are two issues with such a general choice. First,
it is quite difficult to get precise estimates for $\log |T|$ on the
real line without assuming any integrability condition on
$u$. Secondly, in the focusing case such a choice would still miss the
poles of the transmission coefficient.

To remedy both of these issues, it is natural to use much more precise real weights which 
have a holomorphic extension at least in a strip around the real line. Our choice will be 
to use the weights 
\[
\eta_s (\xi) = (1 + \xi^2)^s, \qquad s > -\frac12
\]
which we extend as holomorphic functions to the subdomain $D = U \setminus i[1,\infty)$ 
of the upper half-space $U$.  Thus in the defocusing case we take
\begin{equation}\label{def}
E_s (u) =  -  \frac1{\pi}  \int_\R   (1 + \xi^2)^s \ln|T(\xi/2)|   d\xi, 
\end{equation} 
where for simplicity we changed the sign in the argument of $T$, and 
rewrite it as the real part of a contour integral
\[
E_s (u) =  -\frac1\pi \real \int_\R   (1 + \xi^2)^s \ln T(\xi/2) \,   d\xi 
\] 
and then switch the integration contour to the double half-line $\gamma$
$ [i,i\infty)$. The curve $\gamma$ goes from $i \infty -0$ to $ i-0$, then to $i+0$ and to $i \infty +0$. We extend $(1+\xi^2)^s$ as $(1+z^2)^s$ to the upper half plane minus the half line from $i$ to $ i \infty$. For holomorphic functions $g$  with enough decay we obtain 
\begin{equation} \label{contour} \int_{\gamma} (1+z^2)^s g(z) dz = 2\sin(\pi s) \int_1^\infty  (\tau^2-1)^s  g(i \tau ) d\tau,  \end{equation} 
a formula which we will use repeatedly.  For small $s$ this gives directly 
\begin{equation}\label{def+}
 E_s(u) =  \frac{2\sin(\pi s)}{\pi}   \bigintsss_1^\infty (\tau^2-1)^s   
\real  \ln T(i\tau/2)   dt, \qquad -\frac12 < s < 0.
\end{equation}  
For larger $s$ the above integral will diverge, and to remedy that we
need to remove the appropriate number of terms in the asymptotic
expansion of $\ln T$ at infinity. Precisely, in the range
\[
 s < N+1,  \qquad  u \in H^{\max\{N,s\}}   
\]
we obtain
\begin{equation}\label{def++} 
 E_s(u) =  \frac{2\sin(\pi s)}\pi   \bigintsss_1^\infty (\tau^2-1)^s   \left(\real  \ln T(i\tau/2)
+   \sum_{j=0}^N (-1)^j H_{2j}  \  \tau^{-2j-1} \right) d\tau  + \sum_{j=0}^N \binom{s}{j}  H_{2j}, 
\end{equation}     
where all terms turn out to be well-defined, and which is independent
of the choice of $N$ since for $j> s$ we can undo the change of the
contour in the integral (compare with \eqref{contour} with $g(z)=z^{-2j-1}$ and $ig(it) = (-1)^j t^{-2j-1} $ :
\begin{equation} \label{changeofcontour}  
-\frac{2\sin(\pi s)}{\pi}  (-1)^j   \int_1^\infty (t^2-1)^s t^{-2j-1} dt = 
  \frac{i} \pi   \int_{-\infty+i0}^{\infty+i0}  (1+z^2)^s z^{-2j-1} dz =   \binom{s}{j}. 
\end{equation}                              
For the last equality  we used the Taylor expansion 
\[ (1+ z^2)^s z^{-2j-1} =  \sum_{k=0}^\infty \binom{s}{k} z^{2k-2j-1}. \] 
For definiteness we will almost always choose $  N \le s < N+1$. 
Thus we set

\begin{definition}[Energies in the defocusing case]\label{def-en}
Let $s > -\frac12$  and $u \in H^s$. Then

a) For $-\frac12 < s < 0$, the energy $E_s$ is defined by \eqref{def+}.

b) For $N \geq 0$ and $  N \le s < N+1$, the  energy $E_s$ is defined by \eqref{def++}.
\end{definition}

If $u$ is Schwartz  then the two definitions \eqref{def}, respectively \eqref{def+}, \eqref{def++} are
equivalent. However, the expressions \eqref{def+}, \eqref{def++} are
much more robust, and, as we shall see, are defined directly as 
convergent integrals for $u \in H^s$. 

Consider now the focusing case. The above discussion still applies provided that there are 
no poles for $T$ in the upper half-plane. However, if there are poles then 
the definitions \eqref{def}, respectively \eqref{def+}, \eqref{def++} are no longer equivalent.
Instead we will use  \eqref{def+}, \eqref{def++} directly to define the energies:

\begin{definition}[Energies in the focusing case] \label{def-en-f}
Let $s > -\frac12$  and $u \in H^s$. Then

a) For $-\frac12 < s < 0$, the energy $E_s$ is defined by \eqref{def+}, but with $\ln T$ replaced by $-\ln T$. 

b) For $N \geq 0$ and $  N \le s < N+1$, the  energy $E_s$ is defined by \eqref{def++},
but with $\ln T$ replaced by $-\ln T$. 
\end{definition}

Here we allow for $T$ to have (finitely many) poles on the half-line $i[1/2,\infty)$,  
We also note the  role played by the smallness condition  for $u$ in $l^2 DU^2$,
which is present in Theorem~\ref{energies}. This guarantees that $T$ has a convergent 
multilinear expansion on the  half-line $i[1/2,\infty)$, and in particular  has no poles 
there.

We now summarize the above discussion concerning the relation between 
\eqref{def}, respectively \eqref{def+}, \eqref{def++}.
In the upper half-space we define the function
\begin{equation}\label{Xi-def}
 \Xi_s(z) = \im \int_0^z (1+\zeta^2)^{s} d\zeta, 
\end{equation}
which does not depend on the path of integration. Then we have the following relations:
 
\begin{proposition}[Trace formulas]
\label{prop:trace}  
Let  $N > [s] $ and $ u \in \mathcal{S}$. In the defocusing case  
\begin{equation}\label{defd+} 
\begin{split}  
E_s = & \ \frac1\pi \int_{-\infty}^\infty (1+\xi^2)^{s} (-\real \ln T(\xi/2)) d\xi \\
    = &\  \frac{2\sin(\pi s)}\pi  \int_1^\infty(\tau^2-1)^s \Big[ \real  \ln T(i\tau/2)  + 
\sum_{j=0}^N (-1)^j  H_{2j}  \tau^{-2j-1}\Big] d\tau 
+ \sum_{j=0}^N \binom{s}{j}  H_{2j} 
\end{split} 
\end{equation} 
and in the focusing case 
\begin{equation}\label{deff+} 
\begin{split}  
E_s = & \ \frac1\pi \int_{-\infty}^\infty (1+\xi^2)^{s} \real \ln T(\xi/2) d\xi+   2 \sum_k m_k \red{\Xi_s}(2z_k)  \\
    = & -\frac{2\sin(\pi s)}\pi  \int_1^\infty (\tau^2 -1)^s\Big[\real  \ln T(i\tau/2) - \sum_{j=0}^N (-1)^j  H_{2j} \tau^{-2j-1}\Big] d\tau 
+ \sum_{j=0}^N \binom{s}{j}  H_{2j} 
\end{split} 
\end{equation} 
where the $k$ sum runs  over all the poles $z_k$ of $T$ with multiplicity $m_j$. \end{proposition} 

\red{The function  $\ln T(z)$ is a meromorphic function on the  upper half plane and the previous pages provide a heuristic proof.  In Section \ref{s:t2} we provide a rigorous proof considering $\ln |T(z)|$ as a (super) harmonic function on the upper half plane using also the bounds proven in Sections \ref{s:t2j} and \ref{s:tt2j-exp}. }

For the focusing case we remark on the case when there are infinitely
many poles for $T$ in the upper half-space.  The second expression
above is always a convergent integral, whereas in the first expression
we have a nonnegative integral, plus a sum where all but finitely many
terms are positive. This simultaneously allows us to interpret the
trace of $\ln |T|$ on the real line as a nonnegative measure, and to guarantee
the convergence in the $k$ summation.
We also remark on the contribution of the poles which are on the imaginary axis.
Precisely the function $\Xi_s$  is real analytic away from $z=i$. Thus the 
only nonsmooth dependence on $u$ in $E_s$ via the poles comes from
the poles which are  at  $i/2$. 
Now we are ready to state a more complete version of our main result in Theorem~\ref{energies}:

\begin{theorem}\label{t:main+} 
  For each $s > -\frac12$ and both for the focusing 
  and defocusing case the energy functionals $E_s$  are globally defined
\[
E_s :  H^s    \to \R
\]
with the following properties:
\begin{enumerate}
\item $E_s$ is conserved along the NLS and mKdV flow.

\item For all $u \in H^s$ the trace at the real line  of\footnote{The choice of signs
    $\mp$ corresponds to the defocusing/focusing case} $\mp \log |T|$
  exists as a positive measure, and the trace formulas \eqref{defd+}
  and \eqref{deff+} hold with absolute convergence in all sums and
  integrals.

\item  If $\Vert u \Vert_{l^2_1DU^2}\le 1 $ then   
\begin{equation} \label{est:quad} 
\left| E_s(u) - \|u\|_{H^s}^2 \right| \lesssim  \Vert u \Vert_{l^2_1DU^2}^2 
\Vert u \Vert_{H^s}^2. 
\end{equation} 
\item The  map 
\[ 
H^{\sigma} \times  (-\frac12,\sigma] \ni (u,s) \to  E_s(u)  
\]
is  analytic in $u\in H^\sigma$ in the defocusing case. In the focusing case 
it is continuous and it is analytic provided $\frac{i}2$ is not an eigenvalue. It is also continuous in $s$, and analytic in $s$ for $s < \sigma$.
\end{enumerate}
\end{theorem}
The proof of the above theorem is  completed in Section~\ref{s:proof}.
We remark now on several direct consequences of our trace formulas:
\begin{itemize}
\item The energies $E_s$ are nonnegative in the defocusing case, 
and in the focusing case if $ s\le 0$. 
\item In the focusing case the eigenvalues are localized in a strip
\[ 
(1+|\real z_j^2|)^{s}\im z_j \lesssim_{\|u\|_{H^s}} 1   
\]
Furthermore, the following sum is bounded:
\begin{equation} \label{e:sumbound} 
\sum  m_j \im z_j    (1+|\real z_j|^2)^s  \lesssim_{\|u\|_{H^s}} 1
\end{equation} 
This is an immediate consequence of the trace formulas if $ s\le 0$
since then $\im z^s >0$ in the upper half plane. If $s\ge 0$ the trace
formula for $s=0$ implies that the imaginary part of the eigenvalues
is bounded. Hence there are at most finitely many eigenvalues not
giving a positive contribution to the sum if $\im z_j (1+|\real z_j|^2)^s $.
 Their contribution is bounded by the $L^2$ norm and we
arrive at the bound \eqref{e:sumbound}.
\end{itemize}

One can view the small data part of our result as a stability
statement in $H^s$ for the zero solution of the NLS or mKdV
equations. In the focusing case another interesting class of solutions
are the pure soliton/breather solutions, where $T$ has a finite number
of poles in the upper half-space and $|T| = 1$ on the real line.  As a
direct byproduct of our results, one can also obtain  the $H^s$
stability of these solutions. We consider this in detail in a subsequent paper.

\subsection{Estimates for the transmission coefficient}

Given the definition of the conserved energies in the previous
subsection, the proof of our main result hinges on obtaining precise
estimates for the transmission coefficient. The steps in our analysis are as follows:
\medskip

I. Our analysis begins in Section~\ref{formal}, where we compute a
formal homogeneous expansion of inverse of the transmission coefficient $T$,
namely
\[
T^{-1}(z) = 1 + \sum_{j=1}^\infty T_{2j}(z)
\]
where $T_{2j}(z)$ are multilinear integral forms, homogeneous of degree
$2j$ in $u,\bar u$. There is a similar though less explicit expansion for
$\ln T$,
\[
-\ln T(z) = \sum_{j=1}^\infty \tilde T_{2j}(z)
\]
We have $\tilde T_2 = T_{2}$, while $\tilde T_{2j}$ are still multilinear integral forms of degree  $2j$ in $u,\bar u$.

Nevertheless, it turns out that the $\tilde T_{2j}$ are much more localized 
than $T_{2j}$, which is the key to better decay and better algebraic properties. This observation is an interesting result in itself. Its proof is purely
algebraic, and is based on a Hopf algebra structure underlying the
multilinear integral calculus. This is explained in detail in
Appendix~\ref{a:hopf}. The relevant structure of the integrals $\tilde T_{2j}$ 
is encoded in algebraic statement of Theorem \ref{lnT}.
 \bigskip

II. The first term $T_2(z)$ in the formal expansions is studied in
Section~\ref{s:t2}. This is a quadratic form in $u$, and its
contribution to the conserved norms is exactly $\|u\|_{H^s}^2$.
To prove this, we estimate $T_2(z)$ in the full upper half-space.
\bigskip

III. The remaining terms $T_{2j}$ in the expansion of $T$ are only
needed on the imaginary half-line $[i,i\infty)$. They are estimated in
Section~\ref{s:t2j} in terms of the $H^s$ norm of $u$ and the $DU^2$
norm of $u$. Here the $DU^2$ norm arises as a convenient proxy for the
scale invariant $\dot H^{-\frac12}$ norm of $u$. Appendix  \ref{a:uv} summarizes 
and proves some useful and some new properties of the $U^2$ and $V^2$
spaces. The bounds for $T_{2j}$ suffice in order to establish the convergence
of the formal series for $T$, and also they easily carry over to
$\tilde T_{2j}$. In particular they suffice in order to bound the
corresponding contributions to the energy for large enough $j$ (namely
$j > 2s+1$). Thus the proof of the theorem reduces to estimating
finitely many $\tilde T_{2j}$ terms.

\bigskip

IV. The main bound for $\tilde T_{2j}$ is proved in
Section~\ref{s:tt2j-exp}, taking advantage of the better decay properties
due to the better structure of $\tilde T_{2j}$. This suffices for $j > s+1$.
For smaller $j$ (which  only occurs  if  $ s \ge  1$) we also need to consider the
expansion of $\tilde T_{2j}(z)$ in powers of $z^{-1}$ as $z \to
i\infty$, and estimate the errors.  This is done in Section
\ref{s:tt2j-exp} and used in Section \ref{s:proof} to complete the proof of our main result. 

\bigskip 

V. The Hopf algebra structure allows to explicitly calculate  $\tilde T_{2j}$( Proposition \ref{a:log}). We do these calculations for $\tilde T_4$ and $\tilde T_6$ in Section \ref{s:further}.  

\bigskip 

VI. There are close connections to the KdV
equation. The first relies on the observation that, if $u$ is a real
function and if $\psi$ satisfies the first equation in
\eqref{laxpairnls} then $\phi= \psi_1+\psi_2$ satisfies with $v=
u_x+u^2$
\[  - \phi_{xx} + v \phi = z^2 \phi \] 
which is the eigenvalue  equation of the standard Lax operator for the
KdV equation. As a consequence the transmission
coefficient for the Lax operator for mKdV with real potential $u$
coincides with the transmission coefficient for the Lax operator for
$v=u_x+u^2$ for KdV.

\red{We rewrite the Lax pair as a system and }  use a Lax-pair similar to the one for the NLS, which allows
to reuse most of the constructions developed for the NLS case. \red{ The Lax equation is
 \[ \psi_x = \left( \begin{matrix} -iz &  1 \\ u & iz \end{matrix}\right)\psi. \] }

One difficulty  in this case
is that
\[ \tilde T_2(z)  =   \frac{1}{2z}  \int u  dx  \]  
which can never be bounded by an $H^s$ norm of $u$. However we are able to
prove that 
\[ \tilde T(i\tau )  -\tilde T_2(i\tau) \]   
is well-defined on a small ball in $H^{-1}$, and that  $\tilde T_{4}(z) $ 
plays the same role as  $T_2(z)$ for NLS.  A similar though technically different and interesting 
removal procedure for  the linear term is implemented in \cite{KVZ}, which in contrast to our approach also applies to the periodic KdV equation.

\newsection{The transmission coefficient in the upper
  half-plane}\label{formal} As seen above, the transmission
coefficient in the upper half plane plays a key role in our
construction of the conserved energies.  Here we consider the
multilinear expansion of the transmission coefficient in the upper
half-plane.  This section is only concerned with the algebraic aspects
of this expansion, setting aside for now the questions of bounds and
convergence of the formal series.

For the defocusing problem, $-\ln T$ extends to a holomorphic function
on the upper half plane with nonnegative  real part. For the focusing
problem there may be simple poles corresponding to eigenvalues. To cover both cases, we consider a more general system
of the form
\begin{equation}\label{scatter-re}
\left\{ 
\begin{array}{l}
\dfrac{d\psi_1}{dx} = -i z \psi_1 + u \psi_2 
\\[5mm]
\dfrac{d\psi_2}{dx} =  i z \psi_2 +  \overline v \psi_1
\end{array}
\right.
\end{equation}
and construct the Jost solutions recursively in the upper half
plane. The scattering problems for the defocusing, respectively the
focusing case are obtained by taking $v = \pm u$.  This plays
no role in any of the estimates in the following sections, but for
definiteness we consider the defocusing case.

\begin{lemma} \label{T2exp} 
  There is a formal homogeneous expansion of $T^{-1}$ in terms of  $u,v$
\[  
T^{-1} (z) =1+ \sum_{j=1}^\infty T_{2j}(z)  
\]
with 
\begin{equation} \label{T2j}  
T_{2j}(z) = \int\limits_{x_1<y_1 <x_2< y_2< \dots <  x_j<y_{j}} 
 \prod_{l=1}^j  e^{2iz(y_{l}-x_l)}u(y_{l}) \overline{v(x_l)} dx_1dy_1 \dots d x_{j}dy_j. 
\end{equation}
\end{lemma} 
We remark that, at least as long as $u,v \in L^2$,  each term $T_{2j}$ is pointwise defined
for $\im z >0$.

\begin{proof} 
We solve \eqref{scatter-re} iteratively. 
We first compute the quadratic approximation to \red{$T^{-1}(z)-1$}. Suppose we use $\psi_l$ and begin the iteration with $\psi_{l}^0 = \left(\begin{matrix} e^{-izx} \\ 0 \end{matrix} \right) $. 
The first iteration gives 
\[
\psi_{l,2}^1(x) = - \int_{-\infty}^x  e^{-izx_1} \overline{v(x_1)}  e^{iz(x-x_1)} dx_1   
\]
which inserted into the first equation yields
\[
\begin{split}
\psi_{l,1}^2 (x) =  & \ \red{e^{-izx}} + \int_{-\infty}^x e^{-i z(x-y_1)} u(y_1) 
\int_{-\infty}^{y_1} e^{-izx_1} \overline{v(x_1)}  e^{iz(y_1-x_1)} dx_1 dy_1   
\\
= & \ e^{-izx} \left(1+   \int_{x_1 < y_1 < x}  u(y_1) \overline{v(x_1)} e^{2iz(y_1-x_1)}     dx_1 dy_1                 \right). 
\end{split}
\]
This shows that the quadratic approximation to $T^{-1}-1$ is given by 
\begin{equation} \label{T2} 
T_2(z) =  \int_{x_1<y_1 }  u(y_1) \overline{v(x_1)}  e^{2iz(y_1-x_1)}     dx_1 dy_1 
\end{equation} 
We iterate the procedure and arrive at \eqref{T2j}. 
\end{proof} 

For our modified energies we need to work not with $T$ directly, but rather 
$\ln T$. Since the log is analytic near $1$, the formal series for $T$ will 
yield a formal series for $\ln T$:

\begin{lemma} 
  There is a formal homogeneous expansion of $\ln T$ in terms of  $u,v$
\[  
-\ln T = \sum_{j=1}^\infty \tilde T_{2j}  
\] 
where each term $  \tilde T_{2j}(z) $ is a linear combination of expressions 
of the form 
\begin{equation} \label{tT2j}  
 \int\limits_{\Sigma} 
 \prod_{l=1}^j  e^{2iz(y_{l}-x_l)}u(y_{l}) \overline{v(x_{l})} dx_1dy_1\dots d x_{j}  dy_j. 
\end{equation}
where $\Sigma$ can be any domain which can be represented 
as a linear ordering of $x_l$ and $y_l$ which obeys the constraint $x_l < y_l$ 
for all $l$.
\end{lemma} 

The proof of the lemma is straightforward, we simply observe that 
each term $\tilde T_{2j}$ is a homogeneous polynomial of appropriate homogeneity 
in $T_{2l}$ with $l \leq j$, and that such products can be expressed 
as iterated integrals as above.  The  first terms have the form
\[ 
\tilde T_2=    T_2, \qquad \tilde T_4 = T_4(z) - \frac12 T_2(z)^2 ,
\qquad \tilde T_6 = T_6(z) - T_2(z) T_4(z) +\frac13 T_3(z)^3   
\] 

For the purpose of studying the convergence of this series, such
considerations are sufficient. However, if we want at the same time to
capture also the expansion of $\ln T$ in powers of $z^{-1}$ as $z \to
i\infty$, then this is no longer enough. The reason is quite simple,
namely that, even if $u$ and $v$ are Schwartz functions, the terms
$T_{2n}$ in the expansion of $T$ only have $z^{-n}$ decay, whereas for
the expansion of $\ln  T$ the corresponding term has as leading term a
constant times $z^{-2n+1}$, at least near the imaginary axis for
Sobolev data.  
Thus our goal is now twofold. On one hand we want to
understand the decay rates at infinity for each of these integrals,
and on the other hand we want to see which of them are present in the
expansion of $\ln T$.

It is useful to have a more graphical representation for these integrals.
\red{A short reflection shows that we can represent these integrals by words in two letters $X$ and $Y$, since the indices are irrelevant. Moreover only words with an equal number of $X$'s and $Y$'s occur, and, if we decompose a word at any point into a left and a right part, there are at least as many $X$'s on the left as $Y$'s. Then there is exactly one way to connect each $X$  with a $Y$ by nonintersecting arcs. Accordingly } 
we identify 
the integrals with symbols like \<XXYXYY>. For example
\[ 
\<XXYXYY> = \int_{x_1<x_2<y_1<x_3<y_2<y_3}  e^{ 2iz (y_1+y_2+y_3-x_1-x_2-x_3)}   \prod u(y_j) \overline{v(x_j)} dx_j dy_j. 
\]

It is not difficult to see that the decay rate as $z \to i \infty$ of
these integrals depends on how connected they are. Precisely, if $z =
i\tau$ then, due to the exponential factor, the bulk of the integral
comes from the region where $|x_j-y_j| \lesssim \tau^{-1}$, and the
integral can essentially be estimated by the volume of this region.
If  such an integral \red{over $2n$ varaiables} has $k$ connected components then this volume 
is comparable to $\tau^{-2n+k}$. Thus the desired  decay rate 
$z^{-2n+1}$ is also the best possible decay rate, and is only achieved
for integrals with fully connected symbols. \red{See Appendix \ref{a:hopf} for more details and precise definitions}. 

In the sequel we  call an integral \eqref{tT2j}
connected if for every nontrivial decomposition of the sequence there are more 
$x$ to the left than $y$, or, equivalently, if in the graphical representation 
there is an arc from the first $x$'s to the last $y$'. We denote it by symbols like \<XXYY>.  We arrive at one of the main results of this section,
namely Theorem \ref{lnT} which we repeat here.

\begin{theorem}\label{primitive}
  The terms $\tilde T_{2j}$ of the homogeneous expansion of $\ln T$ in
  $u,v$ are formal linear combinations of connected  integrals.
\end{theorem}

The first terms in the defocusing case are (see Proposition \ref{a:log} )  
\[ -\ln T = \<XY> - 2\<XXYY> + 12 \<XXXYYY>+4 \<XXYXYY> + \dots \]  

There is even a polyhomogeneous formal expansion, with respect to
powers of $z$ and $u$ resp. $\bar u $ which we will explore below.
The proof of this theorem is nontrivial, and requires the
introduction of an additional Hopf algebra type structure. For this
reason, we relegate this part of the proof to Section \ref{a:hopf} of
the appendix.
The classical  conserved quantities are the coefficients of the
expansion of $ \ln T$ in powers of $z^{-1}$ at infinity, which are
called energies.  The above theorem has the consequence that these
energies are given as one dimensional integrals, and not as
polynomials of one dimensional integrals.  This is of course
well-known \red{(see Faddeev and Takhtajan \cite{MR905674})} and typically proven by very different techniques.

\newsection{Positive harmonic functions in the the upper half-space}
\label{s:t2}

A significant step in our analysis is to transfer information about
the transmission coefficient from the positive imaginary axis to the
real line. Given enough a-priori information on $\log |T|$, this is
simply a matter of applying the appropriate form of the  residue theorem. 
However, in our setting such a-priori information is not freely available,
and instead we want to obtain it as a conclusion of our results (namely 
in our trace formulas).  This section is devoted to considering such issues
in an uniform fashion.
We will phrase our results in terms of a nonnegative superharmonic function $G$
in the upper half-space. One should think of $G$ as $\mp \log |T|$, which is harmonic 
in the defocusing case but may have singularities at the poles $z_k$ of $T$ in the focusing case,
\[
-\Delta \log |T| = 2\pi \sum m_k \delta_{z_k} \geq 0
\] 
where $m_k$ denote the corresponding multiplicities. We begin with a first result which provides an integral representation for such functions. 

\begin{lemma}
a) Let $G$ be a nonnegative harmonic function in the upper half-space, which is bounded on the 
positive imaginary axis.  Then it has a trace on the real line, which is a locally finite nonnegative measure $\mu$.
Furthermore, we can represent $G$ via the   Poisson kernel as 
\begin{equation}\label{G-rep}
G(z) =  \frac{1}{\pi} \int_\R \frac{\im z}{|z-\xi|^2} d \mu(\xi)
\end{equation}
b) Similarly, suppose $G$ is a nonnegative superharmonic function in the
upper half-space, which is bounded on the positive imaginary
axis. Then it has a trace on the real line, which is a Radon  measure $\mu$, and $\nu = -\Delta G$ is a Radon  measure in the upper half-plane so that 
\begin{equation}\label{G-rep+}
G(z) =  \frac{1}{\pi} \int_\R \frac{\im z }{|z-\xi|^2} d \mu(\xi) + \frac{1}{2\pi} 
\int_{H} \log \left|\frac{z-\bar z_0}{z-z_0} \right| d\nu(z_0)
\end{equation}
\end{lemma}
\begin{proof}
a) The trace\red{s} of nonnegative harmonic functions are measures, see 
\cite{MR1805196}.   For large $R$ denote
\[
G_R(z) =  \frac{1}{\pi} \int_{-R}^R \frac{\im z }{|z-\xi|^2} d \mu
\]
and let $G_\infty$ be the expression on the right in \eqref{G-rep}.
Then $G_R$ is a harmonic function in the upper half-space which decays at infinity.
Hence $\limsup_{z\to \infty} G(z)  - G_R(z) \ge 0 $ in the upper half-space. 
 It is also nonnegative on the real axis,
therefore by the maximum principle \red{applied to the harmonic function $G-G_R$} we get $G \geq G_R$.  Letting $R \to \infty$, by Fatou's lemma we get
\[
\lim_{R \to \infty} G_R = G_\infty \leq G
\]
Now consider the difference $H = G-G_\infty$, which is nonnegative, harmonic, vanishes on the real line
and is bounded on the purely imaginary axis.  Applying the Harnack inequality  (see \cite{MR688919}) we conclude that 
$H$ is bounded in the upper half-plane. Since it vanishes on the real line, the odd extension of $H$ is a global bounded harmonic function, and thus constant.
Hence we get $H=0$.
 
b) The negative Laplacian of a superharmonic function is a measure. We denote this measure by $\nu$. For $K \subset \{ z : \im z >0\} $ compact we define 
\[ G_K = \frac1{2\pi} \int_K  \ln \Big| \frac{z-\bar z_0}{z-z_0}\Big|  d\nu(z_0). \]
By the maximum principle $G_K \le G$. Similarly,  if $K_0 \subset K_1$ then 
$G_{K_0} \le G_{K_1}$ and 
\[ G_\infty = \sup_{K} G_K \le G \]   
and it satisfies 
\[ -\Delta G_\infty = \nu. \] 
Now we apply part a) to the nonnegative harmonic function $G-G_\infty$. 
\end{proof}

Consider a superharmonic function $G$ as above and $j \geq 0$. 
Then we define the  expressions
\begin{equation} \label{formula} 
G_{2j} = \frac{1}{\pi} \int_\R \xi^{2j} d\mu + \frac{1}{2\pi}\int_{H} \frac{1}{2j+1} \im z^{2j+1}  d\nu
\end{equation} 
which can formally be  interpreted as  the coefficients of the formal expansion of $G$
at infinity,
\[
 G(i \tau) \approx \sum_{j=0}^\infty (-1)^j \tau^{-2j-1} G_{2j}
\]
We make this correspondence rigorous in the following:

\begin{lemma}\label{l:apriori-n}
Assume that the support of $\nu$ is contained in a the strip 
 $ 0 < \im z  < c$ 
 and that on the imaginary axis the function $G$ admits the finite expansion
\[
 G(i \tau) = \sum_{j=0}^{N-1}(-1)^j \tau^{-2j-1}  G_{2j} + O(\tau^{-2N-1})
\]
for some real constants $ G_{2j}$. Then the measures $
(1+\xi^2)^{N}\mu$, $\im z (1+ |z|^2)^N \nu$ are finite. Vice versa, 
if the measures $(1+\xi^2)^{N}\mu$ and $\im z (1+ |z|^2)^N \nu$ are finite and $G_{2j}$, $j\le N$ are defined by \eqref{formula} then  we have \begin{equation}\label{improved}  
G(i \tau) = \sum_{j=0}^N (-1)^j \tau^{-2j-1} G_{2j} + o(\tau^{-2N-1}). 
\end{equation} 
\end{lemma}
\begin{proof}
We argue by induction on $N$. For $N=0$ we use the representation \eqref{G-rep+} to compute 
\[
\lim_{\tau  \to \infty} \tau G(i \tau) =
\lim_{\tau\to \infty} \Big[   \frac{1}{\pi} \int_\R \frac{1}{1+|\xi|^2/\tau^2} d \mu(\xi) + \frac{1}{2\pi} 
\int_{H}\tau  \log \left|\frac{1+i\bar z_0/\tau}{1+iz_0/\tau} \right| d\nu(z_0)\Big]
= 
G_0
\]
where the integrand is positive and increasing in $\tau$ so the limit always exists.
For the induction step, we assume that the result holds for $N-1$ and prove it for $N$.
To achieve this we consider the difference 
\begin{equation} \label{GN} 
\begin{split} 
G_{\geq N}(i\tau) = & G(i\tau) -   \sum_{j=0}^{N-1} (-1)^j \tau^{-2j-1} G_{2j} 
 =  \frac1\pi \int_{\R}  \Big[  \frac{\tau}{\tau^2+\xi^2} -  \sum_{j=0}^{N-1} 
(-1)^j \frac{\xi^{2j}}{\tau^{2j+1}} \Big] d\mu \\
& +  \frac1{2\pi} \int_H  \Big[  \log \frac{|i\tau-\bar z|}{|i\tau-z|} - \sum_{j=0}^{N-1}(-1)^j  \frac1{2j+1} \frac{\im z^{2j+1}}{\tau^{2j+1}}  \Big] d\nu(z)  
\\ = & 
 \frac{1}{\pi} \int_\R \frac{(-1)^N \xi^{2N}}{\tau^{2N-1} (\tau^2+\xi^2)} d\mu(\xi)  
+ \frac{1}{2\pi} \int_H \Xi_{\geq N}(\tau, z) d\nu(z) 
\end{split} 
\end{equation}  
where, for $i\tau \ne z$, using  $|it -\bar z|= |i\tau +z|$, 
\begin{equation} \label{XiN}
\begin{split} 
\Xi_{\geq N}(\tau, z) = & \, \real \Big[ \log (i\tau +z)-\log (i\tau -z) +i \sum_{j=0}^{N-1} (-1)^j \frac1{2j+1} \frac{ z^{2j+1}}{\tau^{2j+1}} \Big] 
\\ = & \, \real  \int_0^z  \frac{d}{d\zeta}  \Big[ \log (i\tau+\zeta) -\log(i\tau-\zeta)  +i \sum_{j=0}^{N-1} (-1)^j \frac1{2j+1} \frac{ \zeta^{2j+1}}{\tau^{2j+1}}\Big]d\zeta  
\\ = &\, \real  \int_0^z    \Big[\frac12\big(  \frac1{\zeta+i\tau}-\frac1{\zeta-i\tau}\big)    +i\sum_{j=0}^{N-1} (-1)^j  \frac{ \zeta^{2j}}{\tau^{2j+1}}\Big]d\zeta  
\\ = &\, \im  \int_0^z    \frac{\tau} {\tau^2+\zeta^2}     -\sum_{j=0}^{N-1} (-1)^j  \frac{ \zeta^{2j}}{\tau^{2j+1}}d\zeta  
\\ = &\,
(-1)^N\im  \int_0^z   \frac{\zeta^{2N}}{\tau^{2N-1} (\tau^2+\zeta^2)} d\zeta.  
\end{split}
\end{equation} 
Then $\Xi_{\ge N }$ 
is independent of the contour of integration. 
The integrand is nonnegative outside a compact set, so using again monotonicity we compute the limit 
\[
\lim_{\tau \to \infty} \tau^{2N+1} (-1)^N G_{\geq N}(i\tau) = G_{2N}
\]
which may be either finite or $+\infty$. By the induction hypothesis it  is finite and the conclusion of the lemma follows.
\end{proof}

Next we turn our attention to noninteger $s$, for which we define the energy type quantities as in \eqref{def++}\red{. For $G$ as above and $ N\le s < N+1$ we define} 
\[
E_s(G) =  \red{\frac{2 \sin (\pi s)}\pi}   \int_1^\infty (\tau^2 -1)^s\Big[-  G(i\tau)  +
\sum_{j=0}^N (-1)^j  G_{2j} \tau^{-2j-1}\Big] d\tau 
+  \sum_{j=0}^N \binom{s}{j}  G_{2j} 
\]
and we recall that the integrand is $-G_{\ge N}$ of \eqref{GN}   and the correction is chosen so $E_s(G)$ does not change when we increase $N$ by \eqref{changeofcontour}. 
For these we have the following:

\begin{lemma}\label{l:apriori-s}
Let $G$ be a superharmonic function in the upper half-space, with the associated measures $\mu$, $\nu$ 
as above, and $N \leq s < N+1$ with the following properties:

\begin{enumerate}
\item   The quantities  $G_{2j}$, $0\le j \le  N$ are finite.
\item  The support of $\nu$ is contained in the region $\im z \lesssim  (1+\real z)^{-2s}$ for $s \leq 0$, respectively 
$\im z \lesssim  1$ for $s \geq 0$.
\end{enumerate}

Then the following trace formula holds with $\Xi_s$ defined in  \eqref{Xi-def}:
\begin{equation}\label{trace-G}
E_s(G) =\red{\frac1\pi} \int_{\R} (1+\xi^2)^s d\mu + \red{\frac{1}{\pi}} \int_H \Xi_s d\nu  
\end{equation}
in the sense that the equality holds whenever the one side is finite. If one of the sides is finite then the integrand of $E_s(G)$ is bounded by $C \tau^{-2s-1} $. 
\end{lemma}

\begin{proof}
We use the representation of $G$ in \eqref{G-rep+} and expansions of the integrands 
to rewrite   $E_s(G)$ by introducing summands to the integrand  to get better decay as in \eqref{GN}. 
We verify the claimed identity first for Dirac measures $\mu = \delta_\xi$:
\[
\frac{2 \sin (\pi s)}{\pi} \int_{1}^\infty (\tau^2-1)^s \frac{(-1)^{N+1}\xi^{2N+2}}{(\tau^2+\xi^2) \tau^{2N+1}} d\tau 
+  \sum_{j=0}^N \binom{s}{j}    \xi^{2j} = \frac{1}{\pi} (1+\xi^2)^s
\]
which is nothing but (the real part of) Cauchy's formula as  for \eqref{def+} applied to integral. We move this contour of integration to the real line, 
undo the splitting of \eqref{XiN} and evaluate the residua. 
\[  
\frac1\pi \im  \int_{-\infty+i0}^{\infty+i0}  (1+\eta^2)^s \Big[ \frac{\eta}{\xi^2-\eta^2} - \sum_{j=0}^{N} \frac{\xi^{2j}}{\eta^{2j+1}}\Big]  d\eta  
= \frac{1}{\pi} (1+\xi^2)^s - \sum_{j=0}^N \binom{s}{j} \xi^{2j}. 
\] 

The same argument applies for the case when $\nu= \delta_z$ is a Dirac mass. Then we need to verify the relation (see \eqref{XiN} for the definition of $\Xi_{\ge N+1} $ and \eqref{Xi-def} for the definition of $\Xi_s$) 

\[
\frac{2 \sin (\pi s)}{\pi} \int_{1}^\infty (\tau^2-1)^s \Xi_{\geq N+1}(z) d\tau +   \sum_{j=0}^N \binom{s}{j}\frac{ z^{2j+1}}{2j+1}  =\frac{1}{ \pi} 
\Xi_s(z)
\]
Both sides vanish at $z = 0$, so it suffices to verify that their derivatives with respect to $z$  are equal, i.e. that
\[
\frac{2 \sin (\pi s)}{\pi} \int_{1}^\infty (\tau^2-1)^s \frac{\xi^{2N+2}}{(\tau^2+z^2) \tau^{2N+1}} d\tau 
+  \sum_{j=0}^N \binom{s}{j}    z^{2j} = \frac{1}{\pi} (1+z^2)^s
\]
This is the same residue theorem computation as above, except that at
$z$ we now get a full residue rather than the prior half residue.

 We first establish the desired bound \eqref{trace-G} if both measures
 $\mu$ and $\nu$ are compactly supported. Then all integrals in
 \eqref{trace-G} are absolutely convergent, and it suffices to
 consider the case when both $\mu$ and $\nu$ are Dirac masses.
 Hence, to conclude the proof it suffices to assume that both measures are supported outside a compact set.
Then we can use the localization of $\nu$ to conclude that the contribution of $\nu$ to both the 
left hand side and the right hand side of \eqref{trace-G} are nonnegative. 
By  
Fatou's lemma  the two expressions are either both finite and equal, or both infinite.
\end{proof}

The leading term in both $T^{-1}-1$ and $\ln T$ away from the real axis is
$T_2(z)$. It is by far the simplest and most important to analyze, and provides 
the quadratic term in all our conserved energies. All the techniques for its 
study will remain relevant for the higher order terms.
In the focusing case $v = - u$ so only the sign of $T_2$ changes. 
There is however no sign change in $\tilde T_2$ 
since we also replace $-\ln |T|$ by $\ln |T|$. 

In both cases we study 
\[
 T_2(z) =  \int_{x<y} e^{2iz(y-x)} u(y) \overline{u(x)} dx dy. 
\] 
We choose an approach which generalizes to the higher  order terms. 
For definiteness we work with the unitary Fourier transform 
\[ \hat u(\xi) = (2\pi)^{-\frac12} \int u(x) e^{-ix\xi } dx, \]
allow $v\ne \pm u$   
and use Fourier inversion to calculate 
\begin{equation}  \label{T2-d} 
\begin{split} 
T_2(z) = & \frac1{2\pi}  \int_{x<y} e^{2iz(y-x) - i (x\xi-y\eta)} \hat u(\eta) \overline{\hat v(\xi)} d\xi d\eta  dx dy 
   \\  = & -\frac1{2\pi} \int_{\R^2} \int_{\R} \frac1{2iz+i\xi}    e^{iy(\eta-\xi) } \hat u(\eta)\overline{\hat v(\xi)}dy d\eta d\xi   
\\   = & i  \int_{\R}  \frac1{2z+\xi} \hat u(\xi) \overline{\hat v(\xi)}  d\xi 
\\   =  &  i   \sum_{j = 0}^N (-1)^j \int \xi^j \hat u(\xi) \overline{\hat v(\xi)}  d\xi  (2z)^{-j-1}
+ i (-2z)^{-N} \int   \frac{\xi^{N+1} }{(2z+\xi)}     \hat u(\xi) \overline{\hat v(\xi)}  d\xi. 
\end{split} 
\end{equation} 
 We focus on $v=u$ and denote $f(\xi)= |\hat u(\xi)|^2$. Then 
\[
 T_2(z) = i  \int_{\R}  \frac1{2z+\xi} f(\xi) d\xi. 
\] 
and we apply Lemma \ref{l:apriori-n} and \ref{l:apriori-s}.

\begin{proposition} \label{p:T2}
a) For $z$ in  the upper half-plane we have 
\begin{equation} 
\label{T2pointwise0} 
| \real  T_2(z)| \le   \int_{\R}   \frac{2\im z}{|2\im z|^2 +|2\real z +\xi|^2} |\hat u(\xi)|^2 d \xi. 
\end{equation}   
Further,  let
\begin{equation}\label{Hj2} 
H_{j,2} =   (-1)^j \int_{\R}  \xi^{j} |\hat u(\xi)|^2 d\xi = 
    \left\{ \begin{array}{ll} \displaystyle \int_{\R}  |u^{(k)}|^2 dx & \text{ if } j=2k  \\[2mm]
                                 \displaystyle  \im   \int_{\R}   u^{(k+1)} \overline{ u^{(k)}}\, dx & 
\text{ if } j=2k+1. \end{array} \right. 
 \end{equation} 
then for $N\ge 0$ 
\begin{equation} 
\label{T2pointwise} 
 \Big|\real\Big( T_2(z/2) -i     \sum_{j=0}^{N-1} H_{j,2} z^{-1-j} \Big) \Big| 
\le  |z|^{-N} \int_{\R} |\xi|^{N} \frac{\im z}{|\im z|^2 + |\real z+\xi|^2} |\hat u(\xi)|^2 d \xi 
\end{equation}                  
                    
b) 
 If $s\notin \N$, $s>-\frac12$  then 
\begin{equation} \label{T2imaginary} 
\Vert u \Vert_{H^s}^2 =  \frac{2\sin \pi s}\pi  \bigintsss_1^\infty \red{-}  (\tau^2-1)^s     \Big[ \real T_2\Big(i\tau/2\Big) -  \sum_{j=0}^{[s]} H_{2j,2} (-1)^{j}  \tau^{-2j-1} \Big]  d\tau
+   \sum_{j=0}^{[s]} \binom{s}{j} H_{2j,2} .                                     \end{equation} 
\end{proposition}                  

\begin{proof} Only the bound \eqref{T2pointwise} and its special case 
 \eqref{T2pointwise0} remain to be proven.  They are immediate consequences of \eqref{T2-d}. Then \eqref{trace-G} applied to the right hand side of 
\eqref{T2imaginary} has the left hand side 
\[      \int  (1+\xi^2)^s   |\hat u(-\xi)|^2 d\xi.  \]  
\end{proof}

\red{The trace formulas in Proposition \ref{prop:trace} follow in the same way.  In Sections \ref{s:t2j} and \ref{s:tt2j-exp} we prove that the right hand sides of the equalities in Proposition \ref{prop:trace} are bounded for $u \in H^s$ and then the formulas follow from Lemma \ref{l:apriori-s}. }

\newsection{Bounding the iterative  integrals $T_{2j}$} 
\label{s:t2j}
Here we consider the question of estimating the integrals $T_{2j}$,
with the goal of establishing the convergence of the formal series for
$T$. Since our energies are expressed in terms of the values of $T(z)$
on the positive imaginary axis, we will focus on this case, but also
comment on the region of validity of the expansion in the upper
half-plane.

The natural setting here is to work at scaling. Naively one might
start with  $u$ and $v$ in the space $\dot H^{-\frac12}$. However,
this does not seem to work so well, so instead we use  its cousins 
$DU^2$ and $DV^2$. Appendix \ref{a:uv} provides a self contained introduction 
for these spaces, including some new results. 

We list the properties which we need in this section. We use a suggestive formal notation 
which will be made precise (and proven) in Appendix \ref{a:uv}.  For simplicity in the 
discussion below we consider spaces of functions defined on $\R$, but similar properties
apply in any bounded or unbounded interval $(a,b) \subset \R$.
\begin{enumerate} 
\item  
$U^2$ and $V^2$  have the same scaling properties as $BV$ and 
$L^\infty$ and satisfy  
\begin{equation}  \label{uv:em}
U^2 \subset V^2 \subset \mathcal{R}  
.\end{equation} 
where $\mathcal{R}$ is the set of ruled functions, .i.e. functions with limits 
from the left and the right everywhere including one sided limits at the 
endpoints \red{$\pm \infty$}. Ruled functions are bounded. Further, we have the embedding
relation with the homogeneous Besov spaces $\dot B^{-\frac12}_{2,q}$,   
(see Corollary \ref{appBemb})
 \begin{equation} \label{appbemb} 
L^1 +  \dot B^{-\frac12}_{2,1} \subset  DU^2 \subset DV^2 \subset \dot B^{-\frac12}_{2,\infty} \cap L^\infty 
 \end{equation}
\red{ Moreover $U^p\subset V^p$ and
 \[ \Vert u \Vert_{V^p} \le 2^{\frac1p} \Vert u \Vert_{U^p}. \]
 If $ f \in L^1$ then
 \begin{equation}\label{eq:conv}  \Vert f* v \Vert_{V^p} \le \Vert f \Vert_{L^1} \Vert v \Vert_{V^p}, \quad \Vert f*u \Vert_{U^p} \le \Vert f \Vert_{L^1} \Vert u \Vert_{U^p}. \end{equation}
 } 
 \item 
$DU^2$ and $DV^2$ are the spaces of distributional derivatives of $U^2$ and $V^2$ 
functions.  In particular  $u \in U^2$ if and only if $\lim_{t \to -\infty} u(t) =0 $ and $u' \in DU^2 $, see definition \ref{B16}. 
Then 
\[ \Vert u \Vert_{U^2} = \Vert u' \Vert_{DU^2} \] 
which is less than infinity if and only if $u \in U^2$. 
Moreover, if $v$  is left continuous and $\lim_{t\to \infty} v(t) = 0 $ then  
\[ \Vert v \Vert_{V^2} = \Vert v' \Vert_{DV^2}. \]  
 \item The bilinear estimates 
\begin{equation} \label{uv:bilinear} 
\begin{split}   \Vert v u \Vert_{DU^2} \le &\ 2 \Vert v \Vert_{V^2} \Vert u \Vert_{DU^2} \\  \Vert vu \Vert_{DV^2} \le & \  \Vert v \Vert_{DV^2} \Vert u \Vert_{U^2} 
\end{split} \end{equation}  
 hold (see Definition \ref{stieltjes} and the discussion thereafter). 
 \end{enumerate} 

To iteratively solve the system \eqref{scatter-re} we define the one step operator
\[
  \phi \to  L\phi(t) =  \int_{x<y<t}  e^{iz(y-x)} u(x) \overline{v(y)} \phi(x) dy dx. 
\]

Our first bound for $L$ is as follows:

\begin{lemma}\label{piece}  Let $\im z> 0$. Then the  operator $L$
satisfies 
\[  
\Vert L \Vert_{V^2\to U^2} \le 8 \Vert e^{-i\real z x} v \Vert_{DU^2} 
 \Vert e^{-i\real z x} u \Vert_{DU^2}. 
\]    
\end{lemma} 
 
\begin{proof} 
It suffices to consider $z = i $, by rescaling and including the oscillatory factor $e^{-i\real z x} $ into $u$. Then, applying the second and third property above several times, with $\eta= \chi_{t<0} e^{t}$  
\[ 
\begin{split} 
\Vert L \phi \Vert_{U^2} = &\  \Vert (L\phi)' \Vert_{DU^2} \\
   = & \  \left\Vert \overline{v(y)} \int_{-\infty}^y  e^{-(y-x)} u(x) \phi(x) dx 
\right\Vert_{DU^2}    \\
\le & \ 2  \Vert v \Vert_{DU^2} \Vert \eta * ( u \phi) \Vert_{V^2} \\
\le & \  4 \Vert v \Vert_{DU^2} \Vert \eta * ( u \phi) \Vert_{U^2} \\
\le & \ 4\Vert v \Vert_{DU^2} \Vert u \phi \Vert_{DU^2} \\
\le &\  8 \Vert v \Vert_{DU^2} \Vert u \Vert_{DU^2} \Vert \phi \Vert_{V^2}. 
\end{split} 
\] 
In addition we have used 
\[ 
\Vert \eta* w \Vert_{U^2} 
=  \Vert (\eta* w)' \Vert_{DU^2} = \| \eta' * w\|_{DU^2} \leq \|w\|_{DU^2} 
\]  
\end{proof} 

This bound is quite sharp on the real line, but as we move $z$ into the upper half-space
we can do better than this.   This is done using some more localized versions of the  $DU^2$ spaces. 
\begin{definition} \label{l2du2}
Given a frequency scale $\sigma > 0$ we define 
\[ \Vert u \Vert_{l^{p}_\sigma DU^2} = \left\Vert \Vert \chi_{[k/\sigma
  ,(k+1)/\sigma]} u\Vert_{DU^2} \right\Vert_{l^{p}_k} \] 
and similarly for $U^2$,
\[ \Vert v \Vert_{l^{p}_\sigma U^2} = \left\Vert \Vert
  \chi_{[k/\sigma, (k+1)/\sigma]} u\Vert_{U^2}
\right\Vert_{l^{p}_k}. \] 
where, for an interval $I$, $\chi_{I}$
is a smooth cutoff associated to the interval $I$, and the above
functions $\chi_{[k/\sigma ,(k+1)/\sigma]}$ form a partition of unity.
\end{definition} 

These spaces are translation invariant, and the norm of a translated
function is at most a fixed constant time the norm of the original
function. Connecting the two spaces we have the following
\begin{lemma}
We have
\begin{equation}\label{lp-udu}
\| u\|_{l^p_\sigma U^2} \lesssim \| \partial u\|_{l^p_\sigma DU^2} + \sigma \| u\|_{l^p_{\sigma} DU^2}
\end{equation}
\end{lemma}
\begin{proof}
Since $\|u\|_{U^2} = \|  \partial u \|_{DU^2}$,
we can write
\[
\begin{split}
\| \chi_{[k/\sigma, (k+1)/\sigma]} u\|_{U^2} \lesssim & \ 
 \| \chi_{[k/\sigma, (k+1)/\sigma]} \partial u\|_{DU^2} +
 \| \partial_x  \chi_{[k/\sigma, (k+1)/\sigma]} u\|_{DU^2}
\\ \lesssim & \ 
 \| \chi_{[k/\sigma, (k+1)/\sigma]} \partial u\|_{DU^2} +
 \sigma \|  \tilde \chi_{[k/\sigma, (k+1)/\sigma]} u\|_{DU^2}
\end{split}
\]
where $\tilde \chi$ is a cutoff selecting a slightly larger set.
Then the conclusion follows after $l^p$ summation.
\end{proof}

In the case $p=2$ we have an alternative  simple characterization of these spaces:
 \begin{lemma}\label{characterization} 
The spaces $l^2_{\sigma} DU^2$ can be characterized as follows
\begin{equation} 
l^2_\sigma  DU^2 =  DU^2 + \sigma^{\frac12} L^2. 
\end{equation} 
\end{lemma}
We postpone the proof for the end of Appendix \ref{a:uv}. As a corollary we have::

\begin{cor}\label{DU2embedding} The following embeddings hold
\begin{equation} 
 B^{-\frac12}_{2,1} \subset l^2_1 DU^2 \subset B^{-\frac12}_{2,\infty}.
\end{equation} 
\end{cor}
\begin{proof} This is an immediate consequence of Lemma \ref{characterization} 
and the analogous homogeneous result in Corollary  \ref{appBemb}.  
\end{proof} 

On the other hand, for $p > 2$ we will use the following embeddings:

\begin{lemma}\label{lpem} 
For $p \geq 2$ we have the bounds
\begin{equation}\label{emb}
\Vert u \Vert_{l^p_\tau DU^2} \lesssim  \tau^{\frac1p -1} \Vert u \Vert_{\dot H^{\frac12-\frac1p}}, 
\end{equation}\label{scale1} 
Also if $0 \le \tau_1 \le \tau_2$ then
\begin{equation} 
\label{refine}
\Vert u \Vert_{l^p_{\tau_2} DU^2} \lesssim \Vert u \Vert_{l^p_{\tau_1} DU^2} \lesssim (\frac{\tau_2}{\tau_1})^{1-\frac1p} \Vert u \Vert_{l^p_{\red{\tau_2}}DU^2}.
\end{equation}  
\end{lemma}

\begin{proof}
The  inequality \eqref{emb} follows via the Sobolev embedding 
\[ \Vert u \Vert_{l^p_\tau DU^2} \lesssim \tau^{\frac1p-1} \Vert u
\Vert_{l^p_\tau L^p} = \tau^{\frac1p-1} \Vert u \Vert_{L^p} \lesssim
\tau^{\frac1p-1} \Vert u \Vert_{\dot H^{\frac12-\frac1p} }. \] 
The second inequality and the third inequality are  a consequences of Lemma \ref{characterization}: It suffices to verify the last inequality for $u$ supported on an interval on length $\tau_1^{-1}$. Then 
\[ \Vert u \Vert_{DU^2+L^2} \lesssim  (\tau_2/\tau_1)^{\frac12}  \Vert_{l^2_{\tau_2}DU^2}  \le (\tau_2/\tau_1)^{1-\frac1p} \Vert u \Vert_{l^p_{\tau_2} DU^2} \] 
where we used the support assumption for the last inequality. 
\end{proof}

We also need to understand the effect of phase shifts:

\begin{lemma}\label{DU2shift} Suppose that $|\xi| \lesssim \sigma$. Then 
\begin{equation} 
\| e^{ix\xi} u\|_{l^2_\sigma DU^2} \lesssim \| u\|_{l^2_\sigma DU^2}, \qquad \| e^{ix\xi} u\|_{l^2_\sigma U^2} \lesssim 
\| u\|_{l^2_\sigma U^2}.
\end{equation} 
\end{lemma}
\begin{proof}
Since $|\xi| \lesssim \sigma$, it follows that the function $e^{ix\xi}$ is uniformly smooth and bounded on the supports
of the cutoff functions $\chi_{[k/\sigma,(k+1)/\sigma]}$ in Definition~\ref{l2du2}. The desired conclusion immediately follows.
\end{proof}

As a consequence we have

\begin{cor}\label{DU2osc} 
Let $\sigma > 0$ and $\xi \in \R$. Then the following estimate holds 
\[  \Vert e^{i\xi x} u \Vert_{l^2_\sigma DU^2} \lesssim  \left(\frac{1+ \sigma+|\xi|}{\sigma}\right)^\frac12    \Vert u \Vert_{l^2_1DU^2}. \] 
\end{cor} 
\begin{proof} 
If  $|\xi|\le \sigma$ then we by the previous Lemma we have
\[
  \Vert e^{i\xi x} u \Vert_{l^2_\sigma DU^2} \lesssim  \Vert u \Vert_{l^2_\sigma DU^2} \lesssim
\max\{ 1, \sigma^{-\frac12}\} \| u\|_{l^2_1 DU^2}.
\]
Else we use $|\xi|$ as an intermediate threshold,
\[
\begin{split}
 \Vert e^{i\xi x} u \Vert_{l^2_\sigma DU^2} \lesssim & \    (|\xi| / \sigma)^\frac12   
\Vert e^{i\xi x} u \Vert_{l^2_{|\xi|} DU^2} \lesssim 
 (|\xi| / \sigma)^\frac12     \Vert u \Vert_{l^2_{|\xi|} DU^2}\\  \lesssim & \ 
  (|\xi| / \sigma)^\frac12 \max\{ 1, |\xi|^{-\frac12}\}
\| u\|_{l^2_1 DU^2}. 
\end{split}
\]
\end{proof}

Using the above spatially localized  norms we can prove a stronger form of  Lemma \ref{piece}:
\begin{lemma}\label{piece2}  Let $\im z>0$. Then the  operator $L$
satisfies 
\[  
\Vert L \Vert_{U^2\to U^2} \lesssim \Vert e^{-i\real z x} u \Vert_{l^2_{\im z}  DU^2} 
 \Vert e^{-i\real z x} v \Vert_{l^2_{\im z} DU^2}. 
\]    
\end{lemma} 

\begin{proof} As above is suffices to prove this for $z=i$. We repeat the 
previous argument, but with some obvious changes:
\[
\Vert L \psi \Vert_{U^2}
\lesssim  \Vert v \Vert_{l^2 DU^2} \Vert \eta * ( u \psi) \Vert_{l^2 U^2} 
\lesssim \Vert v \Vert_{l^2 DU^2} \Vert u \psi \Vert_{l^2 DU^2} 
\lesssim \Vert v \Vert_{l^2 DU^2} \Vert u \Vert_{l^2 DU^2} \Vert \psi \Vert_{U^2}. 
\]
\end{proof}  
 
The main estimate of this section follows:
\begin{proposition}\label{t2j}
  The iterated integrals $T_{2j}(z)$ and $ \tilde T_{2j}$ satisfy
  the following bounds in the upper half-plane:
\begin{equation}\label{eq:t2j}
 |T_{2j}(z)|  + |\tilde T_{2j}(z)|  \le \left( c \Vert e^{i\real z x} u \Vert_{l^2_{\im z} DU^2} \right)^{j}  
\left( c \Vert e^{-i\real z x} v \Vert_{l^2_{\im z} DU^2} \right)^{j}  
\end{equation}
\end{proposition}

\begin{proof} [Proof of Proposition~\ref{t2j}]
We begin with the bound for $T_{2j}$. The first component of the solution to 
\eqref{scatter-re} is defined by 
\[
\psi_1 (t) = \red{e^{-iz t}} \sum_{j = 0}^\infty (L^j 1)( t)
\]
provided that this series converges uniformly. Thus, the transmission coefficient $T$ 
is given by 
\[
T^{-1} (z) = \sum_{j=0}^\infty \lim_{t \to \infty} (L^j(1))(t).
\]
Functions in $U^2$ have left and right limits everywhere which together with 
Lemma \ref{piece2} completes the proof of the $T_{2j}$ part of \eqref{eq:t2j}.
To switch to $\Tilde T_{2j}$ we begin with the relation
\[
\sum_{j = 1}^\infty \tilde T_{2j} = \log \left( 1+ \sum_{j = 1}^\infty T_{2j} \right)
\]
Recalling that both $T_{2j}$ and  $\tilde T_{2j}$ are homogeneous multilinear forms
of degree $2j$, replacing $u$ by $zu$ and $v=\bar u$ by $z v$  we must also have
the formal series relation
\[
\sum_{j = 1}^\infty \zeta^{j} \tilde T_{2j} = \log \left( 1+ \sum_{j = 1}^\infty \zeta^{j} T_{2j} \right)
\]

We can use \eqref{eq:t2j} to bound the size of the coefficients for
the series on the right. To achieve that we introduce a partial order
``$\preceq$'' on holomorphic functions near zero, where $g \preceq h$
means that the absolute value of every coefficient of the Taylor
series of $g$ at zero is bounded by the corresponding coefficient of
the Taylor series of $h$ at zero. This order is easily seen to be
compatible with the addition, multiplication and composition of
holomorphic functions.

In particular we note that
\[ 
\ln(1+\zeta) \preceq \frac{\zeta}{1-\zeta}=:f(\zeta)  
\] 
Bounding the coefficients in the
Taylor series for the logarithm by $1$ and $T_{2j}$ as in
\eqref{eq:t2j},  it follows that the series on the left is dominated by
\[
\sum_{j = 1}^\infty \zeta^{j} \tilde T_{2j} \preceq f \circ f ( C \zeta).
\]
and
\[
C = c^2 \Vert e^{i\real z x} u \Vert_{l^2_{\im z} DU^2}
\Vert e^{-i\real z x} v \Vert_{l^2_{\im z} DU^2}.
\]
Computing the Taylor series for the function on the right 
we obtain
\[
\sum_{j = 1}^\infty \zeta^{j} \tilde T_{2j} \preceq \sum_{j =1}^\infty 2^{j-1}C^j \zeta^j
\]
which yields the inequality
\[
 |\tilde T_{2j}| \leq 2^{j-1}C^j 
\]
thus proving the $\tilde T_{2j}$ part of \eqref{eq:t2j} with $c$ replaced by $2c$.

\end{proof}

As an immediate corollary to Proposition \ref{t2j} we have: 
\begin{cor}\label{c:sum}
Assume that 
\begin{equation}
\|(u,v)\|_{l^2_1 DU^2} \ll 1. 
\end{equation}
Then the  formal series for $T^{-1}$ and for $\ln T$ converge uniformly on the
 half-line $[i,i\infty)$, and more generally in the region  $\{\im z \geq 1+|\real z|\}$.
\end{cor}

The latter statement follows from Corollary~\ref{DU2osc}, which yields  smallness
for $z = \xi+i\tau$ in the above region.

Furthermore, we note that the function $T^{-1}$ is in effect well defined in the entire 
upper half-plane:

\begin{cor}\label{c:sum+}
Assume that  $(u,v) \in {l^2_1 DU^2}$.
 In both the defocusing and the focusing case the map 
$T^{-1}$ is well defined as a  holomorphic function in the whole upper half plane. 
\end{cor}

\begin{proof} 
Consider the ODE \eqref{scatter-re} with  $z = \xi + i \tau$ and $\tau > 0$.
By Corollary~\ref{DU2osc} we have
  \[ 
\Vert e^{ix\xi} u \Vert_{l^2_\tau DU^2} \lesssim \left(\frac{1+|\xi| + \tau}{\tau}\right)^\frac12 \Vert u
  \Vert_{l^2 DU^2}.  
\] 
Here we cannot use the argument in Proposition~\ref{t2j} because we lack smallness. 
To remedy this we partition the real line into three intervals 
\[
\R = (- \infty,x_0] \cup (x_0,x_1) \cup [x_1,\infty)
\] 
so that on the first and the last interval we have smallness for $\Vert e^{ix\xi} u \Vert_{l^2_\tau DU^2} $.
In particular we can apply the argument in Proposition~\ref{t2j} in order to solve 
 \eqref{scatter-re}  up to $x_0$. From $x_0$ we continue by solving the ODE
  up to  $x_1$, and from there on we repeat the previous
  argument.  We obtain a global solution $(\psi_1,\psi_2)$ for  \eqref{scatter-re} with $\psi_1 \in U^2$ 
and $\psi_2 \in DU^2$, and the inverse transmission coefficient $T^{-1}(z)$ is obtained as the limit of 
$\psi_1$ at infinity. Holomorphy in $z$ is obvious.
\end{proof} 

We note that $T^{-1}$ can have no zeroes in the region of applicability of Corollary~\ref{c:sum}.
However, in the focusing case $T^{-1}$ may have zeroes in the upper half-plane.
If $T^{-1}(z)=0$ then $z$ is an eigenvalue of the Lax operator.  Eigenvalues are isolated since
$T^{-1}(z)$ is analytic.  The geometric multiplicity is always one, which is a
consequence of the structure of the Jost solutions.   The order of the zero of $T^{-1}$ is known to be the algebraic
multiplicity of the eigenvalue. This can be seen using the B\"acklund transform, and is beyond our goals in here.

The next step is to assume in addition that $u,v \in H^s$, in which
case we expect that there is additional decay for $T_{2j}$. Precisely,
we have

\begin{proposition}\label{t2j-s}
Assume that $u,v \in H^s$. Then for $s$ in the range 
\begin{equation}
-\frac12 < s \le  \frac{j-1}2
\end{equation}
we have the pointwise bounds
\begin{equation}\label{t2j-s-inf}
|T_{2j}(i\tau)| + |\tilde T_{2j}(i\tau)| \leq c\left( 1+ \frac{1}{2s+1}\right) 
\tau^{-2s-1} \| (u,v) \|_{H^s}^2(c \|(u,v)\|_{l^2_1 DU^2})^{2j-2},
\end{equation}
as well as  the integrated bound  for $j>1$
\begin{equation}\label{t2j-s-int}
\int_1^\infty \!\! \tau^{2s} (|T_{2j}(i\tau/2)| + |\tilde T_{2j}(i\tau/2)|) d\tau 
 \, \leq\, \left(\! 1\! +\! \frac{1}{j-1-2s}\! +\! \frac{1}{(2s+1)^2}\! \right)  \| (u,v) \|_{H^s}^2(c \|(u,v)\|_{l^2_1 DU^2})^{2j-2}.
\end{equation}
\end{proposition}
Here $c$ is a large universal constant, and the above bounds are
uniform in $j$ and $s$.

\begin{proof}[Proof of Proposition~\ref{t2j-s}]
  The starting point for this proof is the bound \eqref{eq:t2j}, so it
  makes no difference whether we work with $T_{2j}$ or $\tilde
  T_{2j}$; we choose the former.  Greek letters $\lambda$ and $\mu$
  are powers of $2$. Sums $\sum_\lambda \dots $ are understood as
  $\sum_{j=0,\lambda=2^ j}^{\infty}\dots $.  We define $u_\lambda$ by
  the Fourier multiplication by a characteristic function $ \hat u_1 =
  \chi_{|\xi|<1} \hat u$, $\hat u_{\lambda} = \chi_{\lambda \le |\xi|<
    2\lambda}\hat u$ , $\hat u_{<\lambda} = \chi_{|\xi|<\lambda} \hat
  u$ and by an abuse of notation we write $(u,v)_\lambda =
  (u_\lambda,v_\lambda)$. Then $ u = \sum_\lambda u_\lambda$ .  We
  expand the functions in $\tilde T_{2j}$ and recall that we only
  consider $\tau\ge 1 $.  We separate the two highest dyadic
  frequencies, which we denote by $\lambda_1 \geq \lambda_2 \geq 1$.
  Each such pair occurs $O(j^2)$ times, which is subexponential and
  thus can be neglected. We carry out the summation in the other terms
  up to frequency $\lambda_2$. Using Proposition~\ref{t2j} we obtain
  the bound
\begin{equation}\label{prelimbound}  
 |T_{2j}(i\tau)|\lesssim \sum_{\lambda_1 \geq \lambda_2} \|(u,v)_{\lambda_1} \|_{l^2_\tau DU^2}
 \|(u,v)_{\lambda_2} \|_{l^2_\tau DU^2} \| (u,v)_{\leq \lambda_2} \|_{l^2_\tau DU^2}^{2j-2}.
\end{equation} 
To prove the claim for $T_{2j}$ we will bound the summands on the right hand side and carry out the summation. 
The first two factors we estimate in the $H^s$ norm using Lemma~\ref{characterization}. If $1<\lambda \le \tau$ we estimate 
\[ \Vert u_\lambda \Vert_{l^2_\tau DU^2} \lesssim  \tau^{-1/2}  \Vert u_\lambda \Vert_{L^2} \lesssim \lambda^{-s} \tau^{-\frac12} \Vert u \Vert_{H^s}, \] 
and if $\lambda \ge \tau $ we also use Corollary \ref{appBemb} 
in the form \eqref{appbemb}  
\[ \Vert u_\lambda \Vert_{l^2_\tau DU^2} \lesssim  \Vert u_\lambda \Vert_{DU^2} 
\lesssim \lambda^{-1/2} \Vert  u_\lambda \Vert_{L^2} \le \lambda^{-s-\frac12} 
\Vert u \Vert_{H^s}. \]  
 
The remaining factors on the right hand side of \eqref{prelimbound} we estimate in terms of the $l^2_1 DU^2$ norm,
using \eqref{refine} if $\lambda \ge \tau$,
\[ \Vert u_{<\lambda} \Vert_{l^2_\tau DU^2} \lesssim \Vert u_{<\lambda} \Vert_{l^2_1DU^2}, \]  
and if $\lambda<\tau$ it follows from Lemma \ref{characterization} and \eqref{DU2embedding} that
\[ \Vert u_{<\lambda} \Vert_{l^2_\tau DU^2} \lesssim \tau^{-1/2} \Vert u_{<\lambda} \Vert_{L^2} \lesssim  (\lambda/\tau)^{\frac12} \Vert u_{<\lambda} \Vert_{l^2_1 DU^2}. \]  
 We obtain a bound of the form
\begin{equation}\label{fix-lambda}
 |T_{2j}(i\tau)|\lesssim \tau^{-2s-1} \sum_{\lambda_1 \geq \lambda_2} 
C(\tau,\lambda_1,\lambda_2) \|(u,v)_{\lambda_1} \|_{H^s}
 \|(u,v)_{\lambda_2} \|_{H^s} \| (u,v) \|_{l^2_1 DU^2}^{2j-2}
\end{equation}
where the constant $C(\tau,\lambda_1,\lambda_2) $ depends on the relative 
position of the entries as follows:
\[
C(\tau, \lambda_1,\lambda_2) = \left\{ \begin{array}{lc} \tau^{2s+1} \lambda_1^{-s} \tau^{-\frac12} \lambda_2^{-s}\tau^{-\frac12} (\lambda_2/\tau)^{j-1} = 
    \left(\dfrac{\lambda_1}{\tau}\right)^{-2s+j-1} \left(\dfrac{\lambda_2}{\lambda_1}\right)^{-s+j-1} 
             & \lambda_2 \leq \lambda_1 \leq \tau \\[4mm] 
\tau^{2s+1}\lambda_1^{-s-\frac12}\lambda_2^{-s} \tau^{-\frac12} (\lambda_2/\tau)^{j-1}=     \left(\dfrac{\tau}{\lambda_1}\right)^{s+\frac12} \left(\dfrac{\lambda_2}{\tau}\right)^{-s+j-1} 
              &  \lambda_2\leq \tau \leq \lambda_1 \\[4mm]  
\tau^{2s+1}\lambda_1^{-s-\frac12} \lambda_2^{-s-\frac12} =     \left(\dfrac{\tau}{\lambda_1}\right)^{s+\frac12} \left(\dfrac{\tau}{\lambda_2}\right)^{s+\frac12}  
     & \tau \leq \lambda_2 \leq \lambda_1
\end{array}
\right.
\]
More precisely, by an application of Schur's lemma and the Cauchy-Schwarz inequality
\begin{equation} \label{CS} 
\begin{split} 
\tau^{2s+1}|T_{2j}(i\tau)| & \lesssim \sqrt{A_p B_p} \Vert (u,v) \Vert^2_{H^s}
\Vert (u,v) \Vert_{l^2_1(DU^2)}^{2j-2}  \\
\int_1^{\infty} \tau^{2s} |T_{2j}(i\tau)|d\tau & \lesssim A_i \Vert (u,v) \Vert^2_{H^s} \Vert (u,v) \Vert_{l^2_1(DU^2)}^{2j-2}
\end{split} 
\end{equation} 
with 
\[  A_p = \sup_{\tau,\lambda_1} \sum_{\lambda_2} C(\tau,\lambda_1,\lambda_2) , \quad 
B_p = \sup_{\tau,\lambda_2} \sum_{\lambda_1} C(\tau,\lambda_1,\lambda_2)  \] 
\[  A_i = \max\left\{ \sup_{\lambda_1} \sum_{\tau,\lambda_2} C(\tau,\lambda_1,\lambda_2) ,  \sup_{\lambda_2} \sum_{\tau, \lambda_1} C(\tau,\lambda_1,\lambda_2)  \right\} .  \] 
We break the (implicit) sum in \eqref{CS} up into three sums given by the constraints 
\[ \lambda_2\le \lambda_1 \le \tau, \quad \lambda_2 \le \tau \le \lambda_2, \quad \tau \le \lambda_2\le \lambda_1.  \]
 They are all given by geometric sums. The pointwise bounds are easy
 consequences and we provide more details for the slightly more
 involved integrated bounds of \eqref{t2j-s-int} and bound the  $A_i$
up to a multiplicative constant by 
\begin{equation} 
 1+  \dfrac1{(j-1-s)(j-1-2s)} \quad \text{ for the case } \lambda_1 \le \lambda_2 \le \tau   
\end{equation} 
\begin{equation}  
1+ \dfrac1{(s+\frac12)(j+1-s)} \qquad \text{ for the case } \lambda_1 \le \tau \le \lambda_2 
\end{equation} 
\begin{equation} 
  1+  \dfrac1{(s+\frac12)^2}    \qquad \text{ for the case } \tau \le \lambda_1 \le \lambda_2.
\end{equation}
This completes the proof.  
\end{proof}

We remark that, as a consequence of this last proposition, the proof
of  Theorem \ref{energies} reduces to considering the terms $\tilde T_{2j}$
for $j \leq 2s+1$ since we obtain the following corollary. 

\begin{cor}
If $\| u \|_{l^2 DU^2} \ll 1 $ then we have the bound for $0<s $ and $[2s+1] \le N$
\begin{equation}
\int_1^\infty\!\! \tau^{2s} \Big| \ln T(i\tau/2) + \! \sum_{j = 1}^{N}  \tilde T_{2j}(i\tau/2)\Big| d\tau \le \Big( 1+ \frac1{\red{N}-2s}
+\frac{1}{(2s+1)^2}\Big)  
 \| (u,v) \|_{\dot H^s}^2(c \|(u,v)\|_{l^2 DU^2})^{2[2s+1]}.
\end{equation}
\end{cor}
In view of the identity for $T_2$ of the previous section, this bound
suffices for the proof of Theorem \ref{energies} in the range $-\frac12 < s
< \frac12$. Restating our main theorem in this case, we have
\begin{proposition} 
Let $-\frac12  < s < \frac12$, and $u \in H^s$ with 
$
\| u \|_{l^2 DU^2} \ll 1$.  
Then $E_s(u)$ is well defined, and 
\begin{equation} 
\left| E_s(u) - \Vert u \Vert_{H^s}^2 \right| \lesssim  \Vert u \Vert_{H^s}^2  \|u\|_{l^2_1 DU^2}^2  .               
\end{equation} 
\end{proposition}

 We conclude the section with  a further discussion of the
case $-\frac12 < s < 0$, which was our original goal in this paper.
   First, we note that, as a consequence of \eqref{eq:t2j}, we have the bound
\[ 
|T_{2j}(z)|^{\frac1{j}}  \le c  \frac{\langle \real z \rangle^{-2s}}{\im z}  \Vert u \Vert_{H^s}^2   
\]    
Thus,  the expansion for $T(z)$ converges for 
$ 
\im z \ge (\real z)^{-2s}  c\Vert u \Vert_{H^s}^2 $.             
In particular, if smallness is assumed,
$
c\Vert u \Vert_{H^s}^2 \le 1/2$ 
then the series converges for $\im z \ge (\real z)^{-2s} $, and the 
following estimate and also the corresponding integrated bound hold:
\[ 
|\ln T(z) +T_2(z)| \le  2 \frac{\langle\real z\rangle^{-2s}}{\im z}\Vert u \Vert_{H^s}^2  . 
\]

\newsection{Expansions for  the iterative  integrals $\tilde T_{2j}$} 
\label{s:tt2j-exp}

In view of the bounds of the previous section, it remains to
separately consider the terms $\tilde T_{2j}$ for $j \leq 2s+1$.  To
avoid the degeneracy at $s = -\frac12$ we will harmlessly assume
throughout the section that $s \geq 0$. The implicit constants in the 
present section depend on $j$, in contrast to the previous section.

It is crucial in this section that we will take advantage of the fact, proved
in Theorem~\ref{primitive}, that the iterated integrals in $\tilde T_{2j}$
have fully connected symbols. Our first result here is as follows:

\begin{proposition}\label{t2j-prime}
The iterated integrals  $\tilde T_{2j}$  satisfy the following bounds:
\begin{equation}\label{tt2j}
 |\tilde T_{2j}(z)|\lesssim \Vert e^{i\real z x} u \Vert_{l^{2j}_{\im z} DU^2}^{2j}  
\end{equation}
\end{proposition}

As before, the next step is to assume in addition that $u,v \in H^s$,
in which case we expect that there is additional decay for
$\tilde T_{2j}$. Precisely, we have

\begin{proposition}\label{t2j-sprime}
Assume that $u,v \in H^s$. Then we have the pointwise bounds
\begin{equation}\label{t2j-sprime-inf}
|\tilde T_{2j}(i\tau)| \lesssim \tau^{-2s-1} \| (u,v) \|_{H^s}^2 \|(u,v)\|_{l^2_\tau DU^2}^{2j-2},
\qquad s \leq j-1
\end{equation}
as well as the integrated bound  
\begin{equation}\label{t2j-sprime-int}
\int_1^\infty \tau^{2s} |\tilde T_{2j}(i\tau/2)| ds 
 \lesssim  \left( 1+ \frac{1}{j-1-s}\right)  \| (u,v) \|_{H^s}^2\|(u,v)\|_{l^2_1 DU^2}^{2j-2},
\qquad 0\le s <  j-1.
\end{equation}
\end{proposition}

We remark that, as a consequence of this last proposition, the proof of Theorem
\ref{energies} reduces to considering the terms $\tilde T_{2j}$ for $0\le j \leq s+1$:

\begin{cor}\label{cor:6.3} 
We have the bound 
\begin{equation}
\int_1^\infty \tau^{2s} \Big| \ln T(i\tau/2)  +\sum_{j = 1}^{[s]}  \tilde T_{2j}(i\tau/2) \Big| d\tau  \lesssim 
\left(\frac1{(s+\frac12)^2}    +   \frac{1}{[s]+1-s} \right)      \| (u,v) \|_{\dot H^s}^2\|(u,v)\|_{l^2 DU^2}^{2[s]}.
\end{equation}
\end{cor}

We note that this (and the obvious improvement for $s<1$ since then there is no term $T_4$) suffices for the proof of Theorem \ref{energies} in the range $-\frac12 < s < 1$. 
We continue with  the proof of the above results.

\begin{proof}[Proof of Proposition~\ref{t2j-prime}] The proof is a
  direct consequence of Proposition~\ref{t2j}, taking into account Theorem \ref{primitive}. We localize spatially
  on the $(\im z)^{-1}$ scale. The key fact is that the multilinear
  forms in $\tilde T_{2j}$ have connected symbols, therefore their
  kernels have exponential off-diagonal decay on the $(\im z)^{-1}$
  scale. Thus only the $l^{2j}$ summability is needed, in order to account 
for the diagonal sum. 
\end{proof}

\begin{proof}[Proof of Proposition~\ref{t2j-sprime}]
  In order to apply the same method as in the proof of
  Proposition~\ref{t2j-s} we need to relate the space 
$l^{2j}_{\tau}  DU^2$ in Lemma \ref{t2j-prime} 
to  an $L^2$ based Sobolev space. For this we modify the estimate of the last section. We use the Sobolev type
  embedding \eqref{emb} for $p=2j$ and frequencies less than $\tau$,
  and simply the $DU^2$ norm for larger frequencies:
 If $1\le \lambda < \tau$ 
\[ \Vert u_\lambda \Vert_{l^{2j}_\tau DU^2} \lesssim   \tau^{\frac1{2j}-1} \Vert u_\lambda \Vert_{L^{2j}} \lesssim \left( \frac{\lambda}{\tau}\right)^{1-\frac1{2j} } 
 \lambda^{-s-\frac12} \Vert u \Vert_{H^s} \]      
and 
\[ \Vert u_{\le \lambda} \Vert_{l^{2j}DU^2} \lesssim \tau^{\frac1{2j}-1} \Vert u_{\le\lambda} \Vert_{L^p} \lesssim \tau^{\frac1{2j} -1} \lambda^{\frac12-\frac1{2j} } \Vert u_{<\lambda} \Vert_{L^2} \le  \left( \frac{\lambda}{\tau}\right)^{1-\frac1{2j} } 
\Vert u \Vert_{l^2_1 DU^2 }. \] 
If $\lambda \ge \tau $ we obtain                     
\[ \Vert u_\lambda \Vert_{l^{2j}_\tau DU^2} \le \Vert u_\lambda \Vert_{l^2_\tau DU^2} 
\lesssim  \Vert u_\lambda \Vert_{DU^2} \lesssim \lambda^{-s-\frac12} \Vert u \Vert_{H^s} \] 
and 
\[ \Vert u_{\le\lambda} \Vert_{l^{2j}_\tau  DU^2} \lesssim \Vert u \Vert_{l^2_1 DU^2}.\]

Using this  we obtain again a bound of the form \eqref{fix-lambda}, but with an improved 
factor for all frequencies below $\tau$: 
\begin{equation} \label{Ctaulambda} 
C(\tau, \lambda_1,\lambda_2) = \left\{ \begin{array}{ll}
    \left(\dfrac{\lambda_1}{\tau}\right)^{2(-s+j-1)} \left(\dfrac{\lambda_2}{\lambda_1}\right)^{ -s+2j-\frac52 +\frac1{2j}} 
             & \lambda_2 \leq \lambda_1 \leq \tau \\[4mm]
  \left(\dfrac{\lambda_1}{\tau}\right)^{-(s+\frac12)} \left(\dfrac{\lambda_2}{\tau}\right)^{-s+2j-\frac52 +\frac1{2j}} 
              &  \lambda_2\leq \tau \leq \lambda_1 \\[4mm]
\left(\dfrac{\lambda_1}{\tau}\right)^{-(s+\frac12)} \left(\dfrac{\lambda_2}{\tau}\right)^{-(s+\frac12)}  
     & \tau \leq \lambda_2 \leq \lambda_1
\end{array}
\right.
\end{equation} 
This allows us to conclude the proof in the same manner as in Proposition~\ref{t2j-s}, but with the  improved threshold $s <  j-1$  instead of  $s < (j-1)/2$.
Because of their importance later on we give the bounds for the constants 
$A_i$, again up to a multiplicative constant.  
\begin{equation} 
1+ \dfrac1{(2j-\frac52+\frac1{2j}-s)(j-1-s)}   \qquad \text{ for } \lambda_1 \le \lambda_2 \le \tau
\end{equation} 
\begin{equation} 
  1+  \dfrac{1}{(2j-\frac52+\frac1{2j}-s)(s+\frac12)}  \qquad \text{ for } \lambda_1 \le \tau \le \lambda_2                 
\end{equation} 
\begin{equation} 
1+ \dfrac1{(s+\frac12)^2} \qquad \text{ for } \tau \le \lambda_1 \le \lambda_2.   
\end{equation} 
where  $2j-\frac52 +\frac1{2j}-s$ is 
uniformly bounded from below.
\end{proof}

In order to understand the contributions of $\tilde T_{2j}$ to the
conserved energies $E_s$ in the remaining range $ s \ge j-1$ we need
to take into account the expansion of $\tilde T_{2j}$ at $i \infty$ in
powers of $z^{-1}$ in a similar but more complicated fashion as for
$T_2$.  Precisely, given $\Sigma $ a connected symbol of length $2j$,
we will consider higher order expansions and bounds for the iterated
integral
\[ 
T_{\Sigma}(z) = \int_{\Sigma}  \prod_{l=1}^j  e^{2iz(y_l-x_l)} u(x_l) v(y_l) dx_l dy_l. 
\] 
We will do this on the positive imaginary axis $z = i \tau$, $\tau > 0$, though the 
results can be easily translated to any $z$ in the upper half plane.

For regular enough $u$ (precisely $u \in H^{\frac12-\frac{1}{2j}}$)
the integral $T_{\red{\Sigma}}$ decays like $|z|^{-2j+1}$ on the positive
imaginary axis. Our goal here is to add to this a formal expansion,
with errors which have better decay at infinity. Precisely, we have
the following:

\begin{proposition}\label{p:t2j-exp}
For any connected  symbol $\Sigma$ of degree  $2j$, the integral $T_{\Sigma}(z)$ admits  a formal expansion 
\[
 T_\Sigma(z) \approx  \sum_{l = 0}^\infty   T_{\Sigma}^l (2z)^{-(2j-1+l)}
\]
where $ T_{\Sigma,l} $ are linear combinations of integrals of the form
\[
T_{\Sigma}^l = \sum_{|\alpha|+|\beta| = l}
c_{\alpha \beta} \int \prod_{l= 1}^{j} \partial^{\alpha_l} u_l  \partial^{\beta_l} v_l  dx
\]
so that the errors in the above partial expansion satisfy the bounds
\begin{equation}
\Big|  T_\Sigma(i\tau/2) -  \sum_{l = 0}^k (i\tau)^{-(2j-1+l)}  T_{\Sigma}^l \Big| \lesssim 
 \sum_{\substack{k+1 \leq |\alpha|+|\beta| \leq  2j-1+k \\  \max\{ \alpha_l, \beta_l\} \leq [\frac{k}2]+1}} |\tau|^{-|\alpha| - |\beta|}
\prod_k \| \partial^{\alpha_k} u_k \|_{l^{2j}_{\tau} DU^2}  \| \partial^{\beta_k} v_k \|_{l^{2j}_{\tau} DU^2}.
\end{equation}
for $\tau \ge 1$.
\end{proposition}

We remark that the formal series above is easily obtained by taking a Taylor
expansion of the $T_\Sigma $ integrand. To simplify the notation we
expand at the left point $x_1$ up to a certain order and obtain for
multiindices $\alpha$ and $\beta$
\[  
T_{\Sigma}^{\alpha \beta}  
=    \int_{\Sigma}  \frac1{\alpha!\beta!} \prod e^{2iz(y_j-x_j)} \partial^{\alpha_j}  u(x_1) \partial_{\beta_j} \overline{v(x_1)} (x_j-x_1)^{\alpha_j} (y_j-x_1)^{\beta_j}dx_j dy_j \] 
By an application of Fubini 
\[ T_{\Sigma}^{\alpha \beta} =  c_{\alpha\beta}  (2z)^{1-2j-|\alpha|-|\beta|}    \int   \frac1{\alpha!\beta!} \prod \partial^{\alpha_j}  u(x) \partial^{\beta_j} \overline{v(x)} dx\] 
where
 \[ c_{\alpha\beta} 
= \frac1{\alpha!\beta!} \int_{\Sigma,x_1=0} \prod e^{-(x_j-y_j)}  x_j^\alpha y_j^\alpha dx_j dy_j 
\]
However, this is not the way we will proceed in the proof of the proposition.

\begin{proof} [Proof of Proposition~\ref{p:t2j-exp}]
Let  $j>1$.  For this proof we relabel in a monotone fashion the set
\[
\{ x_1, y_1, \cdots x_j, y_j \} = \{ t_1, \cdots, t_{2j}\}
\]
and the functions
\[
\{ (u,v,\cdots,u,v)\} = \{ v_1, \cdots, v_{2n} \}
\]

We first consider the $t_{2j}$ dependent part of the integral $T_{\Sigma}$,
which we rewrite by an integration by parts as
\[
\int_{t_{2j-1}}^\infty e^{-\tau t_{2j}} v(t_{2j}) dt_{2j} 
=  - \frac{1}{\tau} e^{-\tau t_{2j}-1} v(t_{2j-1})  - \frac{1}{\tau} \int_{t_{2j-1}}^\infty e^{-\tau t_{2j}} v'(t_{2j}) d t_{2j}. 
\]
We do the same at the left endpoint $t_1$.

To obtain the expansion of $T_\Sigma$ up to degree $2j-1+k$ we
repeatedly apply this computation. At each step we get
a factor of $\dfrac{1}{\tau}$ and two terms,

\begin{enumerate} 
\item The boundary term which leaves us with an integral with a dimension lowered by $1$, and two functions evaluated at the same point.
\item A term where we differentiate one of the functions. 
\end{enumerate} 
We repeat this algorithm until either of the following stopping criteria is fulfilled:

\begin{enumerate} 
\item  If step (i) is applied $2j-1$ times; then all integration variables are equal. These 
terms give  exactly the coefficients $T_\Sigma^l$, $l \leq k$, in the expansion of $T_\Sigma$ up to  degree  $k$.
\item  If step (ii) is applied $k+1$ times. These terms are those where
the integration variables are not all equal; they are part of the
error term which we need to estimate.
\end{enumerate}

We remark on the critical role of the assumption that $\Sigma$ is
connected. This is what guarantees that at each step we are still
integrating a decaying exponential. In order to best balance
derivatives, we alternately apply the above steps at the left and at
the right. Starting on the right, this implies that we do it $\alpha^-
=[\frac{k+2}{2}]$ times on the left, respectively $\alpha^+= k+1 -
[\frac{k+2}{2}]$ times on the right.  To describe a general error term
that we need to estimate, we denote by $j^-$, respectively $j^+$ the
number of times the option (i) is taken on the left, respectively on
the right. These must satisfy
\[
0 \leq j^- \leq \alpha^-, \qquad  0 \leq j^+ \leq \alpha^+, \qquad j^-+j^+ \leq 2j-2
\] 

Our restricted domain of integration is 
$
\tilde\Sigma = \{   t_1 = \cdots =  t_{j^-+1} < \cdots < t_{2j - j^+} = \cdots = t_{2j} \},
$
and the phase  
$
\phi = \tau \sum_{i = 1}^j x_i - y_i
$
rests unchanged but is now restricted to $\tilde \Sigma$. 
 Thus, the corresponding terms in  the error are  linear combinations of integrals of the form
\[
I = \frac{1}{\tau^{2j-1+l}} \int_{\tilde \Sigma} e^\phi \prod_{i = 1}^{2j} \partial^{\alpha_i} v_i(t_i) d t_{j^-+1} \cdots dt_{2j-j^+}
\]
where the differentiation indices $\alpha_i$ must satisfy the constraints 
\[
\sum_{i = 1}^{1+j^-} \alpha_i = [(k+1)/2], \qquad  \qquad \sum_{i = 2j-j^+}^{2j} \alpha_i = [(k+2)/2].
\]
 The exponential $e^\phi$ decays fast 
in the maximum of the distances of the points. 
Denoting the first and the last  function by 
\[
v^- = \prod_{i = 1}^{1+j^-} \partial^{\alpha_i} v_i, \qquad v^+ =  \prod_{i = 2j-j^+}^{2j} \partial^{\alpha_i} v_i,
\]
we can view $I$ as a connected integral form applied to the functions
\[
v^{-}, v_{2+j^-}, \cdots, v_{2j-j^+ -1}, v^+ 
\]

Then the same argument as in the proof of Proposition~\ref{t2j}
yields the bound
\begin{equation}
|I| \lesssim \frac{1}{\tau^{2j-1+l}} \|v^{-}\|_{l^{\frac{2j}{j^-+1}}_\tau DU^2} \|v^{+}\|_{l^{\frac{2j}{j^++1}}_\tau DU^2}
\prod_{i = 2+j^-}^{2j-j^+ -1} \|v_i\|_{l^{2j}_\tau DU^2}
\end{equation}
where we have appropriately rebalanced the $l^p$ indices. We recall the  definition  Definition \ref{l2du2}, the bilinear estimate \eqref{uv:bilinear}
  \[ \Vert  uv \Vert_{l^p_\tau DU^2} 
= \Vert\, \Vert \chi_{[k/\tau,(k+1)/\tau)} (uv) \Vert_{DU^2} \Vert_{l^p} 
\lesssim \Vert\,  \Vert \chi_{[k/\tau,(k+1)/\tau)} u \Vert_{V^2} \Vert_{l^q}  
\Vert  \Vert \chi_{[k/\tau,(k+1) /\tau)} v \Vert_{DU^2} \Vert_{l^r} \]     
whenever $2\le p$, $\frac1p= \frac1q+\frac1r$ by H\"older's inequality. 
We apply the embedding \eqref{uv:em} and Lemma \ref{lp-udu} to bound 
\[   \Vert  u \Vert_{l^q_{\tau} V^2} \lesssim 
\Vert u' \Vert_{l^q DU^2} + \tau \Vert  u \Vert_{l^q DU^2}. \] 
After reordering the $\alpha_j$ so that $\alpha_1$ is the largest among
$\alpha_j$, $j\le j^-+1$ we recursively apply this estimate: 
\[
\| v^{-}\|_{l^{\frac{2j}{j^-+1}}_\tau DU^2} \lesssim \|\partial^{\alpha_1} v_1\|_{l^{2j} DU^2}  \prod_{i = 2}^{j^--1} 
( \|  \partial^{\alpha_i+1} v_i\|_{l^{2j} DU^2} + \tau^{-1} \|  \partial^{\alpha_i} v_i\|_{l^{2j} DU^2}).
\]
We argue  similarly for $v^+$. Summing up the results,
the total number of derivatives is at most $\alpha^+ + \alpha^- = 2j-1+k$,
and the largest is at most $[(k+2)/2]$. Thus the conclusion of the Proposition 
\ref{p:t2j-exp} follows.\end{proof}

Our last task is to convert the above bound into an estimate for $H^s$ functions. We relate the indices 
$s$ and $k$ by the relation
\begin{equation}\label{s-range}
j-1+\frac{k}2 \leq  s \leq j-1+\frac{k+1}2
\end{equation}
Then we have

\begin{proposition}\label{p:t2j-exps}
Let $s$ be as in \eqref{s-range}. Then the error estimates in the above expansion satisfy the pointwise bounds
\begin{equation}\label{tt2j-err}
\left|  T_\Sigma(i\tau/2) -  \sum_{l = 0}^k   T_\Sigma^l (i\tau)^{-(2j-1+l)} \right| \lesssim 
 \tau^{-2s-1} \| (u,v) \|_{\dot H^s}^2\|(u,v)\|_{l^2_1 DU^2}^{2j-2},
\end{equation}
as well as the integrated bound  
\begin{equation}\label{tt2j-err-int}
\int_1^\infty \tau^{2s}\left|  T_\Sigma(i\tau/2) -  \sum_{l = 0}^k   T_\Sigma^l (i\tau)^{-(2j-1+l)}\right|   d\tau 
 \lesssim\frac{1}{|\sin (2\pi s)|}  \| (u,v) \|_{\dot H^s}^2\|(u,v)\|_{l^2_1 DU^2}^{2j-2}.
\end{equation}
Moreover the  error estimates for $\tilde T_{2j}$ satisfy the integrated bound
\begin{equation}\label{tt2j-err-int+}
 \int_1^\infty \tau^{2s}\left|  \real \Big(\tilde T_{2j}(i\tau/2) +  \sum_{l = 0}^k (-1)^{j+l}  i^{l-1}   \tau^{-(2j-1+l)}  T_{2j}^l\Big) \right|   d\tau  
 \lesssim \frac{1}{|\sin (\pi s)|} \| (u,v) \|_{\dot H^s}^2\|(u,v)\|_{l^2_1 DU^2}^{2j-2}.
\end{equation}
\end{proposition}

The first two estimates transfer directly to $\tilde T_{2j}$, which is a linear
combination of primitive integrals of length $2j$.  It completes the estimates
 of Theorem \ref{energies} in all cases, except for $s$ near half
integers where we lack uniformity in the estimate. There  we
need to further take advantage of the fact that the conserved momenta
 $H_{2j+1}$ are real.

\begin{proof}
For the pointwise bound we need to show that 
\[
I= \tau^{-|\alpha|} \prod \| \partial^{\alpha_l} u \|_{l^{2j}_{\tau} DU^2}
\lesssim  \tau^{-2s+1}   \| u\|_{H^{s}}^2 \|u\|_{l^2 DU^2}^{2j-2}
\]
if  $s$ is in the range
\[
 j-1+\frac{k}2 \leq  s \leq j-1+\frac{k+1}2
\]
and 
\[  k+1 \le \sum  \alpha_j \le 2j-1+k,\qquad \max \alpha_j \le \left[\frac{k}2\right] +1. \]
As for Proposition \ref{t2j-sprime} we take a Littlewood-Paley decomposition and expand. Again we  order the terms so that the first two terms are of highest frequency. The worst case is when all derivatives fall on the first two terms, 
more precisely we may replace the condition above by  
\begin{equation} \label{ks}  \alpha_1 \le \left[ \frac{k}2\right],  k+1 \le \alpha_1+\alpha_2 \le 2j-1+k, \qquad \alpha_2 \le 2j-1+\left[ \frac{k+1}2 \right]  \end{equation}
and $\alpha_j=0$ for $j \ge 3$. We assume
\eqref{ks} in the following.  Repeating the argument of Proposition
\ref{t2j-sprime} - but adjusting for the derivatives - we have to
multiply the constants in \eqref{Ctaulambda} by $ \left(
\frac{\lambda_1}{\tau} \right)^{\alpha_1}\left( \frac{\lambda_2}{\tau}
\right)^{\alpha_2} $,
\begin{equation} 
 C(\tau, \lambda_1,\lambda_2,\alpha_1,\alpha_2) = \left\{ \begin{array}{ll}
    \left(\dfrac{\lambda_1}{\tau}\right)^{2(-s+j-1)+\alpha_1+\alpha_2} \left(\dfrac{\lambda_2}{\lambda_1}\right)^{ -s+2j-\frac52 +\frac1{2j}+\alpha_2} 
             & \lambda_2 \leq \lambda_1 \leq \tau \\[4mm] 
  \left(\dfrac{\lambda_1}{\tau}\right)^{-(s+\frac12)+\alpha_1} \left(\dfrac{\lambda_2}{\tau}\right)^{-s+2j-\frac52 +\frac1{2j}+\alpha_2} 
              &  \lambda_2\leq \tau \leq \lambda_1 \\[4mm]  
\left(\dfrac{\lambda_1}{\tau}\right)^{-(s+\frac12)+\alpha_1} \left(\dfrac{\lambda_2}{\tau}\right)^{-(s+\frac12)+\alpha_2}  
     & \tau \leq \lambda_2 \leq \lambda_1.
\end{array}
\right.
\end{equation} 
As above this implies the pointwise bound 
\[  \tau^{2s+1-|\alpha|}\prod \Vert \partial^{\alpha_l}u \Vert_{l^{2j}_\tau DU^2}   \le  \sqrt{A_p^{\alpha_1,\alpha_2}B_p^{\alpha_1,\alpha_2} }   \Vert u \Vert_{H^s}^2 \Vert u \Vert_{l^2_1 DU^2}^{2j-2} \] 
and the integrated bound 
\[ \int_1^\infty  \tau^{2s+1-|\alpha|}\prod \Vert \partial^{\alpha_l}u \Vert_{l^{2j}_\tau DU^2}d\tau  \le   A_i^{\alpha_1,\alpha_2}   \Vert u \Vert_{H^s}^2 \Vert u \Vert_{l^2_1 DU^2}^{2j-2} \]
where the constants are determined by geometric sums. Again we give the result 
for the integrated bounds $A_i$ up to multiplicative constants. 
\begin{equation} 1+  \dfrac1{(2(j-1-s) + \alpha_1+\alpha_2)(2j +\alpha_2-s -\frac52+\frac1{2j})} \le 1 + \frac1{2(j-1+\frac{k+1}2-s)(j-1+ \frac{k+1}2-s+\frac14 )}
\end{equation}  
in the range $\lambda_2 \le \lambda_1 \le \tau $,
\begin{equation} \label{eq:ktauk}1  + \frac1{(s+\frac12-\alpha_1)(2j+\alpha_2-s -\frac52+\frac1{2j})}
\le 1+ \dfrac1{(2(j-1+\frac{k+1}2)-s)(j-1+ \frac{k+1}2-s+\frac14 )} \end{equation} 
in the range $\lambda_1 \le \tau \le \lambda_2$ and finally  
 in the range $\tau \le \lambda_2 \le \lambda_1 $ 
\begin{equation}\label{eq:taukk}   1+  \frac1{(s+\frac12-\frac{k}2)(2s-2j-k)} \end{equation}  
We arrive at \eqref{tt2j-err} and \eqref{tt2j-err-int} and  
 note that estimate \eqref{tt2j-err-int+} follows directly from  \eqref{tt2j-err-int} unless $s$ is close 
to a half-integer.
 It remains to prove \eqref{tt2j-err-int+} near half integers. 
Hence we choose $s_0 = j+k-\frac12$  
 and consider $s$ close to $s_0$. To shorten the notation we define 
\[ \mathbb{T}_{2j}^{N}(i\tau/2)(u) = \tilde T_{2j}(i\tau/2)+ (-1)^j \sum_{l=0}^{N}  H_{2j,l}
 \tau^{-1-2j-l}. \] 
The crucial observation is that most partial estimates above extend to a larger range of $s$. 
The claim will follow from   the inequalities
\begin{equation} \label{eq:k}
\int_1^\infty (\tau^2-1)^s \left| \mathbb{T}_{2j}^{2k}(i\tau/2)(u) -\mathbb{T}_{2j}^{2k}(i\tau/2)(u_{>\tau}))\right|
\lesssim \Vert u \Vert_{H^s}^2 \Vert u \Vert_{l^2_1 DU^2}^{2j-2} 
\end{equation} 
and 
\begin{equation} \label{eq:k+1}
\int_1^\infty (\tau^2-1)^s \left|\mathbb{T}_{2j}^{2k+1}(i\tau/2)(u) -\mathbb{T}_{2j}^{2k+1}(i\tau/2)(u_{<\tau}))\right|
\lesssim \Vert u \Vert_{H^s}^2 \Vert u \Vert_{l^2_1 DU^2}^{2j-2} 
\end{equation} 
for $|s_0- s| \le \frac18$, and the similar estimates \eqref{eq:low} and \eqref{eq:high} below. 
Each term in the  Littlewood-Paley expansion of the factors in 
$ \mathbb{T}_{2j}^{2k+1} (u) - \mathbb{T}_{2j}^{2k+1}(u_{<\tau})$ 
contains at least one factor  with frequency $\ge \tau$ and hence $\lambda_1 \ge \tau$. Thus we only need to consider the cases $\lambda_2 \le \tau \le \lambda_1$
and $\tau\le \lambda_2\le \lambda_1$. In these cases the estimates above extend to $|s-s_0|\le \frac18$, see \eqref{eq:ktauk} and \eqref{eq:taukk}.  

Similarly each term of the Littlewood-Paley expansion of 
$ \mathbb{T}_{2j}^{2k} (u) - \mathbb{T}^{2k}(u_{>\tau})$
contains at least one factor with  frequency below $\tau$. If  $\lambda_2 \le \tau $ then we are in the middle regime which is bounded for $|s-s_0|\le \frac18$.
If $\lambda_2 \ge \tau$  we order the expansion so that $\lambda_3 \le \tau$, which we estimate by 
\[ \Vert u_{\lambda_3} \Vert_{l^{2j}_\tau DU^2} \le c \left( \frac{\lambda_3}\tau\right)^{\frac12-\frac1{2j}}    \Vert_{l^2_1 DU^2}.    \] 
Then 
\[ \int_1^\infty  \sum_{\lambda_3 \le \tau \le \lambda_2\le \lambda_1}  
\tau^{2s}  \prod_{j=1}^3 \Vert u_{\lambda_j} \Vert_{l^{2j}_\tau Du^2}
\Vert u_{\le \lambda_2} \Vert_{l^{2j}_\tau DU^2}^{2j-3}   d\tau 
\lesssim          \Vert u \Vert_{H^s}^2 \Vert u \Vert_{l^2_1 DU^2}^{2j-2}. \]  
The same argument shows that 
\begin{equation} \label{eq:high} \int_1^\infty (\tau^2-1)^s \left| \mathbb{T}^{2k} (u_{>\tau}) \right|d\tau
\lesssim \Vert u \Vert_{H^s}^2 \Vert u \Vert_{l^2_1 DU^2}^{2j-2} 
\end{equation}   
and                           
\begin{equation}\label{eq:low}  \int_1^\infty (\tau^2-1)^s \left| \mathbb{T}^{2k+1} (u_{<\tau}) \right|d\tau 
\lesssim \Vert u \Vert_{H^s}^2 \Vert u \Vert_{l^2_1 DU^2}^{2j-2}. \end{equation}  
Since $ \im \mathbb{T}^{2k} (u_{<\tau}) = \im \mathbb{T}^{2k+1} (u_{<\tau})$
we obtain the uniform estimate below $s_0$ by the triangle inequality.
For $ s_0\le s \le s_0 +\frac18 $ 
$ \im \mathbb{T}^{k} (u_{>\tau}) = \im \mathbb{T}^{k+1} (u_{>\tau})$  
and again the uniform estimate follows by an application of the triangle inequality.   \end{proof}

\newsection{Proof of the main theorems and variants}
\label{s:proof}
In this section we combine the bounds in the previous sections in
order to prove our main results in Theorem~\ref{energies} and Theorem~\ref{t:main+}. 
We also discuss some further developments of our ideas and results.

\subsection{Proof of Theorem \ref{energies}}
This is done in two steps. First we show that the energies $E_s$ defined by the right hand side of 
formulas \eqref{def+}, \eqref{def++} are smooth as functions of $u \in
H^s$, and satisfy the bounds in part (2) of the theorem. Secondly, we
show that they are conserved along the flow.

For this we begin with the multilinear expansion for $\ln T$, namely 
\[
-\ln T = \sum_{j = 1}^\infty \tilde T_{2j}
\]
which, by Corollary~\ref{c:sum}, converges on the half-line $i[1/2,\infty)$ provided that 
\[
\| u\|_{l^2 DU^2} \ll 1.  
\]
Correspondingly, this yields a multilinear expansion for the energies $E_s$,
\[
E_s =  \sum_{j = 1}^\infty E_{s,2j}
\]
We will estimate separately each of  the terms  in the series, in several steps:

\medskip

\subsubsection{The leading term $j=1$} 
As proved in Proposition~\ref{p:T2}, the first term in this expansion is exactly the 
$H^s$ norm,
\[
E_{s,2} = \| u\|_{H^s}^2
\]
It remains to consider the rest of the terms in the series, and show that we
can bound them as error terms.

\subsubsection{Large $j$, $j > 2s+1$}
Here we discuss the  tail of the series, for which we can directly use
 Proposition~\ref{t2j-s}; this contains bounds for $\tilde T_{2j}$ which can be 
obtained directly from similar $T_{2j}$ bounds.
Precisely,  Proposition~\ref{t2j-s};   shows that for large enough $j$ we 
have a favorable bound for $E_{s,2j}$, namely 
\[
|E_{s,2j}| \leq \| u\|_{H^s}^2 (c\| u\|_{l^2 DU^2})^{2j-2}, \qquad j > 2s+1
\]
 This suffices if $s < \frac12$. For larger $s$, however,
we need stronger bounds if $j$ is small. We remind the reader that 
the issue here is that the bounds for $\tilde T_{2j}$ in Proposition~\ref{t2j-s}
are derived from similar bounds for $T_{2j}$; on the other hand,
for smaller $s$ we expect such bounds to hold for 
$\tilde T_{2j}$ but not for $T_{2j}$. Thus, we need to take advantage of the fact that 
$\tilde T_{2j}$ has a better structure than $T_{2j}$ (namely the fact that, unlike $T_{2j}$,
$\tilde T_{2j}$ contains only connected integrals).

\subsubsection{Medium $j$, $j >  s+1$} Here we use the bounds for $\tilde T_{2j}$ in 
 Proposition~\ref{t2j-sprime}, which yield 
\[
|E_{s,2j}| \leq \| u\|_{H^s}^2 (c\| u\|_{l^2 DU^2})^{2j-2}, \qquad j > s+1
\]
We remark that, due to the previous step, for fixed $s$ we need this bound only 
for a finite set of $j$'s, therefore the uniformity of the bounds with respect to $j$ 
is no longer an issue.

\subsubsection{Small  $j$, $j \leq s+1$} In this range we face the additional difficulty 
that in the $2j$ component of the formula \eqref{def++} we need to also take into account
the energy corrections $T_{2j}^l$. This is addressed in Section~\ref{s:tt2j-exp}.
At least for $s$ away from $\Z/2$, we can directly use the bound \eqref{tt2j-err-int}
to conclude that we have the bound 
\begin{equation}\label{all-s}
|E_{s,2j}| \lesssim  \| u\|_{H^s}^2 \| u\|_{l^2 DU^2}^{2j-2}, \qquad j \geq 2
\end{equation}
The only remaining issue is to show that this bound is uniform as well for $s$ 
near integers and half-integers. These two situations are somewhat different, 
and are considered separately.

\subsubsection{$s$ near integers} To set the notations, we fix an
integer $s_0 \geq 1$ and consider $s$ near $s_0$.  On one hand, in the
bound \eqref{tt2j-err-int} we have $(s-s_0)^{-1}$ growth. On the other
hand, in the integral \eqref{def++} we have an additional $\sin (\pi
s)$ factor. These two expressions cancel and we obtain again the bound
\eqref{all-s}, uniformly for $s$ near $s_0$.

\subsubsection{$s$ near half- integers} Here we fix $s_0 \in \N+\frac12$ and again consider 
$s$ near $s_0$. The integral bound for $\tilde T_{2j}(z)$ fails to be uniform as $s$ approaches $s_0$.
Nevertheless, we are saved by the fact that only $\real   \tilde T_{2j}(z)$ is needed. But this satisfies
a favorable bound by Proposition~\ref{p:t2j-exps}.

\subsubsection{Dependence on $s$}
If $u \in H^{s_0}$ with $s_0 > -\frac12$ then, with $N= [s_0]$,  analyticity with respect to 
$s \in (-\frac12,s_0)$ is a consequence of the
integral formula 
\begin{equation} \label{energieformula}
 E_s(u) =    \frac{2\sin(\pi s)}\pi   \bigintsss_1^\infty (\tau^2-1)^s   \Big( \real  \ln T(i\tau/2)+  
\sum_{j=0}^{N} (-1)^j H_{2j}  \  \tau^{-2j-1} \Big) d\tau   +  \sum_{j=0}^N \binom{s}{j}  H_{2j}, 
\end{equation} 
It  allows for an extension of $E_s$  in  a complex neighborhood of $(-\frac12,s_0)$ 
to a holomorphic function. 

Continuity at $s_0$ from the left follows from the 
above analyticity property and the uniform Lipschitz bounds for $E_s(u)$, $s \leq s_0$, 
in terms of $u \in H^{s_0}$,   simply by approximating  $u$ with more regular functions.

Finally, we remark that if  $s$ is an integer then the coefficient of the integral vanishes and 
the sum gives the desired identification \eqref{identification}.

\subsubsection{The conservation of $E_s$}
Here we show that $E_s$ is conserved along the NLS and mKdV flow. It
suffices to do this for Schwartz initial datum, in which case $T$ is
constant along the flow. Solutions with Schwartz initial datum remain
Schwartz functions, and for those the formal analysis of the Jost
solutions becomes rigorous.

\subsection{Proof of Theorem \ref{t:main+}}

Here the goal is to extend the considerations in the previous proof to initial datum $u$ which is no longer
small in $l^2 DU^2$.  We will do this in two steps. First we show that the energies given by 
\eqref{def+}, \eqref{def++} are well defined for all $u$ in $H^s$, and have the appropriate regularity 
properties. Then we will prove that the trace formulas hold.

\subsubsection{ The energies $E_s$ for large data.}
Given $u \in H^s$, we can choose some large $\tau_0$ depending on $\|u\|_{H^s}$ so that 
\[
\| u\|_{l^2_{\tau_0} DU^2} \leq \delta 
\]
Rescaling the bounds for $\ln |T|$ in the proof of Theorem~\ref{energies}, it follows that the 
formal series for $\ln |T|$ converges on $i[\tau_0,\infty)$ with favorable bounds. 
This shows that the part  of the integrals \eqref{def+}, \eqref{def++} for $t \in [\tau_0,\infty)$ 
converges and satisfies the desired $H^s$ bounds. 

It remains to consider the part of the integrals in the interval $i[1,\tau_0]$. In the defocusing case 
we know that in this interval $\ln |T|$ depends analytically on $u \in l^2 DU^2$, which is more than enough
to close the argument.

However, in the focusing case we only know that $T^{-1}$ depends analytically on $u \in l^2 DU^2$.
Thus if $T$ has poles in the interval $i[1,\tau_0]$ then $\ln |T|$ has logarithmic singularities there.
This still suffices for the convergence of the integrals  \eqref{def+}, \eqref{def++} for $t \in [\tau_0,\infty)$,
but we need an additional argument to establish the regularity properties of the integral with respect to $u$.
Precisely, we need to account for the poles of $T$ which cross the interval $i[1,\tau_0]$. Changing the contour 
of integration to one without poles we get the sum of residues 
\[
\sum_k m_k \Xi(2z_k)
\]
over the poles near the line $i[1,\tau_0]$. 
Here $\Xi$ is as defined in \eqref{Xi-def}.
Since $\Xi$ is real analytic outside $z = i$, it follows that the 
contributions of the poles away from $i/2$ is analytic in $u$. At $z = i$ the function $\Xi$ is of class
$C^{s+1}$, so continuity  of $E_s$ follows for $s>-\frac12$.

\subsubsection{The trace formula for small data}
Here  we show that the middle terms of \eqref{defd+} and \eqref{deff+} are defined and coincide with the 
right hand side.  We first consider small datum $\| u\|_{l^2 DU^2} \ll 1$. 

\medskip

a) $s < 0$. In the defocusing case  $- \ln |T|$ is a nonnegative harmonic function in the upper half-space,
so the conclusion \eqref{defd+} follows directly from Lemma~\ref{l:apriori-s}. In the focusing case 
 $\ln |T|$ is a nonnegative superharmonic function in the upper half-space. Further,  denoting  the poles 
of $T$ in the upper half-space by $z_k$ and their corresponding multiplicities by $m_k$, we have 
\[
- \Delta \ln |T(z/2)| = \sum_k m_k \delta_{2z_k}
\]
By Lemma~\ref{piece2} and Lemma \ref{lpem}, these are located in $\{ 0
<  \im z \ll (1+ |\real z|)^s \}$.  Hence the trace formula
\eqref{deff+} is again a consequence of Lemma~\ref{l:apriori-s}.

\medskip

b) $s = N \geq 0$, integer. In the defocusing case we can directly apply  Lemma~\ref{l:apriori-n}. In the focusing 
case the condition $u \in L^2$ guarantees that all poles of $T$ are contained in a strip along the real axis.
Then we can apply  Lemma~\ref{l:apriori-n}.

\medskip

c) $N < s < N+1$, noninteger. Then we apply first step (b) above for $s=N$. This places us in a context where 
we can use Lemma~\ref{l:apriori-s}.

\subsubsection{The trace formula for large data}

Here we take arbitrary $u \in H^s$, and we first claim that we still have 
the bounds
\begin{equation}\label{trace-reg}
\int_\R (1+\xi^2)^s d \mu + \int_{U} \im z (1+|z|^2)^s d\nu  < \infty
\end{equation}
To see that we first observe that for large enough $\tau$ we have
\[
\| u\|_{l^2_{\tau_0}  DU^2} \ll 1
\]
Hence the rescaled function 
\[
u_\tau(x) = \tau_0 u(\tau_0 x) 
\]
has small $l^2  DU^2$ norm, and then we can apply the small data trace regularity.
This proves our claim.

Next we consider the trace formulas for $u$. In view of the trace regularity property
\eqref{trace-reg}, this follows directly as in Lemma~\ref{l:apriori-s}.

\red{To prove part 3 of the theorem we observe that the square of the $H^s$ norm is the quadratic part of $E_s(u)$. Estimate \eqref{est:quad}  follows from  Corollary \ref{cor:6.3}  and inequality \eqref{t2j-sprime-int} in  Proposition \ref{t2j-sprime}.}

\subsubsection{ The global regularity of the energies $E_s$}
In the defocusing case, for arbitrary $u \in H^s$ we consider a large frequency scale 
$\tau$ as above. For $\tau \gg \tau_0$ we can take advantage of the rescaling to conclude 
summability of the asymptotic expansion for $\log |T|$ with good $H^s$ bounds, and analyticity follows.
For $1 \leq \tau \lesssim \tau_0$ we instead use directly the analyticity of $\log |T|$
as a function of $u \in l^2 DU^2$.

In the focusing case we argue in the same manner provided that there are no poles for $T$ in $i[1/2,2\tau_0)$.
Further, the trace formulas show that the poles away from $i$ are in effect harmless, as the function $\Xi_s$ is analytic there. Only poles of $T$ at $i/2$ affect the regularity of $E_s$ as a function of $u$ since we may move the contour of integration away from the endpoint without changing $E_s$.

\bigskip

\subsection{More conserved quantities}

We denote 
\[  \C_+ = \Big\{ z \in \C, \im z > 0 , z \notin i[1,\infty)  \Big\}. \] 

\begin{definition} 
For $ s > -\frac12$ we define the class of weight functions
$\mathcal{E}_{s}$ 
which are holomorphic in $\C_+$,  continuous up to the boundary 
(with possible different limits from the left and the right), real 
on the real axis and which satisfy 
\[  |\mu(z)| \le c (1+|z|^2)^{s}. \]
\end{definition} 
 
By the Schwarz reflection principle they are holomorphic functions in a domain 
symmetric around the real axis. 
  We define $\Xi$ by 
\[ \frac{d}{dz} \Xi = \mu, \quad \Xi(0)=0 \] 
for $\mu \in \mathcal{E}_s$. Let 
\[ [\mu] (\tau) =   \mu(i\tau +0) - \mu(i\tau-0) \] 
for $t >1$. If $\mu$ is even on $\R$ then $\mu(-\bar z) = \overline{\mu(z)} $ 
and if $\mu$ is odd then $\mu(-\bar z) = -\overline{\mu(z)}$. Thus 
$[\mu] $ is purely imaginary if $\mu $ is even and real if $\mu$ is odd 
on $\R$.

\begin{theorem}\label{generaleta}  There exists $\delta>0$ so that the following is true: 
If $s>-\frac12$, $ s\notin \Z/2$,  
$\xi \in \mathcal{E}_s $ then there is a unique 
conserved quantity $E_{\mu}$ on 
\[ \{ u \in H^s, \Vert u \Vert_{l^2_1 DU^2} \le \delta \} \]
 which satisfies 
\[ \left| E_{\mu}(u) - \int \mu(\xi)|\hat u(\xi)|^2 d\xi \right| 
\le c \Vert u \Vert_{H^s}^2 \Vert u \Vert_{l^2_1 DU^2}^2 \]
In particular the quadratic part is given by the Fourier multiplier $\mu$.
Moreover in the defocusing case $E_\mu$ is defined on whole of $H^s$ and 
\begin{equation} \label{Exid} 
\begin{split} 
E_\mu(u) = &\, -\frac1\pi \int \mu(\xi) \ln |T(-\xi/2)| d\xi  
\\   
= &\, \frac1\pi \real \int_1^\infty \overline{[\mu](\tau)}   \Big( \ln T(i\tau/2) 
+   \sum_{j=0}^{N-1}  i^j   H_{j}  \tau^{-j-1}\Big)   dt 
+ \sum_{j=0}^{N-1}   \frac1{j!}\overline{ \mu^{(j)}(0)}  H_{j} 
\end{split} 
\end{equation} 
and in the focusing case 
\begin{equation} \label{Exidf} 
\begin{split} 
E_\xi(u) = & \,\frac1\pi \int \mu(\xi) \ln |T(-\xi/2)| d\xi  
+ \sum_k m_k \im   \Xi(z_k) 
\\   
= &\, - \real \int_1^\infty \overline{[\mu](\tau)}   \Big( \ln(T(i\tau/2)) -  \sum_{j=0}^{N-1}  i^j  H_{j} \tau^{-j-1}\Big)  d\tau 
+ \sum_{j=0}^{N-1}    \frac1{j!} \overline{\mu^{(j)}(0)} H_{j} 
\end{split} 
\end{equation} 
where $N = [2s]$, and the sum runs over the eigenvalues $z_k/2$ with $\im z_k<0$, with  algebraic multiplicities $m_j$. 
It is continuous   in $\mu$ and uniformly analytic in $u$.
If $u \in H^{s_0}$ for some $ s_0 >s$ then the map 
\[ \mathcal{E}_s \ni \mu \to  E_{\mu}(u) \]
is analytic in $\mu$.  
The quartic term is given by 
\[ \frac1{2 (2\pi)^{\frac32}}   \int_{\xi_1+\xi_2= \eta_1+\eta_2}  \frac{\mu(\xi_1) + \mu(\xi_2) -\mu(\eta_1) -\mu(\eta_2)}{(\xi_1-\eta_1)(\xi_1-\eta_2)} \hat u(\xi_1) \hat u (\xi_1) \overline{\hat u(\eta_1) \hat  u(\eta_2)}. \]   
\end{theorem}

\begin{proof} Here we deform the contour to $i (-\infty, -1]$ and use the relation $\mu(\bar z ) = \overline{\mu(z)} $. 

In the defocusing case the middle term of \eqref{Exid}
  is well-defined for $ u\in H^s(\R)$ if $s>-\frac12 $ by trace
  formula of Theorem \ref{t:main+}. In the focusing case smallness
  ensures that there is no eigenvalue on $-i[1/2,\infty)$ and hence the
  evaluation of $\Xi$ is uniquely defined. Again the middle integral is
  well defined by \eqref{e:sumbound} after Theorem \ref{t:main+}.  If
  $s$ is not a half integer, and if we consider either the defocusing
  case or the focusing case under the smallness assumption then the
  integrals on the right hand side of \eqref{Exid} and \eqref{Exidf}
  are well defined. The arguments of Section \ref{s:t2} combined with dominated 
convergence imply  the second equality in  \eqref{Exid} and \eqref{Exidf}.
Analyticity with respect to $\mu$ is obvious for the term in the middle, 
whereas analyticity with respect to $u$ is obvious for the term on the right hand side of \eqref{Exid} resp. \eqref{Exidf}.  The identification of the quadratic term follows as for $E_s$ in Section \ref{s:t2}. The quartic term is identified as in Proposition \ref{prop:quartic}.  
\end{proof}

\subsection{Frequency envelopes}

One consequence of our main theorem is that if the initial datum is
small in $H^s$ then the solution remains in $H^s$ at all times, with a
comparable uniform bound\footnote{We also recall that by rescaling a
  similar result holds for large data, but there the uniform bound is
  no longer universal, instead it will depend also on the data size.}.
A natural follow-up question is whether the $H^s$ energy can migrate 
arbitrarily between frequencies. To provide some insight into this
it is convenient to use weighted norms which are closely related to frequency envelopes. 
Precisely,  we  consider sequences $\{a_k\}_{k \geq 0}$ for which
 the following properties hold:

\begin{enumerate} 
\item  The following bound from above holds: $a_k \lesssim 1$. 
\item  The sequence $a_k$ is slowly varying, $a_k / a_j \leq 2^{c|j-k|} $
Here $c$ is a sufficiently small constant, whose choice may depend on the problem at hand.
\end{enumerate}

For such sequences  we define the weighted norm 
\[ 
\Vert u \Vert_{H^a}^2 = \sum a_j^2 2^{2js} \Vert u_{2^j} \Vert_{L^2}^2. 
\]  
Then one way of saying that the $H^s$ energy of a solution $u$ does not travel much between  
frequencies is if not only the $H^s$ norm stays uniformly bounded, but  for
any sequence $\{a_k\}$ as above, $\Vert u(t) \Vert_{H^a}$ is uniformly bounded 
in terms of $\Vert u(0) \Vert_{H^a}$.  This is not exactly the standard frequency envelope 
formulation, but the two are equivalent.

A small variation of the proof of Theorem \ref{generaleta} shows indeed that this is the case:

\begin{theorem} 
For all $s>0$ there exists $\varepsilon>0$ so that the following is true. 
Let $(a_j)$ be a sequence satisfying 
\begin{enumerate} 
\item  $ 2^{(\varepsilon -\frac12) j} \le a_j  \le     2^s $ 
\item If $j_1< j_2<j_3$ then 
\[   \frac{ \ln a_{j_2}-\ln a_{j_1} }{j_2-j_1}  \le \frac{\ln a_{j_3}-\ln a_{j_2}}{j_3-j_2}  + \varepsilon .\] 
\end{enumerate} 
Then there exists  an even $\mu \in \mathcal{E}_s$ so that 
 \[  \sum_{j=1}^\infty  a_j \Vert u_{2^j} \Vert^2_{L^2} \approx  \int  \mu(\xi)  |\hat u|^2 d\xi   \] 
and 
\begin{equation} \label{quarticerror}   |E_\mu(u) - \int \mu(\xi)  |\hat u(\xi)|^2 | \le c \Vert u \Vert_{H^a}^2 \Vert u \Vert^2_{l^2_1 DU^2} \end{equation}  
provided $\Vert u \Vert_{l^2_1 DU^2}\le \delta$. 
In particular if  $s > -\frac12$, and $u_0$ small in $l^2_1 DU^2$ then 
\[ \Vert u(t) \Vert_{H^a} \lesssim  \Vert u_0 \Vert_{H^a}. \] 
\end{theorem}

\begin{proof} The first part implies the second one. By a superposition argument it suffices to prove the theorem only for a sequence 
\[  a_j = \left\{\begin{array}{ll}   2^{s_0 j} \qquad & \text{ for } j \le j_0\\
                              2^{s_1 j + (s_0-s_1)j_0}\quad  & \text{ for } j \ge j_0.      \end{array} \right.
\] 
with $-\frac12 < s_1 < s_0 < s_1+\varepsilon $.  Then  
\[
\| u\|_{H^s_{a}} \approx  \| u_{<2^{j_0}} \|_{H^{s_0}} +    2^{(s_0-s_1)j_0}   \| u_{> 2^{j_0}} \|_{H^{s_1}}
\]
We choose 
\[
\mu_{s_0,s_1,j_0} (\xi) =   (1+ \xi^2)^{s_0} (1+(2^{-j_0}\xi)^2)^{s_1-s_0}
\]
so that the quadratic part of the energy $E_\mu$ is given by this
Fourier multiplier by Theorem \ref{generaleta} and we have to prove the bound
\eqref{quarticerror}. This follows by an adaptation of the proof of Theorem \ref{energies} resp. an application of various estimates proven there. 

\begin{enumerate} 
\item If $-\frac12 <s_1<s_0 < 1$ then the estimate follows as in the proof of Proposition \ref{t2j-sprime}. The geometric series allow this range of exponents.
\item If $ j \le  s_1 < s_0< j+1/2$ or
$j+1/2 < s_1 <s_2 < j+1$ 
then the bound follows similarly from the integrated bounds \eqref{tt2j-err-int} resp. their  proof. This again deteriorates as we approach $j+\frac12$ but not at 
the other endpoints.  
\item  If $2s_0$ and $2s_1$ are close to $2j+1$ then we want to argue as for \eqref{tt2j-err-int+}. We check  that for $\tau\ge 1 $ 
 \[
[\mu]( \tau) \sim  \left\{ \begin{array}{ll} 
 -\dfrac{2 \sin \pi s_0}\pi (\tau^2-1)^{s_0} (1- \tau^2 2^{-2j_0})^{s_1-s_0} & \qquad \text{ for } 1\le \tau \le 2^{j_0} \\
 -\dfrac{2\sin \pi s_1 }{\pi} ( \tau^2-1)^{s_0}(\tau^2 2^{-2j_0} -1)^{s_1-s_0} & \qquad \text{ for } 
2^{j_0} \le \tau 
\end{array} \right.             
\] 
Again a direct adaption of the proof gives the result. 
\item If $s_1<  N \le s_0 $  we need a more careful correction. 
In this case we try to bound 
\[  \real  \int_1^\infty  [\mu](\tau) \Big( \ln T(i\tau /2)   +  
 \sum_{j=0}^{N-1} (-1)^j E_{2l}(u)  \tau^{-2j-1} +  (-1)^N E_{2N}(u_{<2^{-j_0}})  \tau^{-2N-1}
dt  +  \dots \]     
This follows in the same fashion as part of the proof of \eqref{tt2j-err-int+}.
\end{enumerate} 
\end{proof} 

We also remark on an interesting but more straightforward  case of  the above result. If we 
confine ourselves to negative Sobolev norms, then the above theorem holds
for any frequency envelope satisfying
\[
a_k \geq  a_{k+1} \geq 2^{\delta-\frac12} a_k
\]
The simplification here is that no new energies are needed; instead it suffices to work with 
$E_{-\frac12 +\delta}$ and its rescaled versions.

\subsection{Generalized momentae}

As noted earlier in the paper, the conserved quantities $E_s$ in our
main result in Theorem~\ref{energies} are positive definite in the focusing case, and in general the quadratic part is positive definite. They  can
be viewed as inhomogeneous extensions of the  even conserved
Hamiltonians $H_{2k}$.  In particular they are not connected at all with
the odd conserved energies, i.e. the generalized momenta.

A natural question would be whether one can similarly define 
a continuous family of momenta. This is indeed the case, as
one can replicate as follows the construction of the conserved
energies $E_s$. Starting from the observation that, at least in the 
defocusing case,  the odd 
Hamiltonians can be expressed in the form
\[
H_{2k+1} =\frac1\pi  \int_{\R} \xi^{2k+1} \ln |T(-\xi/2)| d \xi
\]
it is natural to seek to define the generalized momenta as 
\[
P_s =  \frac1\pi \int_{\R} \xi (1+\xi^2)^{s-\frac12} \ln |T(-\xi/2)| d\xi
= -\frac1\pi \int \xi (1+\xi^2)^{s-\frac12} \ln |T(\xi/2)| 
, \qquad s > -\frac12,
\]
so that if $s$ is a half-integer we recover the (linear combinations of) odd Hamiltonians, i.e. the 
momenta.

To render this definition useful for non-decaying data, as well as in the focusing case,
we switch this as in the definition of the energies $E_s$ to a contour integral over the 
double half-line $i[1,\infty)$. For small $s$ this is done directly,  
\begin{equation}
P_s =  \frac{2\cos(\pi s)}{\pi}    \int_{1}^\infty  \tau(\tau^2-1)^{s-\frac12}
\im  \ln T(i\tau/2) \ dt, \qquad -\frac12 < s < \frac12  
\end{equation} 
For larger $s$ we need to remove a number of terms in the formal expansion of $T(it)$
 in order for the above integral to converge. Precisely, for $s$ in the range
\[
 N+\frac12 \leq s < N+\frac32 , \qquad N \geq 0
\]
we set in the defocusing case
\begin{equation}
P_s = \frac{2\cos(\pi s)}{\pi}    \int_{1}^\infty\! \tau(\tau^2-1)^{s-\frac12}\left(\!\im  \ln T(i\tau/2)  + \!\sum_{j = 0}^N (-1)^j  
H_{2j+1} \tau^{-2j-2} \! \right) d\tau -   \sum_{j = 0}^N \binom{s-\frac12}{j} H_{2j+1} ,
\end{equation}  
and in the focusing case
\begin{equation}
P_s = \frac{2\cos(\pi s)}{\pi}    \int_{1}^\infty\! \tau(\tau^2-1)^{s-\frac12}\left(\!-\im  \ln T(i\tau/2)  + \!\sum_{j = 0}^N (-1)^j  
H_{2j+1} \tau^{-2j-2} \! \right) d\tau -   \sum_{j = 0}^N \binom{s-\frac12}{j} H_{2j+1} .
\end{equation}

We obtain the following result, which for $ s \in (-\frac12 , \infty) \backslash (\Z/2)$  is
a special case of Theorem \ref{generaleta}:

\begin{theorem}\label{t:momentae} 
For each $s > -\frac12$ and $ \delta>0$  and both for the focusing and defocusing case 
 the functional 
\[
P_s : \{ u \in H^s |  \Vert u \Vert_{H^{-\frac12+\delta}} \ll  1 \}   \to \R^+
\]
is conserved along the NLS and mKdV flow.  Further, we have 
\begin{equation}
\left| P_s(u) - \langle u, i \partial_x u \rangle_{H^{s-\frac12}}^2 \right| \lesssim  \Vert u \Vert_{l^2_1 DU^2}^2 
\Vert u \Vert_{H^s}^2. 
\end{equation}
\end{theorem}

The case $s \in \Z/2$ follows as for Theorem \ref{t:main+}, but with 
the role of integers and half integers reversed.

\newsection{The terms of homogeneity $4$ and $6$ in $u$}
\label{s:further}
In this section we will derive a simple expression for $T_4$, $T_6$, 
$E_{s,4}$ and $E_{k,6}$ for integers $k$.

\subsection{The quartic term in $-\ln T(z)$ and $E_s(u)$} 

Here we discuss in more detail the expression $\tilde T_4(z)$, as well
as the corresponding quartic term in a homogeneous expression
$E_s(u)$.  By the same techniques one can obtain similar results for
$\tilde T_{2j}(z)$.
We begin by computing the symbol for $\tilde T_{4}(z)$ viewed as a
translation invariant quadrilinear form.  We recall that $\tilde
T_4(z)$ is given by (using Fourier inversion in the second step)
\[
\begin{split} 
\tilde T_4(z) = & \ - 2 \int_{x_1<x_2 < y_1 < y_2} e^{2iz(y_1+y_2 - x_1-x_2)}u(y_1) u(y_2) \overline{ u(x_1) u(x_2)}  
dx dy
\\ = & \ -\frac1{2 \pi^2}\int_{\{x_1<x_2<y_1<y_2\} \times \R^4} e^{2iz(y_1+y_2-x_1-x_2)- i (x_1\xi_1+x_2\xi_2 - y_1 \eta_1 -y_2 \eta_2)}\overline{ \hat u(\xi_1) \hat u(\xi_2)} 
 \hat u(\eta_1) \hat u(\eta_2) 
\\ & \  \qquad \times dx_1 dx_2\, dy_1dy_2\,  d\xi_1 d\xi_2\,  d\eta_1 d\eta_2. 
\end{split}   
\]
To compute the kernel we first symmetrize with respect to the  $x$ variables 
and integrate exponentials successively:
\[
\begin{split}  
K_z(\xi,\eta ) = & \ -\frac1{(2\pi)^2} \int_{x_1,x_2 <y_1<y_2} 
e^{2iz(y_1+y_2-x_1-x_2)- i (x_1\xi_1+x_2\xi_2 - y_1 \eta_1 -y_2 \eta_2)} dx_1dx_2 dy_1 dy_2 
\\= & \, \frac{i}{(2\pi)^2 } \frac{1}{2z+\xi_1}\frac{1}{2z+\xi_2}\frac1{2z+\eta_1}   
 \int  e^{- i y_1(\xi_1+\xi_2-\eta_1 -\eta_2))} dy_1               
\\ =& \,   \frac{i}{2\pi  } \frac1{2z+\xi_1} \frac1{2z+\xi_2} \frac{1}{2z+\eta_2}  \delta_{\xi_1+\xi_2-\eta_1-\eta_2} .
\end{split} 
\] 
Replacing the above kernel with the symmetrization with respect to $\eta$ we obtain 
\begin{equation}\label{Kz}   
\tilde T_4(z) = \frac{i}{4\pi} 
\int_{\R^3} 
\frac{4z+\xi_1+\xi_1}{(2z+\xi_1)(2z+\xi_2)(2z+\eta_1)(2z+\eta_2)}  
\overline{\hat u(\eta_1+\eta_2-\xi_1) \hat u(\xi_1)} \hat u(\eta_1) \hat u(\eta_2)  d\xi_1 d\eta_1d\eta_2.
\end{equation} 
Of course we can do the same calculation with the roles of $\xi$ and
$\eta$ interchanged. This allows to replace the product of the Fourier
transforms by the real part of the Fourier transforms.  By an abuse of
notation we write the last integral as
$\int_{\xi_1+\xi_2=\eta_1+\eta_2} $.  Undoing the symmetrization we
arrive at
\begin{equation}\label{t4sa}   \tilde T_4(z) = \frac{i}{2\pi} 
\int_{\xi_1+\xi_2=\eta_1+\eta_2} 
\frac{1}{(2z+\xi_1)(2z+\eta_1)(2z+\eta_2)}\real \{ \overline{\hat u(\xi_1) \hat u(\xi_2)} \hat u(\eta_1) \hat u(\eta_2) \} 
\end{equation}  
From \eqref{Kz}  we deduce the Laurent expansion of the quartic term at infinity.
\begin{lemma}\label{lHj4} 
Suppose that $ u $ is a Schwartz function. Then we obtain the asymptotic series
\begin{equation}  \tilde T_4(z) \sim i   \sum_{j=2}^\infty H_{j,4}(2z)^{-j-1}  
\end{equation} 
where 
\begin{equation} \label{Hj4} 
H_{j,4}=- \real  \left[ i^j    \sum_{\alpha_1+\alpha_2+\alpha_3=j-2} (-1)^{\alpha_1}    \int  u^{(\alpha_2)}  u^{(\alpha_3)}  \overline{ u^{(\alpha_1)}  u } dx\right]. \end{equation}   
\end{lemma}

\begin{proof} We expand \eqref{t4sa} in negative powers of $2z$. Then 
\[
\begin{split} 
 H_{j,4}  = &  (2\pi)^{-1}  \real   \sum_{\alpha_1+\alpha_2+\alpha_3=j-2}(-1)^j   \int_{\xi_1+\xi_2=\eta_1+\eta_2}  \xi_1^{\alpha_1} \eta_1^{\alpha_2} 
\eta_2^{\alpha_3} \overline{ \hat u(\xi_1)\hat u (\xi_2)}\hat u(\eta_1)\hat u(\eta_2)\\ = & 
(2\pi)^{-1} \real \Big[  i^{j-2} \sum_{\alpha_1+\alpha_2+\alpha_3=j-2} (-1)^{\alpha_1}    \widehat{u^{(\alpha_2)}}* \widehat{u^{(\alpha_3)}}* \widehat{\overline{u^{(\alpha_1)}}} * \widehat{\bar u}(0) \Big]
\\ = & -(2\pi)^{\frac12} \real \Big[ i^{j-2} \sum_{\alpha_1+\alpha_2+\alpha_3=j-2} (-1)^{\alpha_1} \mathcal{F}(u^{(\alpha_2)}   u^{(\alpha_3)}\overline{ u^{(\alpha_1)} u})(0) \Big]
\\
 = &    \real \Big[  i^{j-2}  \sum_{\alpha_1+\alpha_2+\alpha_3=j-2 }  (-1)^{\alpha_1}   \int  u^{(\alpha_2)}   u^{(\alpha_3)}\overline{ u^{(\alpha_1)} u}  dx \Big] . 
\end{split} 
\]    
\end{proof} 

Next  we turn our attention to the quartic term in the energies $E_s$. 
The 4-linear component of $E_s^4$ is given by the contour integral
\eqref{def+} or \eqref{def++} over the half-line $i[1,\infty)$ with
$\ln T$ replaced by $\tilde T_4$. However we can change the contour of integration back to the real line to obtain the representation that corresponds to \eqref{def} from \eqref{Kz} and 
\[
E_{s,4}  =    \frac1\pi\int_{\xi_1+\xi_2=\eta_1+\eta_2} K_s^4(\xi,\eta)
 \real (\overline{ u(\xi_1)u(\xi_2) }\hat u(\eta_1) \hat u(\eta_2)) d\xi    
\]
where
\[
\begin{split}
K_{s,4}(\xi,\eta) = &   \frac1{2\pi^{2}}    \int_{\R+i0}  (1+z^2)^{s}
 \im   \left( \frac{2z+\xi_1+\xi_2} {(z+\xi_1)  (z+\xi_2)(z+\eta_1)(z+\eta_2)}   \right) dz 
\\ = & (2\pi)^{-1} 
\int_{\R}  (1+\xi^2)^{s}
 \frac{1}{(\xi_1-\eta_1)(\xi_1-\eta_2)}
 \left(\delta_{\xi_1} + \delta_{\xi_2} - \delta_{\eta_1} - \delta_{\eta_2}\right) d\xi.   
\end{split}
\]
Thus we immediately obtain the following:
\begin{proposition} \label{prop:quartic} 
The quartic part of $E_s$ is given by
\[ E_{s,4} =  \frac1{2\pi}   \int_{\xi_1+\xi_2=\eta_1+\eta_2}  
 \frac{(1+\xi_1^2)^{s}+(1+\xi_2^2)^{s} - (1+\eta_1^2)^{s}
- (1+\eta_2^2)^{s}}{(\xi_1-\eta_1)(\xi_1-\eta_2)} \overline{ \hat u (\xi_1) \hat u(\xi_2)}  \hat u(\eta_1) \hat u (\eta_2) d\xi_1 d\eta                \]         \end{proposition}

We observe that this coincides with $H_{k,4}$ for $s=k \in
\mathbb{N}$.  Of course, this is no surprise, as this expression is
exactly the quartic I-method energy correction term, see
\cite{MR2995102}, \cite{MR2353092} and \cite{MR2376575}.

\subsection{The term  $\tilde T_{6} (z)$ } 
For completeness we also provide a brief computation for the expansion of the $\tilde T_6$
term, which gives the sixth-linear terms in the conserved energies.

\begin{lemma} \label{lHj6} 
The following identity holds:
\begin{equation}\label{T6} 
\begin{split} 
\tilde T_6(z) = & \frac{i}{(2\pi)^{2}}    \int_{\xi_1+\xi_2+\xi_3= \eta_1+\eta_2+\eta_3} 
 \frac1{2z+\xi_1} \frac1{2z+\xi_2} \frac1{2z+\eta_2} 
\frac1{2z+\eta_3} \left(\frac1{2z+\eta_1} +  \frac1{2z+(\eta_2+\eta_3-\xi_3)} \right) 
\\ & \qquad \times \hat u(\xi_1) \hat u(\xi_2) \hat  u(\xi_3) \overline{\hat u(\eta_1) \hat u(\eta_2) \hat u(\eta_3)} d\xi d\eta 
\sim  i   \sum_{j=4}^\infty H_{j,6} (2z)^{-j-1}     
\end{split} 
\end{equation} 
where 
\begin{equation} \label{T6exp} 
H_{j,6} =  \real \Big[ i^j \sum_{|\alpha|=j-4} 
(-1)^{\alpha_1+\alpha_2} \int   u^{(\alpha_1)}  u^{(\alpha_2)}   u \overline{ 
 u^{(\alpha_3)}   u^{(\alpha_4)}     u^{(\alpha_5)} }
 +   u^{(\alpha_1)}   u^{(\alpha_2)}\bar u  
 \frac{d^{\alpha_3}}{dx^{\alpha_3}}\Big( u   \overline{   u^{(\alpha_4)} u^{(\alpha_5)}} \Big)  
dx\Big] .  
\end{equation} 
\end{lemma} 
 
\begin{proof} 
By Proposition  \ref{a:log} below  we have $\tilde T_6(z) =   12 \<XXXYYY> + 4 \<XXYXYY> $.  
As above we compute (with $v=\bar u$) 
\[
\begin{split} 
12  \<XXXYYY>  = & \  12 \int_{x_1<x_2<x_3<y_1<y_2<y_2} 
\prod_{j=1}^3 e^{2iz(y_j-x_j)} u(y_j) \overline{u(x_j)} dx dy                            
\\ = & \int_{x_1,x_2<x_3 < y_1,y_2,y_3} e^{2iz(y_j-x_j)} u(y_j) \overline{u(x_j)} dx dy 
\\ = & \ \frac{1}{(2\pi)^3} 
        \int   \int_{x_1,x_2<x_3<y_1,y_2,y_2} e^{2iz(y_1+y_2+y_3-x_1-x_2-x_3) - i \sum\limits_{j=1}^3 (\xi_j x_j-\eta_j y_j)}  \prod_{j=1}^3 \hat u(\eta_j) \overline{\hat  u(\xi_j)} dx dy d\xi d\eta 
\\ = & \frac{i}{(2\pi)^3} \int   \int 
\frac{ e^{-ix_3(\xi_1+\xi_2+\xi_3-\eta_1-\eta_2-\eta_3)}  }{(2z+\xi_1)(2z+\xi_2)(2z+\eta_1)(2z+\eta_2)(2z+\eta_3)} dx_3
\prod_{j=1}^3 \hat u(\eta_j) \overline{ u(\xi_j)}  d\xi d\eta
\\   = & \frac{i}{(2\pi)^2}\!\!\!   \int\limits_{\xi_1+\xi_2+\xi_3=\eta_1+\eta_2+\eta_3}\!\!\!\!\!  \frac1{(2z+\xi_1)(2z+\xi_2)(2z+\eta_1)(2z+\eta_2)(2z+\eta_3)}  
\prod_{j=1}^3 \hat u(\eta_j) \overline{ u(\xi_j)} 
 d\xi_1d\xi_2  d\eta .
\end{split} 
\] 
The next calculation is similar so we only point out the
differences. In the first step we only symmetrize with respect to
$\xi_1$ and $\xi_2$, and with respect to $\eta_1$ and $\eta_2$. This
leads to a factor $\frac14$ and an inner integral
\[
\int_{y_1<x_3} e^{2iz(x_3-y_1)+i (x_3(\xi_3-\eta_2-\eta_3)- y_1(\eta_1-\xi_1-\xi_2))}  dy_1dx_3 = \frac{2\pi i}{2z + (\eta_2+\eta_3-\xi_3)} \delta_{\xi_1+\xi_2+\xi_3-\eta_1-\eta_2-\eta_3} .  
 \]                                 
This gives 
\[ 4\<XXYXYY> =\frac{i}{(2\pi)^{2}  } \int_{\R^5}  \frac{ \hat u(\xi_1) \hat u(\xi_2) \hat  u(\xi_3) \overline{\hat u(\eta_1) \hat u(\eta_2) \hat u(\eta_3)} }{(2z-\xi_1)(2z-\xi_2)(2z-\eta_2) 
(2z-\eta_3)(2z-(\eta_2+\eta_3-\xi_3))}d\xi_1d\xi_2 d\eta. 
\] 
The claimed formula for $H_{j,6}$ follows by an expansion in inverse powers of $2z$, representing the integral as an integral over five convolutions on the Fourier side, and Fourier inversion.  
\end{proof} 

\newsection{The KdV equation}
\label{s:kdv} 

The arguments in the proof of our  results for NLS and mKdV carry over easily to  
the KdV equation. In effect the algebraic part of the analysis is virtually identical.
For this reason we outline here the corresponding results for the KdV equation
\begin{equation}
\label{kdv}
u_t + u_{xxx} - 6 uu_x   = 0
\end{equation}
where we consider real solutions on the real line. The scaling for this equation is 
\[
u(x,t) \to \lambda^2 u(\lambda x, \lambda^3 t)
\]
which corresponds to the critical Sobolev space $\dot H^{-\frac32}$.

This is also a completely integrable flow, and admits an infinite number of conservation laws,
of which the first several are as follows: 
\[
\begin{split}
E_0 = & \   \int u^2 dx
\\
E_1 = & \  \int  u_x^2 +  2 u^3 dx
\\
E_2 = & \  \int  u_{xx}^2 + 10 u u_x^2 + 5 u^4 dx
\end{split}
\]
In a Hamiltonian interpretation, these energies generate commuting Hamiltonian flows with the 
Poisson structure  defined by    
\[
\omega(u,v) = \int u v_x dx.
\]
The first of these flows is the group of translations, the second is the KdV, etc. As in the NLS/mKdV case we first state a simplified version of the main result,
which does not require knowledge of the scattering transform. Later we will return 
with a more complete version, see Theorem~\ref{energies-kdv+}, together with 
the appropriate trace formulas \red{in Proposition}~\ref{trace-kdv}. We have

\begin{theorem}\label{energies-kdv} 
  There exists $\delta > 0$ so that for each $s \geq -1 $  there exists an energy functional 
\[
E_s : H^s\cap \{ \| u \|_{H^{-1}} \leq \delta \}    \to \R^+
\]
with the following properties:
\begin{enumerate}
\item $E_s$ is conserved along the KdV flow.
\item  We have 
\begin{equation}\label{quadratic} 
\left| E_s(u) - \|u\|_{H^s}^2 \right| \lesssim  \Vert u \Vert_{H^{-1}} 
 \Vert u \Vert_{H^{\max\{-1,s-\frac14+\varepsilon\}}}^2. 
\end{equation}
\item $E_s$ is continuous in $s$ and $u \in H^\sigma$ for 
$-1\le  s \le \sigma$, analytic in $u$,  and analytic in both variables if in addition $s < \sigma$. 
  \end{enumerate}
\end{theorem}
With a little more work it is possible to choose $\varepsilon=0$ and we will indicate how this is done below. 
As in the case of NLS and mKdV, we establish the energy conservation
result for regular initial data. By the local well-posedness theory,
this extends to all $H^s$ data above the (current) Sobolev local
well-posedness threshold, which is $s \geq -\frac34$. If $s$ is below
this threshold, then the energy conservation property holds for all
data at the threshold.  It is not known whether the KdV equation is
well-posed in the range $-1 \leq s < -\frac34$ ; however, it is known
that local uniformly continuous dependence fails in this range, see
\cite{MR2376575}, and also that the problem is ill-posed in $H^s$ for
$s < -1$, see \cite{MR2830706}.

Again, one consequence of our result is that if the initial datum is in
$H^s$ then the solutions remain bounded in $H^s$ globally in
time. This is immediate if $\|u_0\|_{H^{-1}} \ll 1$, but also follows
by scaling for larger data. This has been proven  by making use of the Miura map 
by \cite{MR3400442} for $s=-1$ and independently by making clever use of Fredholm
determinants  by \cite{KVZ} for $-1\le s \le 1$. For the remainder of this section we outline the proof of this result,
and connect it to the proof of our NLS and KdV results.

\subsection{The scattering transform for KdV}

Here we recall some basic facts about the inverse scattering transform for 
KdV.  The Lax pair for KdV is given by the pair of operators $(\mathcal L,\mathcal P)$ defined 
as follows
\[
\mathcal L = - \partial_x^2 + u, \qquad \mathcal P = -4 \partial_x^3 + 3(u \partial_x + \partial_x u)
\]

The scattering transform associated to KdV is defined via the spectral problem 
for the operator $\mathcal L$, namely 
\begin{equation}
\mathcal L \psi = z^2 \psi
\end{equation}
As before, we use this equation first for  $z$ on the real line, and
then for $z$ in the upper half-space. As $u$ is real, it follows that
the operator $\mathcal L$ is self-adjoint.  Its continuous spectrum is the
positive real line $\R^+$ and it may have isolated (but possibly
infinite) negative eigenvalues, as well as a resonance at frequency
$0$, but no eigenvalues inside the continuous spectrum if $u$ is a
Schwartz potential.

To relate this to the corresponding NLS spectral problem   we also rewrite it as a linear
system for $(\psi_1,\psi_2) = (\psi,\psi_x+iz \psi)$, namely
\begin{equation}\label{scatter-kdv}
\left\{ 
\begin{aligned}
\frac{d\psi_1}{dx} =& -i z \psi_1 +  \psi_2 
\\
\frac{d\psi_2}{dx} = & i z \psi_2 +  u \psi_1
\end{aligned}
\right.
\end{equation}

The scattering data for this problem is
obtained for $z=\xi\in \R$ by considering the relation between the
asymptotics for $\psi$ at $\pm \infty$.  Precisely, one considers the
Jost solutions $\psi_l$ and $\psi_r$ with asymptotics
\begin{eqnarray*} 
\psi_l (\xi,x,t) &= & \left( \begin{array}{c}    e^{-i\xi x}  \cr 
0 \end{array}\right) + o(1) \qquad  \text{ as $x \to -\infty$},
\\  \psi_l (\xi,x,t) & =& \left( \begin{array}{c}   T^{-1}(\xi)  e^{-i\xi x} - i \xi^{-1}R(\xi) T^{-1}(\xi)  e^{i\xi x} 
    \\[1mm] 
R(\xi) T^{-1}(\xi)  e^{i\xi x} \end{array}\right) + o(1) \qquad  \text{ as }x \to \infty,
\end{eqnarray*} 
respectively
\begin{eqnarray*} 
\psi_r (\xi,x,t)&  = & \left( \begin{array}{c}  L(\xi)  T^{-1}(\xi)  e^{-i\xi x} -i \xi^{-1} T^{-1}(\xi)  e^{i\xi x} \\[1mm]
T^{-1}(\xi)  e^{i\xi x} \end{array}\right) + o(1) \ \ \ \text{ as $x \to -\infty$},
\\  \psi_r (\xi,x,t) & =& \left( \begin{array}{c}   -i\xi^{-1} e^{i\xi x}   \cr 
e^{i\xi x} \end{array}\right) + o(1) \ \ \ \text{ as $x \to \infty$}.
\end{eqnarray*}  
These are viewed as initial value problems with datum at $-\infty$,
respectively $+\infty$.  Since the Wronskian $ \det ( \psi_l,\psi_r)$
is constant we can evaluate it at both sides and see that both the
transmission coefficient and the reflection coefficients are the same
on both sides.

Let $\phi(\xi) $ and $\phi(-\xi)$ denote the first component of $\psi_l(\pm \xi ) $. They both satisfy 
\[   -\phi'' + u \phi= \xi^2 \phi, \] 
and hence their Wronskian 
\[ \det \left( \begin{matrix} \phi(\xi)  & \phi'(-\xi)  \\ \phi(\xi) & \phi'(-\xi)  \end{matrix} \right) \] 
is constant. We evaluate it on both sides. At $-\infty$ we  obtain $2i\xi$ 
and at $+\infty$ 
\[ \det \left( 
\begin{matrix} 
a_{+}  e^{-i\xi x} + b_+ e^{i\xi x} 
&  a_- e^{i\xi x} +b_- e^{-i\xi x} 
\\
-i\xi a_+ e^{-i\xi x} + i\xi b_+ e^{i\xi x} &  
i\xi a_- e^{i\xi x} -i\xi b_- e^{-i\xi x} 
\end{matrix} \right) = 2i\xi (a_+ a_- -b_+b_-) 
\]  
and hence
\[   T(\xi) T(-\xi)=1- R(\xi) R(-\xi).  \] 
Since  
\[ \psi_l(-\xi,x,t) = \overline{\psi_l(\xi,x,t)}, \] 
and hence  $T(-\xi) =\overline{T(\xi)} $ and  $R(-\xi)= \overline{R(\xi)}$,
we arrive at 
\begin{equation} |T(\xi)|^2 +|R(\xi)|^2=1.  \end{equation}

More generally for any $z$ in the closed upper half plane there exist the Jost 
solutions 
\[
\begin{split} 
\psi_l (\xi,x,t) = & \left( \begin{array}{c}    e^{-iz x}  \cr 
0 \end{array}\right)+ o(1) e^{\im z x}  \ \ \ \text{ as $x \to -\infty$},
\\ 
 \psi_l (\xi,x,t) = &  \left( \begin{array}{c}   T^{-1}(z)  e^{-iz x}  \cr 
0 \end{array}\right) + o(1)e^{\im zx}  \ \ \ \text{ as $x \to \infty$}.
\end{split} 
\] 
This provides a holomorphic extension of $T^{-1}$ to the upper half
space. Thus for $T$ we obtain a meromorphic extension, with poles only at
those $z$ on the positive imaginary axis for which $z^2$ is an
eigenvalue for $\mathcal L$.  We also note the symmetry
\begin{equation}\label{T-even}
T(-\bar z) = \bar T(z).
\end{equation}
As  $u$ evolves along the \eqref{kdv} flow, the functions $L,R,T$
evolve according to

\[
T_t = 0, \qquad L_t = - 8i \xi^3 L, \qquad R_t = 8i\xi^3 R.
\]
The scattering map for $u$ is given by 
\[
u \to R,
 \]
and the map $u \to R$ conjugates the KdV flow \eqref{kdv} to (the
Fourier transform of)  the linear Airy flow. Reconstructing
$u$ from $R$ requires solving a Riemann-Hilbert problem, see
\cite{MR1263128} for this approach for the KdV
equation.

One  difference  between the case of NLS-mKdV and that of KdV is 
that for the latter we need\footnote{Or some other setting where 
$\int u$ is defined.}   $u\in L^1$ in order to define  $T(\xi)$ in 
the upper half-plane.  This is due to the linear term in $T$, which 
after one iteration is seen to have  the form
\[
T(z,u) = 1 - \frac{1}{2iz} \int_\R u(x) dx  + \text{quadratic and higher terms}
\]
However the linear effects are best understood at the level of $\log T$,
for which not only we have the similar relation 
\[
\log T(z,u) =  - \frac{1}{2iz} \int_\R u(x) dx  + \text{quadratic and higher terms}
\]
but the remainder term in this approximation can be defined  in
terms of only $L^2$ type norms of $u$, and no longer requires
integrability for $u$. This can be viewed as a renormalization of $\log T$,
and is discussed in detail in the next subsection.

For now we note one consequence
of this, namely the bound
\begin{equation} \label{kdv_positivity} 
 -\real (z^2 \log  T(z/2) )  - \im z \int_\R  u(x) \, dx \geq 0   
\end{equation} 
This is easily seen for Schwartz functions $u$, where the above expression
is superharmonic in the upper half-space, nonnegative on the real axis and vanishing at infinity.
Then  by density it holds for all $u \in H^{-1}$.

\subsection{The transmission coefficient in the upper 
half-plane and conservation laws} 

As in the NLS-mKdV case, the  conserved energies are constructed using 
the transmission coefficient $T$ in the upper half-space. Assuming in a 
first approximation that  $u$ is a Schwartz function,
the transmission coefficient $T$ will be a meromorphic function in the upper 
half-space, with poles in a compact subset of the positive imaginary axis.
Further, $\ln |T|$ is a Schwartz function on $\R \setminus \{0\}$, and 
$\ln T$ has a formal series  expansion as $z \to \infty$ in the upper half-space, 
\[ 
\ln T \approx  i   \sum_{j=-1}^{\infty}  E_{j} (2z)^{-2j-3}
\]

where $E_j$ are  the conserved energies for $j \geq 0$,
and
\[
E_{-1} =  \int_\R u \ dx. 
\]
Here in view of the symmetry \eqref{T-even} we have only odd terms in the 
expansion.

In particular if there are no poles for $T$ in the upper half-plane, 
then the $E_j$'s can be expressed by Cauchy's theorem as
\begin{equation}  
E_j= - \frac1\pi \int \xi^{2j+2}  \ln |T(\xi/2)|d\xi \label{energy-j}   
\end{equation} 
We would like to define our conserved energies as 
\[
E_s = - \frac1\pi \int  (1+\xi^2)^{s} \xi^2 \ln |T(\xi/2)|d\xi, \qquad  s \geq -1
\]
except that this formula would have a very restrictive range of validity,
namely for Schwartz $u$ for which $\mathcal L$ has no negative eigenvalues.
To avoid this difficulty, we again switch the contour of integration to the 
positive imaginary line.

Precisely, if  $u$ satisfies the smallness condition $\| u\|_{H^{-1}} \ll 1$,
the transmission coefficient $T$ will be a holomorphic function 
in $H \setminus i[0,1/4]$. Then, as in the NLS-mKdV case, 
we define the energies $E_s$ in the range 
\[
 N \leq s < N+1, \quad N \geq -1
\]
by
\begin{equation}\label{def++kdv} 
 E_s(u) =  - \frac{ 2\sin(\pi s)}\pi    \bigintsss_1^\infty  \tau^2 (\tau^2-1)^s   \Big(\real \ln T(i\tau/2)+
\sum_{j=-1}^N (-1)^j E_j   \tau^{-2j-3} \Big)  d\tau + \sum_{j=0}^N \binom{s}{j} E_j  , 
\end{equation}     
where the last term accounts for the half-residues at zero of the
corrections.  Again, if $u$ is Schwartz and $T$ has no poles in the
upper half-space then this agrees with the previous formula. However,
if $\|u\|_{H^{-1}} \ll 1$ then this last expression is defined as an
absolutely convergent integral for all $u \in H^s$. 

We remark here that the $E_{-1}$ contribution is subtracted for all values of $s$,
which corresponds to the renormalization of $\log T$ previously alluded to.

 Just as in the NLS/mKdV case, here we also have trace formulas
relating the above energies to  the corresponding integrals on the real line:

\begin{proposition}[Trace formula] \label{trace-kdv} 
a)  Let $s > -1$.  Define
\[ 
\Xi_s (t) =   \int_0^t  \zeta^2(1-\zeta^2)^s d\zeta.
\] 
Then for $ u \in \mathcal{S}$ satisfying the smallness condition $\| u\|_{H^{-1}} \ll 1$
 we have
\begin{equation}\label{tracekdv} 
\begin{split}  
E_s = & \  - \frac1\pi  \int_\R   \xi^2 (1+\xi^2)^{s} \real  \ln T(\xi/2) d\xi  +  \sum_j  \Xi_s(2\kappa_j) 
\end{split} 
\end{equation} 
where the sum  runs  over the eigenvalues $i\kappa_j$ of $\mathcal L$ on the positive imaginary axis.

b) In the case $s = -1$ we have,  provided that $-1/4$ is not an eigenvalue of the Lax  operator,  
\begin{equation} \label{tracekdv-} 
\begin{split} 
E_{-1} = & \ - \frac1\pi   \int_\R \frac{\xi^2}{1+\xi^2} \real \ln T(\xi/2)   d\xi 
+  \sum  \Big[ (\ln(1+2\kappa_j)-\ln |1-2\kappa_j|) - 4\kappa_j\Big]   
\\   = & \, \real \ln T(i/2) - \int  u \, dx.   
\end{split}  
\end{equation}
\end{proposition}

The proof of the trace formula is similar to the corresponding proof in the NLS/mKdV case,
using Lemmas~\ref{l:apriori-n},\ref{l:apriori-s}, with the notable difference that these are now 
applied to the nonnegative superharmonic function
\begin{equation}\label{G-kdv}
G(z) =  \real \left(- z^2 \ln T(z/2) + iz \int u \ dx\right).
\end{equation}
We note that if $u \not \in L^1$ then in the first integral in \eqref{tracekdv} 
one should replace $\xi^2 \real \ln T(\xi/2) d\xi $ with  the trace of $G$ on the real line,
which is defined as a bounded measure, and omit the $\int u dx $ term on the right hand side. The case $s=-1$ is even simpler: Assuming that $\frac{i}2$ is not an eigenvalue we use the residue theorem to express the term in the bracket of $ G(i)$   by a Cauchy integral which we then move the contour to the real line.  

Now we are ready to state the complete form of our main result for the KdV equation. 

\begin{theorem}\label{energies-kdv+} 
  For each $s >-1 $ and if $s=-1$ and $-1/4$ is not in the spectrum of the Lax operator  $\mathcal L$,
both sides of \eqref{tracekdv} resp.\eqref{tracekdv-}  are well-defined. 
They define  a continuous map 
\[
E_s : H^s  \to \R^+
\]
and 
\[ E_{-1} : H^{-1}\backslash\{ u : -1/4 \text{ is a $\mathcal L$ eigenvalue }\} \to \R^+
\] 
with the following properties:
\begin{enumerate}
\item $E_s$ is conserved along the KdV flow.
\item The trace of the function $G$ defined by \eqref{G-kdv} on the
  real line is defined as a locally finite measure on the real line,
  and the trace formulas in \eqref{tracekdv} and \eqref{tracekdv-}
  hold. 
\item  If $ \Vert u \Vert_{H^{s}}  \le \delta $ then we have 
\begin{equation}
\left| E_s(u) - \|u\|_{H^s}^2 \right| \lesssim  \Vert u \Vert_{H^{-1}} 
 \Vert u \Vert_{H^{\max\{-1,s-\frac18\}}}^2. 
\end{equation}
\item $E_s$ is continuous in $s$ and $u \in H^\sigma$ for 
$-1\le  s \le \sigma$, analytic in $u$ if $-1/4$ is not an eigenvalue
of the Schr\"odinger operator,  and analytic in both variables if in addition $s < \sigma$. 
\item There exists $ \delta >0$ so that all eigenvalues of $\mathcal L$  are above $-\frac18$
if $\Vert u \Vert_{H^s} \le \delta$. 
  \end{enumerate}
\end{theorem}

The proof of the theorem follows the same outline as in the NLS/KdV case. The transition from the 
small data case $\|u \|_{H^{-1}} \ll 1$ and the large data case is based again on the trace formulas
in the preceding proposition. Hence for the remainder of this section we discuss the small data 
case, where the bulk of the analysis is devoted to the estimates for the terms in the asymptotic
expansion of $\log T$. 

\subsection{Estimates for the transmission coefficient}

We begin with the  formal homogeneous 
expansion of the transmission coefficient $T$, namely
\[
T(z) = 1 + \sum_{j=1}^\infty T_{2j}(z)
\]
where $T_{2j}(z)$ are multilinear integral forms, homogeneous of degree
$j$ in $u$.  There is a similar though less explicit expansion for
$\ln T$,
\[
-\ln T(z) = \sum_{j=1}^\infty \tilde T_{2j}(z)
\]
We have $\tilde T_4 = -T_{4}$, while $\tilde T_{2j}$ are still multilinear integral forms of degree
$j$ in $u$.

 Here we take advantage of  the similarity between the systems \eqref{scatter-kdv} and \eqref{scatter}.
Precisely, the multilinear forms $T_{2j}$ and $\tilde T_{2j}$ are the same as before, with the 
only difference that these forms apply to the pair of functions $(1,u)$ rather than $(u,\pm \bar u)$.
In particular, we can still take advantage of the improved structure of $\tilde T_{2j}$.
We now discuss the successive terms in the $\ln T$ series.

\bigskip

\subsubsection{ The role of the $T_{2}$ term.} This is linear in $u$, and equal
\[
T_2(z) = \frac{i}{2z}   E_{-1}   =  \frac{i}{2z} \int u(x) dx
\]
We note that such a term did not arise for NLS-mKdV, and it cannot be bounded
by $\|u\|_{H^s}$. Thus the correct strategy is to view this as a renormalization term.
Precisely, while $T_2(z)$ is not generally defined for $u \in H^s$, the expression 
$\ln T(z)  - T_2(z)$ restricted to $i[1/2,\infty)$  will extend smoothly to all $u$ which are small in $H^{-1}$.
Indeed, we have

\begin{lemma}\label{analytic}
The map $u \to \exp(\frac{1}{2iz}\int u dx ) T(z)$ is analytic for $\|u\|_{H^{-1}} \ll 1$ and $z \text{ near }  i[1/2,\infty)$.
\end{lemma}

\begin{proof} 
We seek to solve the system
\begin{equation}\label{scatter-re-kdv}
\left\{ 
\begin{aligned}
\frac{d\psi_1}{dx} =& -iz \psi_1+ \psi_2 
\\
\frac{d\psi_2}{dx} = & iz   \psi_2 + u \psi_1
\end{aligned}
\right.
\end{equation}
and prove Lemma~\ref{analytic}.
If $u \in L^1 \cap H^{-1}$   then a direct iterative scheme yields the existence of a unique solution
\[
(\psi_1, \psi_2) \in (\dot W^{1,1} \cap \dot W^{1,2}) \times (W^{1,1} \cap L^2)
\]
with the property that 
\[
\lim_{x \to -\infty}(\psi_1,\psi_2) = (1,0), \qquad \lim_{x \to -\infty}(\psi_1,\psi_2) = (T^{-1}(iz),0).
\]
Our goal here is to renormalize $T(i\tau)$, and show that the expression 
\[
S(i\tau) = T(i\tau) e^{-\frac1\tau \int u} 
\]
depends analytically only on $u$ in $H^{-1}$. Setting $z = i\tau $, we will rewrite the equation 
in a more favorable form. First we multiply $\psi_1$ and $\psi_2$ by $e^{\tau x}$. 
By an abuse of notation we keep the notation for $\psi_1$ and $\psi_2$, 
which now satisfy
\begin{equation} 
\begin{split} 
\psi_1' = &\psi_2  \\
\psi_2' = &-2\tau  \psi_2 + u \psi_1. 
\end{split} 
\end{equation} 
Here we start with initial datum $(1,0)$ at $-\infty$, and $\psi_2(\infty) = 0$ while $\psi_1(\infty) = T(i\tau)^{-1} $. To peel off the low regularity part of $\psi_2$ we split $u$ into low and high frequencies and define 
\[ U=  \int_{-\infty}^x e^{-2\tau(x-y)} u(y) dy,\qquad u_{lo} = \tau \int_{-\infty}^x e^{-(x-y)} u(y) dy .\] 
Then $u = U_x + u_{lo}$ and 
\[ \hat u_{lo} (\xi) = \frac{1}{1 + i\xi/\tau} \hat u \] 
so that we have
\begin{equation} \Vert u \Vert_{H^{-1}} \le  \tau \Vert u_{lo} \Vert_{L^2}  + \Vert U \Vert_{L^2}  \le  2 \Vert u \Vert_{H^{-1}} \end{equation}

We now define renormalized variables $(w_1,w_2)$ by
\[
w_2 = e^{ -\frac1\tau \int_{-\infty}^x  u_{lo}} \left( \psi_2 - U  \psi_1\right) , \qquad   w_1 = e^{-\frac1\tau\int_{-\infty}^x u_{lo} } \left( \psi_1 + \psi_2\right)  
\]
and  rewrite the system in terms of $w_1$ and $w_2$ as
\begin{equation}\label{scatter-re4}
\left\{ 
\begin{aligned}
w_1' =& -(u_{lo} + U) w_2 - U^2 (w_1-w_2)   
\\
w_2' = & -2 \tau w_2  -  U w_1    
+ u_{lo} (w_1-2w_2)    -   U^2 (w_1-w_2)
\end{aligned}
\right.
\end{equation}
Here $(w_1,w_2)$ start with the same data 
$(1,0)$ at $-\infty$, but now $w_2(\infty) = 0$ while $w_1(\infty) = S(i\red{\tau})$.
Hence it suffices to show that we can solve the equation 
\eqref{scatter-re4} iteratively provided that $\| u\|_{H^{-1}} \ll 1$. 
Notably, we are no longer assuming that $u \in L^1$.

 Let $ C_c$ be the space of continuous functions with
limits at $\pm \infty$. We define
\begin{equation} 
X= C_c\times l^2_\tau L^\infty  
\end{equation}  
and the linear map  $ L: X \to X  $ where $(v_1,v_2) = L(w_1,w_2)$
if $(v_1,v_2)$ solve the linear system 
\[
\begin{split}
v_1' =& \  -(u_{lo} + U) w_2 + U^2(w_1-w_2) 
\\  v_2'   = & \ -2\tau  v_2 + u_{lo} (w_2-w_1) + U w_1  + U^2(w_1-w_2) 
\end{split}
\]
with zero Cauchy data at $-\infty$. We can now recast the system \eqref{scatter-re4} in the form
\begin{equation}\label{how-to}
(w_1,w_2) = (1,0) + L(w_1,w_2)
\end{equation}
In order to solve this iteratively we need the following

\begin{lemma} We can decompose $L = L_1+L_2$ where 

(i) $L_1$ is linear in $u$ and $(w_1,w_2)$ and satisfies  
\[ \Vert L_1 \Vert_{X\to X} \lesssim  \Vert u \Vert_{H^{-1}}.\] 

(ii) The map $L_2$ is linear in $(w_1,w_2)$ and quadratic in $U$, and  satisfies 
\[ \Vert L_2 \Vert_{X \to X } \le c \Vert U \Vert_{L^2}^2 \] 
\end{lemma}

\begin{proof} 
By the first equation 
\[ \Vert v_2 \Vert_{l^2 L^\infty } \le  \Big(\Vert u_{lo}  \Vert_{L^2}+ \Vert U \Vert_{L^2} + \Vert   U \Vert_{l^2_\tau L^2}^2 \Big)             
\Big(\Vert w_2 \Vert_{L^\infty} +  \Vert w_1 \Vert_{L^\infty} \Big) 
 \] 
and by the second 
\[ \Vert v_1 \Vert_{L^\infty} \le \Vert u_{lo} + U \Vert_{L^2} \Vert w_2 \Vert_{L^2} 
+ \Vert U \Vert_{L^2}^2  \Vert  w_1 \Vert_{L^\infty} + \Vert w_2 \Vert_{L^\infty}. 
\] 
\end{proof} 

Hence, if $\|u\|_{H^{-1}} \ll 1$ then $\|L\|_{X \to X} < 1$, which allows us to solve
the equation \eqref{scatter-re4} by
\[
(w_1,w_2) = \sum_{j = 0}^\infty L^j (1,0)
\]
Finally, we have 
\[ 
(S(i\red{\tau}),0) = \lim_{x \to \infty} L^j (1,0)(x)
\]
and the conclusion of Lemma~\ref{analytic} follows.
\end{proof} 

Lemma \ref{analytic} together with a rescaling implies that there is a constant $C$ so that with 
\[ h(t) = 1+ t^2 e^t \]  
\[  e^{-T_2(i\tau)} T(i\tau)  \preceq h(C \tau^{-1/2} \Vert u \Vert_{H^{-1}_\tau}) .   \]   
where terms with like homogeneity are separately compared, and
\[
 \| u\|_{ H^{-1}_\tau}^2 = \int (\tau^2 +\xi^2)^{-1} |\hat u(\xi)|^2 d\xi 
\] 
Since $\ln ( 1+t) \preceq \dfrac{t}{1-t} $, this gives 
\[ 
\ln (e^{-T_2(i\tau)} T(i\tau)) \preceq  \frac{h(C(\tau^{-1/2})\Vert u  \Vert_{H^{-1}_\tau})}{1- h(C(\tau^{-1/2})\Vert u  \Vert_{H^{-1}_\tau})}.  
\]
Hence there exists $C>0$ so that , for $j \ge 2$,
\begin{equation}
\label{t2j-first}
   |\tilde T_{2j}|  \le  C (\tau^{-1/2} \Vert u \Vert_{H^1_\tau})^j. 
\end{equation}

\bigskip
 
\subsubsection{The $\tilde T_{4}$ term.} This is quadratic in $u$, and plays
exactly the same role played by $T_2$ before. We have
\[
\begin{split} 
 \tilde T_4(z) =&  -2 \int_{x_1<x_2<y_1<y_2} e^{2iz(y_1+y_2-x_1-x_2)}  u(x_1) u(x_2) dx dy  
\\ =& \frac{1}{4 z^2} \int_{x_1<x_2}e^{2iz(x_2-x_1)} u(x_1) u(x_2) dx_1 dx_2  
\end{split} 
\] 
and hence, as for the NLS case, 
\begin{equation} 
E_{s,4} = \| u\|_{H^s}^2
\end{equation} 

\bigskip

\subsubsection{ Low regularity bounds for the 
$\tilde T_{2j}$ term, $j \geq 3$.}
We note that  unlike the case of NLS-mKdV, no such bound holds
for $T_{2j}(i\tau)$, which contains contributions from the forbidden energy $E_{-1}$.
From Lemma~\ref{analytic} we obtain  
\begin{equation}\label{t2j-first+}
 |\tilde T_{2j}(i\tau )| \lesssim \tau^{-\frac{j}2} \Vert u \Vert_{H^{-1}_\tau}^j 
\le 
\tau^{-\frac{3j}2}  \| u\|_{L^2}^j  \lesssim \tau^{-\frac{3j}2} \Vert u \Vert_{H^{\frac{j-2}2}}^2     \Big( c   \Vert u \Vert_{H^{-1}}\Big)^{j-2}.  
\end{equation}
The trivial bound 
\[|\tilde T_{2j}(i\tau )| \lesssim \tau^{-\frac{j}2} (c\Vert u \Vert_{H^{-1}_\tau})^j 
\le 
\tau^{-\frac{j}2}  (c\| u\|_{H^{-1}})^j 
\] 
suffices for the pointwise estimate up to $s=\frac{j}4-\frac32$.
For $-\frac34 \le s \le \frac{3(j-2)}4$ we interpolate the two estimates to arrive at 
\[ |\tilde T_{2j}(i\tau)|\lesssim \tau^{-3-2s} \Vert u \Vert_{H^{s-1/4}}^2 \Vert u \Vert_{H^{-1}}^{j-2}, \]  
with the case $j=3$ being the worst one. 
The benefit  in having such a bound is two-fold:
\begin{enumerate} 
\item  It fully covers all $-1\le s \leq 0$; to account for larger $s$ we will no longer have to deal with $H^{-1}$ norms.
\item  For every $s > 0$ it accounts for all large enough $j$, so in the sequel we only need to obtain bounds 
for fixed $j$.
\end{enumerate}

\bigskip 

\subsubsection{ Higher regularity bounds for the 
$\tilde T_{2j}$ term, $j \geq 3$.}
Here we take advantage of the structure properties of $\tilde T_{2j}$, namely that it contains only connected 
integrals. Because of this we can use an unbalanced version of the bound \eqref{tt2j}, namely 
\begin{equation} 
|\tilde T_{2j}(i \tau)(u,v) | \lesssim \| u\|_{l^j_{\tau} DU^2}^j \|v\|_{l^\infty_{\tau} DU^2}^j .
\end{equation}
We apply this with $v = 1$, which satisfies 
\[
\| 1\|_{l^\infty_{\tau} DU^2} \lesssim \tau^{-1},
\]
to obtain
\begin{equation} 
|\tilde T_{2j}(i \tau) | \lesssim \tau^{-j} \| u\|_{l^j_{\tau} DU^2}^j. 
\end{equation}
Using the Sobolev embedding \eqref{emb}, this allows us to extend the bound 
 \eqref{t2j-first+} to the range $-1 \leq s \leq \frac12 - \frac1j$. In particular 
we have  the estimate
\begin{equation}
|\tilde T_{2j}(i\tau )| \lesssim \tau^{-(2j-1)}  \| u\|^j_{H^{\frac12 -\frac1j}}.
\end{equation}
Then  by interpolation we get the uniform bounds 
\begin{equation}
|\tilde T_{2j}(i\tau )| \lesssim \tau^{-(2j-1)} \| u\|_{H^{\frac{3(j-2)}4}}^2 \| u\|_{H^{-1}_\tau}^{j-2}    
\end{equation}
and
\begin{equation}
|\tilde T_{2j}(i\tau )| \lesssim  \tau^{-3-2s}     \| u\|_{H^{s-\frac14}}^2 \| u\|_{H^{-1}_\tau}^{j-2}    
\end{equation}
for $-1 \le s \le j-2$. In particular, for such $s$ we obtain for all $\varepsilon>0$ 
\begin{equation} \label{middle}  
E_{j,s} \lesssim \Vert u \Vert_{H^{s-\frac14+\varepsilon}}^2 \Vert u \Vert_{H^{-1}}^{j-2}. 
\end{equation} 
Thus it remains to consider the case $s > j-2$.

\bigskip

\subsubsection{ The expansion for the 
$\tilde T_{2j}$ term, $j \geq 3$.}  Here the proof of Proposition~\ref{p:t2j-exp}
applies with minor but important changes to give
\begin{equation}\label{t2j-main}
\left |\tilde T_{2j}(i\tau/2 ) -  \sum_{l=0}^N (-1)^{j+l} T_{2j,l}  \tau^{-2j-2l-1}   \right| \lesssim 
\sum_{2N+1 \leq |\alpha| \leq j -1 + 2N }^{\max \alpha_k \leq N+1}  \tau^{-2j-1-|\alpha|}{\red\prod_k}  \| \partial^{\alpha_k} u\|_{l^j_{\tau} DU^2}.
\end{equation}
This we seek to use in the range $j-2+N \leq s < j-1+N$. Indeed, for $j-2+N + \frac14 \leq s \leq  
j-1+N+\frac14$ the above bound implies that
\begin{equation} 
\Big|\tilde T_{2j}(i\tau/2 ) -  \sum_{l=0}^N (-1)^{j+l} T_{2j,l}  \tau^{-2j-2l-1}   \Big| \lesssim 
\tau^{-3-2s} \Vert u \Vert_{H^{s-\frac14}}^2 \Vert u \Vert_{H^{-1}}^{j-1}. 
\end{equation} 
Due to the room of $\frac14$ we easily obtain the integrated bound 
\[ 
\int_1^\infty \tau^{2s+2} 
\Big|\tilde T_{2j}(i\tau/2) - \sum_{l=0}^N (-1)^{j+l} T_{2j,l} \tau^{-2j-2l-1}\Big| d\tau \le c \Vert u \Vert_{H^{s-\frac14+}}^2 
\Vert u \Vert_{H^{-1}}. 
\] 
One can further refine the argument using a Littlewood-Paley decomposition, 
as in the proof of Proposition~\ref{p:t2j-exp}, to reach the $ H^{s-\frac14}$ bound in the last estimate and in \eqref{middle}.

\subsection{The cubic term in the KdV energies}
For completeness we also provide the Fourier expression for the cubic term in $\ln T(z)$,
namely $\tilde T_6(u)$, as well as the cubic term in our energy functionals $E_s^3$. 
These are given by specializing \eqref{T6} Lemma  \ref{lHj6} and setting formally  $\bar u =1$ there
\begin{equation}\label{tildeT6}  
  \tilde T_6(z/2) = \frac{i}{(2\pi)^{\frac12}} 
 \int_{\xi_1+\xi_2+\xi_3=0}  \frac{1}{z^2(z+\xi_1)(z+\xi_2)}
\left(\frac{1}{z} +\frac{1}{z-\xi_3}\right) \hat u(\xi_1) \hat u(\xi_2) \hat u(\xi_3) d\xi_1d\xi_2, 
\end{equation}
similarly from  \eqref{T6exp} 
\begin{equation}   E_{k,3} =  \sum_{\alpha_1+\alpha_2=2k-2} 
\int u^{(\alpha_1)} u^{(\alpha_2)} u dx+
\sum_{\alpha_1+\alpha_2+2\alpha_3=k-2} (-1)^{k+1+\alpha_3}
\int u^{(\alpha_1)} u^{(\alpha_2)} u^{(\alpha_3)}dx. \end{equation}    
To compute the cubic component $E_{s,3}$ of the energy it is useful to symmetrize \eqref{tildeT6} in $\xi$ using 
\[
\begin{split} 
 \frac13\Big[  (2z-\xi_1)(z+\xi_1)(z-\xi_2)(z-\xi_3)\hspace{-5cm} & \hspace{5cm} +(2z-\xi_2)(z+\xi_2)(z-\xi_3)(z-\xi_1)
\\ & + (2z-\xi_3)(z+\xi_3)(z-\xi_1)(z-\xi_2) \Big]  \\  
= & 2 z^4 - (\xi_1+\xi_2+\xi_3) z^3-\frac13 (\xi_1^2+\xi_2^2 +\xi_3^2)z^2 
\\ & +[\xi_1\xi_2\xi_3 + \frac13 \{\xi_1^2(\xi_2+\xi_3) + \xi_2^2 (\xi_3+\xi_1) + \xi_3^2 (\xi_1+\xi_2)\}  ]  z
\\ &    - \frac13 (\xi_1+\xi_2+\xi_3)\xi_1\xi_2\xi_3 
\end{split} 
\] 
and 
\[ 2 [ \xi_1^2(\xi_2+\xi_3) + \xi_2^2 (\xi_3+\xi_1) + \xi_3^2 (\xi_1+\xi_2)]
= (\xi_1+\xi_2+\xi_3)( (\xi_1+\xi_2+\xi_2)^2 - \xi_1^2-\xi_2^2-\xi_3^2) 
- 6 \xi_1\xi_2\xi_3.\] 
Since we integrate over $\xi_1+\xi_2+\xi_3=0$ we arrive at 
 \begin{equation} 
  \tilde T_6(z) = \frac{i}{(2\pi)^{\frac12}} 
 \int_{\xi_1+\xi_2+\xi_3=0}  \frac{2 z^2- \frac13 (\xi_1^2+\xi_2^2 +\xi_3^2)}{z(z^2-\xi_1^2)(z^2-\xi_2^2)(z^2-\xi_3^2)}
 \hat u(\xi_1) \hat u(\xi_2) \hat u(\xi_3) d\xi_1d\xi_2 
\end{equation}
Then we have  
\[
E_{s,3} = \real \frac{i}{\pi  (2\pi)^{\frac12}} \int_{\R+i0} (1+z^2)^s 
 \int_{\xi_1+\xi_2+\xi_3=0}   \frac{2z^3- z(\xi_1^2+\xi_2^2 +\xi_3^2)/3}{(z^2-\xi_1^2)(z^2-\xi_2^2)(z^2-\xi_3^2)} \hat u(\xi_1) \hat u(\xi_2) \hat u(\xi_3) d\xi_1d\xi_2\,dz .
\]
To express the complex conjugate of the above integral in a similar manner we use the fact that  
$u$ 
is real, which implies that  we have $\hat u(-\xi) = \bar{ \hat u}(\xi_1)$.
Then a straightforward computation yields
\[
E_{s,3}  =- \frac{1}{\pi} \int_{\R+i0}\!\! (1+z^2)^s 
 \int_{\xi_1+\xi_2+\xi_3=0}  \!\!  \im \frac{2z^3- z (\xi_1^2+\xi_2^2 +\xi_3^2)/3}{(z^2-\xi_1^2)(z^2-\xi_2^2)(z^2-\xi_3^2)}\real( \hat u(\xi_1) \hat u(\xi_2) \hat u(\xi_3)) d\xi_1d\xi_2\,  dz
\]
We apply Fubini and move the inner $z$ integral - including the factor $1/\pi$ - to the real axis:
\[
\begin{split}
K(\xi_1,\xi_2,\xi_3)  = &\ \frac1{\pi(2\pi)^\frac12}  
\int_{\R+i0} \im \frac{2z^3- z (\xi_1^2+\xi_2^2 +\xi_3^2)/3}{(z^2-\xi_1^2)(z^2-\xi_2^2)(z^2-\xi_3^2)} dz 
\\
= & \ \frac{2}{3\xi_1 \xi_2 \xi_3}  \left( \xi_1 (\delta_{\xi_1} + \delta_{-\xi_1})
+  \xi_2 (\delta_{\xi_2} + \delta_{-\xi_2})+ \xi_3 (\delta_{\xi_3} + \delta_{-\xi_3})\right)
\end{split}
\]
where we used that 
\[
\begin{split} 
\hspace{-1.9pt} [(6\xi_1^2-(\xi_1^2+\xi_2^2+\xi_2^2)]\xi_1\xi_2\xi_3 - 2 \xi_1(\xi_1^2-\xi_2^2)(\xi_1^2-\xi_3^2) = &\, \xi_1 \Big[  (5 \xi_1^2-\xi_2^2-\xi_3^2)\xi_2\xi_3
-2(\xi_1^2-\xi_2^2)(\xi_1^2-\xi_3^2)\Big]
\\ & \hspace{-5cm}   =   \, \xi_1 (\xi_1+\xi_2+\xi_3) (2\xi_1^2-\xi_2-\xi_3)(\xi_2+\xi_3-\xi_1) 
\end{split} 
\] 
 to compute the coefficient of $\delta_{\xi_1}$, from which we can read of the other coefficients.  Hence we obtain 
\begin{equation}  
E_{s,3} = \frac{2}{3(2\pi)^{\frac12}} \int_{\xi_1+\xi_2+\xi_3=0} \frac{(1+\xi_1^2)^s \xi_1 +(1+\xi_2^2)^s\xi_2 +(1+\xi_3)^s\xi_3}{\xi_1
\xi_2\xi_3 }\hat u(\xi_1) \hat u(\xi_2) \hat u(\xi_3) d\xi_1 d\xi_2  
\end{equation} 
where the latter formula agrees again with the I-method prediction, see 
\cite{MR3058496,MR3292346}. 
We can check the calculation by evaluating $E_{1,3}$ using  
\[ \xi_1^3 +\xi_2^3+\xi_3^3-3\xi_1\xi_2\xi_3= (\xi_1+\xi_2+\xi_3)(\xi_1^2+\xi_2^2+\xi_3^2-\xi_1\xi_2-\xi_2\xi_3-\xi_3\xi_1) \] 
which gives
\[
\begin{split} 
 E_{s,3} = &\,  \frac{2}{(2\pi)^{\frac12} }  \int_{\xi_1+\xi_2+\xi_3=0} \hat u(\xi_1) \hat u(\xi_2) \hat u(\xi_3) d\xi_1 d\xi_2 
\\ = &\,  \frac{2}{(2\pi)^{\frac32} } \hat u * \hat u *\hat u(0) 
\\ = & \,   \frac{2}{(2\pi)^{\frac12} } \widehat{u^3}(0) 
\\ = & 2 \int u^3 dx. 
\end{split} 
\] 
\appendix

\newsection{A Hopf algebra} 
\label{a:hopf}

\renewcommand{\theequation}{\Alph{section}.\arabic{equation}}
\renewcommand{\thesubsection}{{\Alph{section}.\arabic{subsection}}}
\renewcommand{\thetheorem}{\Alph{section}.\arabic{theorem}}

The aim of this section is to frame the algebra of iterated integrals
as a Hopf algebra, and then to take advantage of the Hopf algebra
structures in order to prove Theorem~\ref{primitive}.  Our iterated
integrals are linear combinations of integrals of the form
\[
\int_\Sigma u(x_1) \overline{v(y_1)} \cdots u(y_j) \overline{v(y_j)} dx dy
\]
where $\Sigma$ represents a complete ordering of the variables $x_l$,
$y_l$ with the constraint that
\[
\Sigma \subset \{ x_l < y_l; \ l = 1,...,j\}
\]
Omitting the indices this can be represented as a word with two letters
$X$ and $Y$ like \[  XXYY \qquad \text{ for } \qquad  x_1 < x_2 < y_1<y_2.\]

Here  we will restrict our attention to words which have an equal number of
$X$'s and $Y$'s.  Further, the requirement $x_l <y_l$ leads to the
following additional property of our words: For any splitting of the
word into two words, the left word contains at least as many $X$'s as
$Y$'s. 

To each word in this class we associate a graphical representation by
nonintersecting arcs as follows.  Replace $X$ by $\<X>$ and $Y$ by
$\<Y>$. Then there is exactly one way of connecting each $\<X>$ arc by an 
$\<Y>$ arc so that the arcs do not intersect, for instance
\[ 
XXYY \to \<X>\<X>\<Y>\<Y> \to \<XXYY>. 
\]
Thus by  a slight  abuse  of language we call the words in $H$ nonintersecting, and, by a bigger abuse, all nonintersecting words we call intersection (like $\<X>$, which shows the abuse). 

The shuffle product $\shuffle$ on words in the alphabet $\{X,Y\}$ maps two words 
to the sum of all the words which are obtained by shuffling the two words, i.e. 
forming a word from the letters of the two words while respecting the ordering 
of both words. For example 
\[
 X \shuffle XY = 2XXY + XYX. 
\] 
It defines a ring of formal power series with the words as unknowns and 
the shuffle product defining the commutative multiplication. We denote 
this ring by $H^{\shuffle}$ and refer to Lothaire \cite{MR675953} for a more detailed discussion.

Here we will not work with the full ring $H^{\shuffle}$.  We denote the
subset of $H^{\shuffle}$ containing only formal series of
nonintersecting words by $H$.  The shuffle product of nonintersecting
words can easily seen to be a sum of nonintersecting words.  Hence $H$
endowed with the shuffle product is a subalgebra of $H^{\shuffle}$.

We define the degree of a symbol to be half of the length of the 
nonintersecting word. A symbol of length $j$ is identified with 
an integral over a subset of $\R^{2j}$ containing $j$ factors $u$ and 
$j$ factors $v$. 
The length $2j$  introduces a grading on $H$ which is compatible with the shuffle 
product: Words of length $2j$ have degree $j$.
The elements of degree $>m$ form trivially an ideal $I_m$. 
The quotient $H_m= H/ I_m$ is a finite dimensional algebra.

We identify nonintersecting words with integrals. For two given
Schwarz functions $u$ and $v$, any nonintersecting word defines an
integral as above. For example 
\[ 
\<XXYXYY> = \int_{x_1<x_2<y_1<x_3<y_2<y_3}  \prod u(x_j) \overline{v(y_j)} dx_j dy_j. 
\]
For two nonintersecting words $a$ and $b$ of
length $2n$ and $2m$ the product can written by an application of
Fubini's Theorem as an integral over a domain $U$ in $\R^{2(m+n)}$
given by the restrictions defined by the words $a$ and $b$. We
decompose $U$ into completely ordered sets of variables, neglecting
sets of measure zero.  The simplest nontrivial example is
\[
\begin{split}
 \left( \int_{x_1<y_1} u(x_1) \overline{v(y_1)} dx_1 dy_1 \right)^2 
= & \ 2 \int_{x_1<y_1<x_2<y_2} u(x_1) \overline{v(y_1)} u(x_2)\overline{v(y_2)} dx dy 
\\ & \  +  4 \int_{x_1<x_2<y_1<y_2} u(x_1) \overline{v(y_1)} u(x_2) \overline{v(y_2)} dx dy
\end{split}
\]
At the level of words this becomes  
\[ XY \shuffle XY = 2 XYXY + 4 XXYY \] 
and at the level of symbols 
\begin{equation}\label{relation} 
  \<XY>\shuffle \<XY> = 2 \<XY>\<XY> + 4 \<XXYY> . 
\end{equation}   

Then the product of the integrals is the same as the sum over
the integrals defined by the summands in the shuffle product of the
two words. In short the shuffle product on nonintersecting words is
compatible with the product of integrals. In abstract terms

\begin{lemma} Let $u$ and $v$ be Schwartz functions. The evaluation of
  the integrals is a ring homomorphism from the finite sums in $H$ to
  $\C$.
\end{lemma}   

The product of integrals introduces obvious relations: Linear
combinations of monomials of the same homogeneity are equivalent to
zero, if they vanish for every choice of functions $u,v\in
\mathcal{S}$. We do not know whether this leads to nontrivial relations 
in $H$.

A fundamental subset of the set of symbols is the set of connected symbols. 
We call a symbol connected if there is an arc from the first to the  last 
letter.   For instance
$\<XY>$, $\<XXYY>$, $\<XXYXYY>$ are connected, while $\<XY>\<XY>\ \ $ is not.
By extension we call an integral connected if the symbol is connected.  

Our interest comes from  the transmission coefficient $T$, which 
can be expressed as a power series (see Lemma \ref{T2exp})  
\[
T^{-1}  = 1+ \<XY> + \<XY>\<XY> + \dots
\]
Using the Taylor series for the log and the shuffle product, this also gives
a series for $\ln T$.
An  involved calculation gives

\begin{proposition} \label{a:log} The following formula holds 
\begin{equation} 
\begin{split} 
-\ln T = & \<XY> -2\<XXYY> +12 \<XXXYYY> + 4  \<XXYXYY> 
-144\<XXXXYYYY>-72\<XXXYXYYY> -24\<XXXYYXYY>-24\<XXYXXYYY>-8\<XXYXYXYY>
\\ & + 2880 \<XXXXXYYYYY> +1728\<XXXXYXYYYY>+864\<XXXXYYXYYY>
+864\<XXXYXXYYYY> +432\<XXXYXYXYYY> +288\<XXXXYYYXYY> \\& +288\<XXYXXXYYYY>+144\<XXXYXYYXYY> 
  +144\<XXYXXYXYYY>+48\<XXXYYXYXYY>+48\<XXYXXYYXYY>+48\<XXYXYXXYYY>+16\<XXYXYXYXYY> 
\\ & + \dots 
\end{split} 
\end{equation} 
by which we mean that the difference of the two  sides is  in $I_6$.

\end{proposition} 

The remarkable property is that the right hand side can be written as a sum 
of multiples of connected symbols.  Indeed, we have
\begin{theorem} \label{lnT} 
$\ln T$ is a formal sum of connected integrals. 
\end{theorem}   

To prove this we need more structure, namely a notion of a coproduct and a bialgebra, or 
even  a Hopf algebra.  Let $\tilde H$ be a graded algebra so that $\tilde H_m$ is always 
finite dimensional.
 A coproduct is a map
\[ \Delta: \tilde H  \to \tilde H \otimes \tilde H \] 
which is coassociative and codistributive (see \cite{MR2986663}).

An important standard example here is the shuffle algebra
$H^{\shuffle}$ of words in the alphabet $\{X,Y\}$ as above, see \cite{MR1231799,MR675953} for the facts about the shuffle algebra used below.  The
shuffle product is known and easily seen to be commutative, associative
and distributive. The coproduct $\Delta $ on $H^{\shuffle}$ 
is defined  on words as the sum over  all splittings
\begin{equation}  \Delta a = \sum_{a_1a_2=a}  a_1 \otimes a_2.
\label{coproduct} 
\end{equation}  

The coproduct of an intersecting  symbol is a linear combination of tensor
products of symbols with at least one intersecting  factor. The coproduct
of a nonintersecting  symbol will contain some tensor products of symbols with
nonintersecting  factors, but also some tensor products of symbols with
intersecting factors. In the latter case, the first factor must contain
more $X$'s, and the second more $Y$'s.  Then it is natural to define
the coproduct on $H$ as the nonintersecting  part of the $H^{\shuffle}$
coproduct. Thus the formula \eqref{coproduct} still applies, but with
$a$, $a_1$ and $a_2$ restricted to nonintersecting  factors.
As such, we note that $H$ is \emph{not} a subalgebra of $H^{\shuffle}$.

The shuffle on the tensor product is defined in a natural fashion 
\[ 
 (a\otimes b)\shuffle (c\otimes d)  =  (a\shuffle c)\otimes (b\shuffle d). 
\]
The coproduct of $H^{\shuffle}$ is coassociative: If 
\[  \Delta^{\shuffle}   a = \sum b_i \otimes c_i \] 
then 
\[ \sum (\Delta^{\shuffle} b_i) \otimes c_i = \sum b_i \otimes
(\Delta^{\shuffle} c_i). \] 
This implies the same property for the
coproduct on $H$.  The product and the coproduct of a Hopf algebra
satisfy the crucial compatibility condition
\begin{equation}  
\Delta (a \shuffle b)= \Delta a \shuffle \Delta b, 
\end{equation} 
see Reutenauer \cite{MR1231799}, Proposition 1.9. 
 
This is again inherited by $H$ from $H^{\shuffle}$, but it is less obvious.
The point is that if $a$ and $b$ are both nonintersecting, then nonintersecting  terms in $\Delta a \shuffle \Delta b$ can only arise from nonintersecting  terms 
in $\Delta a$ combined with nonintersecting  terms in $\Delta b$; else, we have too many
$X$'s in the first factor and too many $Y$'s in the second.

Now that we have the appropriate set-up, we return to our goal of
proving Theorem~\ref{lnT}.  We call an element $g \in H$ group-like 
if 
\[ 
 \Delta g = g \otimes g. 
\]
\begin{lemma} The set $G$ of all group-like elements endowed with the shuffle product is a group. 
\end{lemma} 
\begin{proof} 
 If  $g,h \in G$ then 
\[ \Delta (g\shuffle h)  = \Delta g \shuffle \Delta h = (g\otimes g) \shuffle (h\otimes h) 
= (g\shuffle h) \otimes (g\shuffle h ). \] 
Every element of $G$ starts with $1$. 
The inverse can be determined recursively: 
If 
\[ g= 1+ g_1 + g_2 \dots \] 
then 
\[ 
h= 1-g_1 + g_1\shuffle g_1 -g_2 \dots 
\]   
\end{proof} 
In particular, we note that 
\begin{lemma}\label{TG} 
The expression 
\[
T = 1 + \<XY> + \<XYXY> \  + \<XYXYXY> \ \ \ + \cdots
\]
belongs to $G$. 
\end{lemma}
\begin{proof}
Denoting by $\<XY>^{[k]}$ the expression obtained concatenating $k$ $\<XY>$'s, we have
\[
T = \sum_{n=0}^\infty \<XY>^{[n]}
\]
On the other hand, we have 
\[
\Delta T =   \sum_{n=0}^\infty \Delta \<XY>^{[n]}
= \sum_{n=0}^\infty \sum_{k=0}^n  \<XY>^{[k]}\otimes \<XY>^{[n-k]} = T \otimes T
\]
\end{proof}

The linear subspace  of $H$ where we seek to place $\ln T$ can be described 
as follows:
\begin{definition}
The  primitive elements are
\[ 
P = \{ p \in H : \Delta p = 1\otimes p + p \otimes 1 \}. 
\] 
\end{definition}
Indeed, we have

\begin{lemma} 
Primitive elements are formal linear combinations of  connected  symbols.
\end{lemma} 

\begin{proof} Let us call linear combinations of connected symbols
  indecomposable. Then trivially, indecomposible elements are
  primitive. Consider now a primitive element $p$, and show that it is
  indecomposable. We argue by contradiction. Let $a$ be the homogeneous
  part of $p$ of lowest grading $m$ which is not a linear combination
  of connected symbols.  Subtracting them, we may assume that $a$ is a
  part of lowest grade in $p$ which does not vanish.  Since $\Delta$
  preserves the grading (we equip the tensor product with the natural
  grading), we must have
\[ 
 \Delta a = 1\otimes a + a \otimes 1.  
\] 
It remains to deduce that $a$ is a linear combination of connected
symbols of the same degree. After subtracting multiples of connected
symbols we reduce the problem to proving that if $a$ is homogeneous,
primitive and a linear combination of non connected symbols than it
has to be $0$. This in turn is a consequence of injectivity of
$\Delta$.
\end{proof} 

Given the last two lemmas, Theorem~\ref{lnT} is a  consequence 
of the following more general result:

\begin{theorem} \label{logarithm} 
The above subgroup $G$ and the subspace $P$ of $H$ are related by
\[ 
G = \exp P
\] 
\end{theorem}  
\begin{proof} 
 We  first show that the exponential of any primitive $H$ element  
must belong to $G$,
\[
 \exp P \subset G.
\] 
Indeed, for primitive $h \in P$ we can write
\[
\exp(h) = \sum_{n=0}^\infty \frac{h^n}{n!}  
\]
and thus 
\[
\begin{split}
\Delta( \exp(h)) = & \   \sum_{n=0}^\infty \Delta \frac{h^n}{n!} 
\\ = & \     
 \sum_{n=0}^\infty  \frac{ (\Delta h)^n}{n!} 
\\ = &  \ \sum_{n=0}^\infty  \frac{ ( h \otimes 1 + 1 \otimes h  )^n}{n!} 
\\ = & \ \sum_{n=0}^\infty  \sum_{k=0}^n \frac{ h^k \otimes h^{n-k} }{k! (n-k)!} 
\\ = & \  \exp(h) \otimes \exp(h)
\end{split}
\]

We will show by induction that for all $n \in \N$ we have
\[ 
G / I_m = (\exp P)/I_m. 
\]

Since 
\[ 
1= \exp(0) 
\] 
the formula holds in grade  $m=0$. Suppose now that $ 
G / I_m = \exp p/I_m$.
We want to prove that then $
G/I_{m+1} = \exp P/I_{m+1}$.
Clearly 
\[ \exp P/I_{m+1} \subset G/I_{m+1} \] 

Let $g \in G$. By the induction hypothesis, there exists $p \in P$ so that 
\[  
g - \exp p \in I_{m}. 
\] 
Let $h = \exp(-p)\shuffle  g $. It satisfies 
\[ 
h-1  \in I_{m}  
\] 
and, since it is group-like as product of group-like elements, 
\[ 
\Delta h = h\otimes h .
\]
Let $h_{m+1}$ be the part of grade $m+1$ of $h$.
Then  identifying the terms of grade  $m+1$ in the last identity we get
\[ 
\Delta h_{m+1} = 1 \otimes h_{m+1} + h_{m+1} \otimes 1 
\] 
Thus $h_{m+1} \in P$ 
and 
\[ 
h - \exp (1+ h_{m+1}) \in I_{m+1}. 
\] 
\end{proof} 

We are now in a position to complete the proof of Theorem \ref{lnT}.
By Lemma \ref{TG} we have $T \in G$. Theorem \ref{logarithm} ensures
that there is a primitive element $h$ so that $\red{T} = \exp h$. Checking
the formal power series shows that $h$ is unique. Its homogeneous
parts are linear combinations of connected symbols. Thus Theorem
\ref{lnT} follows.

\newsection{The spaces $U^p$  and $V^p$} 
\label{a:uv}

\setcounter{equation}{0}\setcounter{theorem}{0} 

The spaces $U^2$ and $V^2$, respectively  $DU^2$ and $DV^2$, are crucially 
used in this article as substitutes for $\dot H^{\frac12}$, respectively 
the scale invariant space $\dot H^{-\frac12}$. The latter spaces are  plagued by
the failure of Sobolev embeddings and the failure of a good regularity
theory when integration against those functions. Instead, the former spaces 
are close neighbors, and have both the same scaling and good multiplicative 
properties.  

The aim of this section is set up the functional context for the estimates 
in this paper, and in particular to clarify the relation between the pointwise 
definition of $U^p$ and $V^p$, and their interpretation as distributions.
We first recall the definition of the spaces $U^p$ and $V^p$. 

We will consider functions $u, v$ defined on the open interval $(a,b)$ and
set always $u(a) = 0$ and $v(b)=0$, even if $a=-\infty$ and/or $b=\infty$.  All  functions  in this
section are ruled functions i.e. functions which have left and right
limits everywhere play a central role. If $X$ is a suitable function space we denote by $X_{rc}$ resp. $X_{lc}$ 
the subset of right (left) continuous functions with limit $0$ at the left
(right)  endpoint.

Let $-\infty \le a<b \le \infty$. A partition $\tau$ 
is a finite monotone sequence 
\[  a< t_1 < \dots t_N = b. \]
We denote the set of partitions by $\mathcal{T}$. There is the obvious
notion of  refinements of partitions.  Step functions are functions associated to a partition
which are constant on the open intervals between points of the
partition.  We denote by $\step$ the space of step functions, and by 
$\step(\tau)$ the space of step functions associated to a given partition $\tau$.

 \begin{definition} 
a) Let $1<p<\infty$.  We define the space $V^p = V^p(a,b)$ as the space of those functions
in $(a,b)$ for which the following norm is finite:
\[ 
\Vert v \Vert_{V^p} = \sup_{\tau \in \mathcal{T}}  \left( \sum_{j} |v(t_{j+1})-v(t_j)|^p\right)^{\frac1p} 
\] 
where we set $v(b)=0$. 

b) A $U^p$ atom is a step function
\[ 
u = \sum_j  \phi_j \chi_{[t_j,t_{j+1})} (x)   \qquad \text{ if  } \sum |\phi_j|^p \le 1. 
\] 
A $U^p(a,b)$ function is an $l^1$ sum of atoms,
\red{
  \[
  U^p(a,b) = \{ \sum_{j} \lambda_j a_j : (\lambda_j)\in l^1, a_j % \text{ atom }
  \}  .
  \]
}
We equip $U^p$ with the obvious norm
\red{
  \[ \Vert u \Vert_{U^p(a,b)}= \inf \{  \sum_j |\lambda_j|: u=\sum_j \lambda_j a_j,
     a_j \text{$U^p$  atom } \} .
  \]
}
\end{definition}

The spaces $U^p$ and $V^p$ are invariant under continuous monotone
coordinate changes and we suppress the interval in the notation unless
we specifically need it. We note that by this definition the $U^p$
functions are right continuous and have limit $0$  at the left end point $a$, whereas $V^p$
functions are just ruled, and vanish by definition at $b$.  The supremum norm
is bounded by the $V^p$ norm (taking the partition $\{t,b\}$) and the
$U^p$ norm (this is checked on atoms). Moreover, if $p< q$
\[ 
U^p \subset U^q, \qquad V^p \subset V^q  
\]
with norm estimates with constant $1$ and $U^p \subset V^p$ with 
\[ 
\Vert u \Vert_{V^p} \le 2^{\frac1p} \Vert u \Vert_{U^p}. 
\]
It is not hard to see but less obvious
that the $V^p$ functions are ruled.

There is a natural pairing between ruled functions $v$ and right continuous 
step functions $u$ with $u(a)=0$ (and the notation $t_0=a$) 
\[  
 B(v,u) = \sum_{j=1}^{N-1}  v(t_j) (u(t_j)-u(t_{j-1}) )  = \sum_{j=1}^{N-1}  u(t_j) (v(t_j)-v(t_{j+1}) ) 
\] 
where the  sum runs over the points  of the partition and $t_0 = a$. The two sums are  equal since $u(a)=0$ and $v(b)=0$.
By an abuse of notation we use the suggestive notation
\[ 
B(v,u) = \int v du. 
\] 
Suppose now that $v$ is a left continuous step function associated 
to a different partition $\sigma= \{s_j\}$. Then a simple algebraic 
computation shows that 
\[
B(v,u) = \sum_{j=1}^{N-1}  v(s_j) (u(s_j)-u(s_{j-1}) )  = \sum_{j=1}^{N-1}  u(s_j) (v(s_j)-v(s_{j+1}) ) 
\]

Let $\dfrac1p+\dfrac1q=1$ and let $u$ be an $U^q$ atom and $v \in V^p$. 
By definition
\[
 |B(v,u)| \le \Vert v \Vert_{V^p}.
\]
It is not surprising but not entirely obvious that  $B$ extends to a bilinear map from $V^p \times U^q \to \R$ with
\begin{equation}\label{Buv}
 |B(v,u)| \le \Vert v \Vert_{V^p} \Vert u \Vert_{U^q},  
\end{equation}
see \cite{MR2526409}.  
Thus $B$ induces a map from $V^p \to (U^q)^*$.
Further, we have 

\begin{proposition}\label{p:up-dual}
 The bilinear form $B$ induces an isometric isomorphism from $V^p \to (U^q)^*$. 
\end{proposition}

\begin{proof}
We briefly sketch the proof and refer to \cite{MR2526409} for more details.
We first observe that 
\[
 \Vert v \Vert_{V^p} = \sup_{\| u\|_{U^q} \leq 1} B(u,v),
\]
which is easily seen by restricting $u$ to the class of $U^p$ atoms. Therefore,
 $B$ induces an isometry from $V^p \to (U^q)^*$. To see that this map is in effect
an isomorphism,  let $L: U^q \to \R$ be a linear functional. To it we associate the function 
\[
   v(t) = L(\chi([t,b) ).  
\] 
 Then 
\[  
B(v,\chi(t,b)) = v(t) = L(\chi([t,b))) 
\] 
and hence the same identity holds it we replace the characteristic function by 
right continuous step functions, which are dense. 
\end{proof}

\begin{remark} \label{lr} By symmetry, we can use the same bilinear
  form $B$ to also pair left continuous $U^q$ functions $v$ with ruled
  $V^p$ functions $u$ with the same boundary conditions.
\end{remark}

The duality implies characterizations if the  $V^p$ norm:
\[ \Vert v \Vert_{V^p} = \sup \{ B(v,u) : \Vert u \Vert_{U^p} = 1\} \]
It is immediate that we even may restrict the supremum to step functions, and even atoms. Moreover, of the supremum over atoms on the right hand side is finite 
then $v \in V^p$. Similarly 
\[ \Vert u \Vert_{U^q} = \sup \{ B(v,u): \Vert v \Vert_{V^p} =1\}  \] 
and again we may restrict the supremum on the right hand side to step
functions. We will see that right continuous functions for which the
right hand side is bounded when we take the supremum over step
functions are indeed in $U^q$.  But this time the proof is highly
nontrivial.

\begin{theorem}\label{Upbound} 
  Suppose the right continuous ruled function $u$ with    $\lim_{t\to a} u(t) = 0 $       satisfies 
\[ 
\sup   |B(v,u)|  <\infty 
\] 
where the supremum is taken over all  step functions in $V^p_{lc}$
(left continuous step functions in $V^p$) with norm at most $1$ and compact support in $(a,b)$.
Then $u \in U^q$ and  
\[
 \Vert u \Vert_{U^q} = \sup B(v,u). 
\] 
\end{theorem}   

Before we turn to the proof we study the behavior of 
these spaces under decomposition of the interval. Here we 
include atoms with $t_1=a$ and denote those spaces by $\tilde U^p$. 

\begin{lemma}\label{updecomp} The following inequality holds for all functions in $\tilde U^p$.
\[ \Vert u \Vert_{\tilde U^p(T_0,T_1)}^p \le \Vert u \Vert_{\tilde U^p(T_0,t)}^p 
+ \Vert u \Vert_{\tilde U^p(t,T_1)}^p \] 
\end{lemma}

\begin{proof}  

Let $b$ and $c$ be $p$ atoms and $\lambda,\mu \in \mathbb{C}$. Then 
\[  a=  \frac{\lambda}{(|\lambda|^p +|\mu|^p)^{\frac1p} }  b
+ \frac{\mu}{(|\lambda|^p+|\mu|^p)^{\frac1p}} c \] 
is a $p$ atom on $U^p(T_1,T_2)$. Now let 
\[ u = \sum \lambda_j b_j,   v= \sum \mu_j c_j \] 
and $L = \sum |\lambda_j|\le C  \Vert u \Vert_{U^p} $, $\Mu = |\mu_j |
\le C \Vert v \Vert_{U^p}$. We can find such a decomposition by definition 
for every $C>1$. Then (with an extension by $0$)  
\[ \sum_j \lambda_j a_j + \sum_l \mu_k c_l 
= \sum_{jl} \frac{\lambda_j |\mu_l|}{\Mu} a_j + \frac{|\lambda_j|\mu_l}{L} c_l \] 
and hence 
\[ 
\begin{split} 
\Vert u+v \Vert_{U^p} 
\le &  \sum_{j,l} |\lambda_j||\mu_l| \left\Vert \frac{\lambda_j}{|\lambda_j|} \frac{a_j}{\Mu} + \frac{\mu_l}{|\mu_l|} \frac{c_l}{L} \right\Vert_{U^p} 
\\ \le & L \Mu \left( \Mu^{-p} + L^{-p}\right)^{1/9} 
\\ = & (L^p+\Mu^p)^{\frac1p} 
\\ \le & C^{2p}\left( \Vert u \Vert_{U^p}^p +\Vert v \Vert_{U^p}^p \right)^{\frac1p} 
\end{split}   
\] 
This inequalities holds for all $C>1$ and hence also for $C=1$.
\end{proof}

There is an analogous statement for functions in $V^p$.

\begin{lemma} \label{l:vp-sum}
Let $ v \in V^p$, $\tau$ be a partition and suppose that $v$ vanishes at the points of the partition. Then 
\[ \Vert v \Vert_{V^p(T_0,T_1)}^p \le   2^{p-1}   \sum \Vert v \Vert^p_{V^p(t_j,t_{j+1})} \]
\end{lemma} 

\begin{proof} Let $\tilde \tau$ be another partition and 
  \[ I = \sum |v(\tilde t_{j+1})-v(\tilde t_j)|^p \] We may assume
  that $\tilde \tau$ contains a point in all the intervals of
  $\tau$. Otherwise we omit that interval from $\tau$.  Taking the
  coarsest joint refinement leads to the second factor on the right
  hand side. It may decrease the sum at most by the factor $2^{1-p}$.
\end{proof} 

To left, respectively right continuous functions and partitions we associate step 
functions as follows:

\begin{definition} Given a left continuous function  $v$ and a partition $\tau$ we define 
an associated left continuous step function $v_\tau$ so that
\[  
v_{\tau}(t) = v(t) 
\] 
for each point of the partition. Similarly, for each right continuous function $v$ 
we define an associated right continuous step function $v_\tau$.
\end{definition} 

Then the following lemma is straightforward:

\begin{lemma} Suppose that $v$ is a left continuous step function with
  partition $\tau$. Then for all right continuous functions $u$ we
  have
\[ 
B(v,u) = B(v,u_{\tau}) 
\] 
Symmetrically, if $u$ is a right continuous step function then for all
left continuous functions $B$ we have
\[
 B(v,u) = B(v_\tau,u). 
\]
Moreover 
\[
 \Vert v_\tau \Vert_{V^p} \le \Vert v \Vert_{V^p} 
\] 
and 
\[
 \Vert u_\tau \Vert_{U^p} \le \Vert u \Vert_{U^p}. 
\]
\end{lemma}

\begin{lemma} \label{approx}
Let $v_j$ be a sequence of step functions with 
\[   \Vert v_j \Vert_{V^p} \le 1 \qquad \text{ and } 
\Vert v_{j} \Vert_{sup} \to 0. \]
Then there exists a subsequence $v_{j_l}$  so that 
\[ \left(\Big\Vert \sum_{l=1}^m v_{j_l} \Big\Vert_{V^p}^p\right)^{\frac1p}   \le  (2m)^{\frac1p}. \]
\end{lemma} 

\begin{proof} 
We construct the subsequence and a partition recursively. Suppose we have defined $j_l$ for $l\le m$ so    so that $v^m= \sum_{l=1}^m v_{j_l}$ satisfies $\Vert v^m \Vert_{V^p}  \le (2m)^{\frac1p}$. 
We search for $j_{m+1}$. Let $\tau$ be a partition. Then 
\[ \sum |(v^m+v_j)(t_{k+1}) - (v^m+v_j)(t_k)|^p 
\le   \sum (|v^m(t_{k_{l+1}})-v^m(t_{k_l})|+ 2 \Vert v_j \Vert_{sup} )^p 
+ \sum |v_j(t_{k+1})-v_j(t_k)|^p 
\] 

where the first sum runs only over those indices for which $v^m$ has a jump. 
This implies the statement if we choose $j$ sufficiently large.            
\end{proof}

\begin{proof}[Proof of Theorem  \ref{Upbound}] 
Let 
\[ 
\Vert u \Vert_{\tilde U^p} = \sup \{ B(v,u) : v \in \step_{lc}, \Vert v\Vert_{V^p}=1\} 
\]
and let $\tilde U^p$ be the Banach space of all ruled, right continuous
functions for which this norm is finite. We
have to show that $\tilde U^p= U^p$. By the duality property 
$(U^p)^* = V^q$ we know that $U^p \subset \tilde U^p$, with 
the norm bound  
\[
\| u\|_{\tilde U^p} \lesssim \|u\|_{U^p}
\]
Further, we claim that equality holds. This is because step functions
are dense in $U^p$, and for step functions it suffices to test them with other 
step functions.  It remains to prove the converse inclusion. We begin with an elementary but crucial observation.

\begin{lemma} Let $ u \in \tilde U^p(a,b)$. There exists a step function 
$v \in V^q_{lc}$ with $\Vert v \Vert_{V^q}=1$   
with compact support in $(a,b)$ so that 
\[  B(v,u) \ge \frac14 \sup\{  B(w,u) : w \text{ is a $V^p_{lc}$ 
step function with } \Vert w \Vert_{V^p}=1 \}. \]
\end{lemma} 

\begin{proof} 
There exists a step function $v_0 \in V^q_{lc}$ of norm $1$ so that 
\begin{equation}  B(v_0,u) \ge \frac34 \sup\{  B(w,u) : w \in V^q_{lc}:  \Vert w \Vert_{V^p}=1 \}. \label{compact} \end{equation} 
Let $\tau=(a,t_1 \dots t_{N-1} , b) $ be the points of the partition associated to $v_0$. By the left  continuity condition,
$v_0$ already vanishes near $t = b$. 
We add one extra point, $t_0 \in (a,t_1)$  and define 
$v_1$ to be $0$ left of $t_0$  and $v_0 $ in $[t_0,b]$.
Then 
\[
\begin{split} 
 B(u,v_0-v_1) =&  - v_0(t_1) B(u, \chi_{a,t_0}) = - v_0(t_0) u(t_0)  
\end{split}  
\] 
By  right continuity 
\[
\lim_{t_0 \to a} u(t_0) = 0
\]
so we can choose $t_0$  so that 
\eqref{compact} holds for $v= v_1$. Since $\Vert v_1 \Vert_{V^q} \le 2^{\frac1q}$
this implies the result. 
\end{proof} 

 Assume by contradiction that  $U^p \not\subset \tilde U^p$.
Then there exists $u \in \tilde U^p$ so that 
\begin{equation}\label{tup-far}
\| u\|_{\tilde U^p} = 1, \qquad \dist(u,U^p) > \frac12.
\end{equation}
Exactly one of the following two alternatives holds for this function $u$:

\begin{enumerate}
\item There exists $\varepsilon>0$  and $t \in [a,b]$  so that for any interval  $(t_0,t_1)$ 
containing $t$ we have
\[ \Vert u-u(t) \Vert_{\tilde U^p(t,t_1)} + \Vert u-u(t_0) \Vert_{\tilde U^p(t_0,t)}\ge \varepsilon \]
with the obvious modification at the endpoints. 
\item For every $\varepsilon>0$ and every $t \in [a,b]$ there exists a 
neighborhood $(t_0,t_1)$ of $t$ so that 
\[ \Vert u-u(t) \Vert_{\tilde U^p(t,t_1)} + \Vert u-u(t_0) \Vert_{\tilde U^p(t_0,t)}<  \varepsilon \]
with the obvious modification at the endpoints. 
\end{enumerate} 
 We will show that neither alternative is true, thereby reaching
a contradiction. 

\bigskip

Suppose that the first alternative holds. Then we must have either
\begin{equation} \label{right}   
\Vert u-u(t) \Vert_{\tilde U^p(t,t_1)} \ge \varepsilon/2  \qquad \text{ for all } t_1 > t
\end{equation} 
or 
\[ 
 \Vert u-u(t_0) \Vert_{\tilde U^p(t_0,t)} \ge \varepsilon/2 \qquad \text{ for all } t_0 < t
\]
In either case we claim that there exists a sequence of step functions
$v_j$ in $V^q$ with disjoint support so that
\[ 
B(u,v_j) \ge \varepsilon/8, \qquad  \Vert v_j \Vert_{V^p} = 1
\] 
Without loss of generality let us assume \eqref{right}. By Lemma \ref{compact} 
there exists $v_1 \in \step_{lc}(t,t_1)$, compactly supported, so that 
\[ 
B(u-u(t_0),v_1) = B(u,v_1) \ge \varepsilon/8  
\] 
Let $t_2$ be the first point of the partition of $v_1$ to the right of
$t$. Then $v_1$ vanishes in $(t,t_1)$. We repeat the above argument in
$(t,t_2)$ to produce a function $v_2 \in \step_{lc}(t,t_1)$, etc.  Let
\[ v = \sum_{j=1}^N  v_j. \]
 Then 
\[ B(u,v) \ge  N \frac{ \varepsilon}8,
 \] 
while by Lemma~\ref{l:vp-sum} we have
\[ 
\Vert v \Vert_{V^q} \le 2^{\frac1q} N^{\frac1p} 
\] 
Hence
\[
 N \frac{\varepsilon}8 \le (2N)^{\frac1q}. 
\] 
This cannot be true for $N$ large. This is a contradiction, which shows that  
the first alternative cannot hold for any function $u \in \tilde U^p$.

It remains to disprove the second alternative. Here we will use the fact that $u$ satisfies 
\eqref{tup-far}. For each $\epsilon > 0$ we can cover the interval $[a,b]$ with intervals
\[
[a,b] = \bigcup_{t\in [a,b]} I_t 
\]
where  $I_t$ are the intervals given by the second alternative.  By compactness,
it follows that there is a finite subcovering,
 \[
[a,b] = \bigcup_{k=1}^N  I_{t_k} 
\]
Then we can find a partition $\tau=\tau(\varepsilon)$ of $[a,b]$ with the property that 
for each interval $I_j(\tau) = [t_j,t_{j+1}]$  in the partition we have 
\begin{equation}\label{small-up}
\| u - u(t^\varepsilon_j)\|_{\tilde U^p(I_j(\tau))} \leq \epsilon
\end{equation}

We now consider the right continuous step function $u_\tau$ which
matches $u$ at the points of the partition $\tau$.  The function
$u-u_\tau$ vanishes on the partition $\tau$, so it is natural to split
it with respect to the partition intervals $I_j(\tau)$,
\[
u - u_\tau = \sum u_j, \qquad u_j = 1_{I_j(\tau)} (u - u_\tau)
\]

On one hand, by \eqref{tup-far} we have
\begin{equation} \label{far}
\| u - u_\tau\|_{\tilde U^p} > \frac12
\end{equation}
On the other hand, as a consequence of  Lemma \ref{updecomp} we have
\begin{equation} \label{tup=split} 
\Vert u-u_\tau \Vert_{\tilde U^p}^p 
\le \sum 2 \Vert u_j \Vert_{\tilde U^p}^p. 
\end{equation} 
For each $u_j$ we can find a corresponding function $v_j \in \step_{lc}$, with similar support,
so that  
\[
B(u_j,v_j) \geq \frac14 \|u_j\|_{\tilde U^p} , \qquad \|v_j\|_{V^p} = 1.
\]
Combining these functions with appropriate weights
we produce the step function $w \in \step_{lc}$,
\[
w = \left(\sum_j  \|u_j\|_{\tilde U^p}^{p}\right)^{-1}  \sum_j  \|u_j\|_{\tilde U^p}^{p-1} v_j  
\]
These will have the following three properties:

\begin{enumerate} 
\item  Boundedness in $V^q$. Indeed, by Lemma~\ref{l:vp-sum} we have
\[
\| w\|_{V^q}^q \leq 2^{q-1}  \left(\sum_j  \|u_j\|_{\tilde U^p}^{p}\right)^{-1}    \sum_j   \|u_j\|_{\tilde U^p}^p    \|v_j\|_{V^q}^q \leq 2^{q-1}.
\]
\item   Large $B(w,u)$. Here we compute
\[
B(w,u) = B(w, u-u_\tau) = \sum B(v_j,u_j) \geq    \sum_j  \|u_j\|_{\tilde U^p}^{p} \geq 2^{-p}
\]
\item  Small pointwise norm. For each $v_j$ we can bound the pointwise norm by $1$,
and they are all nonoverlapping. Thus using \eqref{small-up} we obtain 
\[
\|w\|_{L^\infty} \leq  \sup_j  \|u_j\|_{\tilde U^p}^{p-1}   \left(\sum_j  \|u_j\|_{\tilde U^p}^{p}\right)^{-1} 
\leq \frac14  \sup_j  \|u_j\|_{\tilde U^p}^{p-1}.
\]
\end{enumerate}

Taking a sequence $\epsilon_j \to 0$, this construction yields a sequence of partitions $\tau_j$ 
and associated left continuous step functions $w_j$ with the following properties: 
\begin{enumerate} 
\item $\| w_j\|_{V^p} \lesssim 1$.
\item  $B(u,w_j) \gtrsim 1$.
\item $\|w_{j}\|_{L^\infty} \to 0 $.
\end{enumerate}
To conclude we argue as in the first alternative, this time using
Lemma \ref{approx} to reach a contradiction.
 \end{proof}

Let
\[ V^p_C = \{ v \in V^p\cap C := V^p_{rc} \cap V^p_{lc}: \lim_{t\to a,b} v(t) = 0 \} \] 

\begin{theorem}\label{weak*}  $V^p_C$ is weak* dense in $V^p$. Moreover, if $v,w \in V^p$
  then there exists a sequence $v_j \in V^p_C$ so that $v_jw \to vw$
  in $V^p$ in the weak* sense.
\end{theorem} 

\begin{proof} We claim that $V_{lc}^p \subset V^p_C$ is weak * dense. 
Let $v \in V^p_C$. There exists an at most countable set of times $t_j$ 
at which $v$ is not left continuous. We include $b$ if $\lim_{t\to b} v(t) \ne 0$. At each point $t_j$ there is a strictly increasing sequence $t_{j,k} $ 
converging to $t_j$ such that either $t_j $ is contained in one of the intervals \[  [t_{l,k},t_l] \] 
for some $l<j$, or $[t_{j,k},t_k]$ is disjoint from all those intervals. 
 For each $k$ we recursively modify $v$ in $(t_{j,k},t_j)$ 
so that in that interval 
\[  v_k(t) =  \frac{t_{j}-t}{t_{j}-t_{j,k}}   v(t_{j,k}) +  \frac{t-t_{j,k}}{t_j-t_{j,k}}   v(t_j) \]
if $t_j$ is not contained in the previous intervals, with a simple modification at $b$. 
Then it is not hard to see that 
\[ \Vert v_k \Vert_{V^p} \le \Vert v \Vert_{V^p} \] 
and 
\[  B(v_k,u) \to B(v,u) \] 
for all $u \in U^q$.

If $a=-\infty$ and $b= \infty$ (which we assume without loss of generality) 
and if $v \in V^p_{lc}$ and if $a$ is a right continuous step function then 
\[  \lim_{h>0, h\to  0}  B(v(.+h),a) = B(v,a). \]                     
Using the atomic decomposition we see that this convergence is true for 
$u \in U^q$ instead of $a$.

We pick a smooth function $\rho$ with support in $(-1,0)$ and integral $1$ and define for $v \in V^p_{lc}$ and $h \in (0,1)$ 
\[ v_h (t) = \int_0^1 v(t+hs)\rho(s) ds \]   
Then again as $h\to 0$ 
\[ B(v_h,u) \to B(v,u). \] 
Since $ u(t) \to 0 $ and $t\to a$ we can truncate on the left and obtain functions in $C^\infty_0$ converging to $v$ in the weak* sense. 
The second claim is proven similarly. 
\end{proof}

\begin{theorem} \label{t:vp-dual}
The map 
\[ u \to (v \to B(v,u) ) \]
is an isometric isomorphism from $U^q$ to $(V^p_C)^* $. 
\end{theorem} 

\begin{proof} 
We first observe that, in view of the bound \eqref{Buv}, the above map is 
bounded,
\[
\| u\|_{(V^p_C)^*}= \sup \{ B(v,u) : v \in V^p_C , \Vert v \Vert_{V^p} =1 \} = \|u\|_{U^q}
\]
The last equality is a consequence of Theorem  \ref{weak*}. 
  It remains to show that this 
map is onto.

Start with  $L \in (V^p_C)^*$. By the theorem of Hahn Banach is has a same norm  extension 
to the left continuous functions in $V^p_{lc} $ which we denote again by $L$.
By the embedding $U^p_{lc} \subset V^p_{rc}$, the map 
\[ 
 v \to L( v ) \in (U^p_{lc})^* 
\] 
has a unique representative $ u  \in V^{q}$. In particular this means that for 
any left continuous step function $v$ we have
\[
L(v) = B(v,u)
\]
Here we have yet no continuity condition on $u$. However, if we replace it 
by its right continuous version $u_{rc}$, then we have for $v \in \step_{lc}$ 
by an abuse of notation 
\[
\ B(v,u_{rc})= \lim_{h >0, h\to 0}  B(v(.-h),u). 
\]
and hence 
\[ \sup\{  B(v,u_{rc}) : v \in \step_{lc}, \Vert v \Vert_{V^p} \le 1\} 
\le   \sup\{  B(v,u) : v \in \step_{lc}, \Vert v \Vert_{V^p} \le 1\}. \] 
Now we are in a position to apply Theorem~\ref{Upbound}  to conclude that
$u_{rc} \in U^q$, and 
\[
\| u_{rc} \|_{U^q} \leq \|L\|
\]
To conclude the proof, we note that for $v \in V^p_C$ we have 
\[
B(v,u_{rc}) = B(v,u) = Lv. 
\]
Thus $u_{rc} \in U^q$ is a representation of $L$ via the bilinear form $B$,
and 
\[
\| L \| \leq \| u_{rc} \|_{U^q} \leq \|L\|
\]
therefore equality must hold. The proof is complete. 
\end{proof}

\begin{lemma} $C^\infty_0$ is dense in $V^p_C$ and weak* dense in $U^p$  and $V^p$. 
\end{lemma} 

\begin{proof} The first statement is simple. For the second it
  suffices to approximate atoms with compact support. Then any
  standard regularization gives the statement.
\end{proof}

Now we turn our attention to bilinear estimates. We begin with the algebra type 
properties:

\begin{lemma} $U^p$ and $V^p$ are algebras, and the following estimates hold:
\begin{equation} \Vert vw \Vert_{V^p} \le \Vert v \Vert_{sup} \Vert w \Vert_{V^p} 
+ \Vert v \Vert_{V^p} \Vert w \Vert_{sup} 
\end{equation}
\begin{equation} \Vert uw \Vert_{U^p} \le 2\Vert u \Vert_{U^p} \Vert w \Vert_{U^p} 
\end{equation} 
\end{lemma} 

\begin{proof} The first part is obvious.  
For the second part we have
  to consider a product of two atoms.
\end{proof}

We are now in a position to define a Stieltjes type integral. 

\begin{definition}\label{stieltjes}  Let $ v\in V^p$ and $u \in U^q$ with $\frac1p+\frac1q=1$.
We define 
\[U=  \int_a^t v(s) du(s) \] 
by 
\begin{equation} \label{vdu}   B(w,U)= B(wv,u) \qquad \text{ for all } w \in V^p_C \end{equation} 
and 
\[ V = - \int_t^b u(s) dv(s) \] 
by 
\[  B(V, w) = B(v,wu) \qquad \text{ for all } w \in U^q\] 
\end{definition} 

The right hand side defines a continuous linear map 
\[  V^p_C \ni w \to B(wv,u) \]
resp
\[ U^q\ni w \to B(v,wu) \]
which by duality has a unique representative. Moreover by Theorem 
\ref{weak*} \eqref{vdu} holds for all $w \in V^p$. In particular

\[ \int_a^t v_1(s) d(\int_a^s v_2(\sigma) du(\sigma)) 
= \int_a^t v_1(s) v_2(s) du(s). \] 
To see this, define $U = \int_a^t v_2(s) du(s)$. Then almost by definition
\[ B(v_1v_2w,u) = B(v_1w, U)  \]
Similarly, if $u_1,u_2 \in U^q$ and $V = -\int_t^b u_2(s) dv(s)$ then 
\[ B(v,w u_1u_2) = B(V,wu_1 ) \] 
and hence 
\[ -\int_t^b u_1u_2 dv = - \int_t^b u_1 \int_s^b u_2 dv. \]   

For PDE applications we are also interested in the distributional interpretation 
of $U^p$ and $V^p$ functions. This is obvious for $U^p$, as any $U^p$ function 
is bounded and uniquely determined by its values a.e.. However, in the case 
of $V^p$ we have allowed single point jumps, which are invisible when testing against 
$C_0^\infty$  test functions. Thus for this purpose it is more natural to work with 
the smaller space $V^p_{lc}$, whose elements are uniquely identified with distributions.

Next we consider the space of distributional derivatives of such functions:

\begin{definition}\label{B16} 
We define 
\[
DU^p = \{ u';\ u \in U^p \}, \qquad  DV^p = \{ v';\ v \in V^p_{lc} \}
\]
with the induced norm.
\end{definition}

There is no difference in the definition of $DV^p$ if we replace $V^p_{lc}$ 
by $V^p_{rc}$ and indeed the second one is more natural for solving  Cauchy 
problems.  

Based on our duality results above, we have the following dual 
characterization of these spaces:

\begin{proposition} 
The spaces $DU^p$ and $DV^p$ can be characterized as the spaces 
of distributions for which the following norms are finite:
\begin{equation} \label{dup}
\Vert f \Vert_{DU^p} = \sup\left\{ \int f \phi dt : \Vert \phi \Vert_{V^q} \le 1,
\phi \in C^\infty_0 \right\} 
\end{equation}
\begin{equation}\label{dvp}
\Vert f \Vert_{DV^p} = \sup\left\{ \int f \phi dt : \Vert \phi \Vert_{U^q} \le 1,
\phi \in C^\infty_0 \right\} 
\end{equation}

for all distributions $f$.
\end{proposition} 

\begin{proof}
At the heart of the proof is a very simple observation, namely that 
for $\phi \in C_0^\infty$ we have
\[
B(v,\phi) = \int v \phi' dt, v \in V^p
\]
respectively 
\[
B(\phi,u) = - \int u \phi' dt
\]
This is easily verified directly for step functions, then by density for $U^p$ 
functions, and by $V^p \subset U^q$ embeddings for $V^p$ functions.

Because of this and the bound \eqref{Buv}, it is clear that for $f \in DU^p$ we have
\[
\left|\int f \phi dt \right| \leq \|f\|_{DU^p} \|\phi\|_{V^p}
\]
Conversely, let $f$ be a distribution for which the norm on the right
in \eqref{dvp} is finite.  Then by the Hahn-Banach theorem (or by
density) $f$ extends to a bounded linear functional on $V^p_C$ with
the same norm. Hence by Theorem~\ref{t:vp-dual} there exists $u \in U^q$,
with the same norm, so that
\[
B(u,\phi) = \int f\phi dt \qquad \text{ for all } \phi \in C_0^\infty
\]
Thus 
\[
 - \int u \phi' dt =   \int f\phi dt
\]
i.e. $f = u'$ in the sense of distributions. 

The argument is similar for $DV^p$, using instead the duality in
Proposition~\ref{p:up-dual}.
\end{proof}

Next we consider products of $U^p$ and $V^p$ functions with $DU^p$,
and $DV^p$ functions. Definition \ref{stieltjes} can be understood as
a version of such product. By an abuse of notation we write the
bilinear estimates as
\begin{equation}  \Vert vu' \Vert_{DU^p} \le c \Vert v \Vert_{V^q} \Vert u' \Vert_{DU^p} \end{equation}
and 
\begin{equation}  \Vert u v'  \Vert_{DV^p} \le c \Vert u \Vert_{U^q} \Vert v' \Vert_{DV^p}. \end{equation}

This is an abuse of notation since the product depends on whether the
derivative is in $DU^p$ or in $DV^p$.

There is an interpolation inequality.

\begin{theorem}\label{interpol}  Let $1<p<q< \infty$. There exists $C>0$ so that for  $v \in V^p_{rc}$ and $M>1$ there
exist $u \in U^p$ and $w \in U^q$ so that $v=u+w$ and 
\[ \frac1M \Vert u \Vert_{U^p} + e^{M} \Vert w \Vert_{U^q} \le c \Vert v \Vert_{V^p}. \]
In particular $V^p_{rc} \subset U^q$.
\end{theorem} 

\begin{proof} We may assume that $\Vert v \Vert_{V^p} \le 1$. 
Let 
\[ \omega_p(f,t) = \sup_{t_1< t_2 < t_n \le t} \sum |v(t_{j+1})-v(t_j) |. \]
Then $\omega$ is monotonically nondecreasing and $\omega(t) \to  0 $ as $t$ 
tends to the left limit of the interval. 
We define 
\[  t_{jk} = \inf \{ t : \omega(t)> j2^{-k}   \]
and 
\[ u_k (t) = \left\{  \begin{array}{ll} 0 & \text{ if } t_{k,2j} \le t < t_{k,2j+1} \\
u(t_{k,2j+1}) - u(t_{k,2j})&  \text{ if } t_{k,2j+1} \le t < t_{k,2j+2}
\end{array} \right. \] 
Then
\[ \Vert u_k \Vert_{U^r} \le 2^{k (\frac1r-\frac1p)}  \] 
and 
\[ \Vert \sum_{k=1}^N u_k\Vert_{U^p} \le N \] 
and
\[ \Vert \sum_{k=N+1}^\infty u_k \Vert_{U^p} \le   \frac{2^{(N+1)(\frac1q-\frac1p})}{1-2^{(\frac1q-\frac1p})}.\] 
We choose 
\[ N = \frac{M}{\ln 2(\frac1q-\frac1p)} \] 
and $C$ large. 
\end{proof} 

\begin{cor} Suppose that $p <q$. Then $V^p_{rc} \subset U^q$, 
\end{cor} 
\begin{proof} We apply the Theorem \ref{interpol} with $M=1$. \end{proof}

We want to relate the spaces to standard Besov spaces. This relies on the following lemma. 

\begin{lemma}\label{dist}  Let $h>0$. Then 
\[ \Vert v(.+h) -v \Vert_{L^p(\R)} \le c h^{\frac1p} \Vert v \Vert_{V^p}. \] 
\end{lemma} 

\begin{proof} By scaling we may assume that $h=1$. For $ t \in [j,j+1)$ 
\[ |v(t+1)-v(t)| \le \sup_{j \le t_1,t_2  \le j+2} v(t_2)-v(t_1). \]
In particular there exists  sequences $(t_{j,1})$ and $(t_{j,2})$  in $[j,j+2]$ 
so that 
\[  \sup_{j \le t_1,t_2  \le j+2} (v(t_2)-v(t_1)) \le (1+\varepsilon) (v(t_{j,2}) -v(t_{j,1})) \] 
Taking subsequences one see that 
\[ \left( \sum_j |v(t_{j,2})-v(t_{j,1})|^p \right)^{\frac1p} \le 4^{\frac1p} \Vert v \Vert_{V^p}.   \] 
We integrate over the unit sized interval and sum with respect to $j$ to obtain the estimate. 
\end{proof} 

\begin{cor}\label{appBemb}  The following embedding estimates hold. 
\begin{equation} \Vert v \Vert_{B^{\frac1p}_{p,\infty}} \le c \Vert v \Vert_{V^p} \end{equation} 
\begin{equation} \Vert u \Vert_{U^p} \le c \Vert u \Vert_{B^{\frac1p}_{p,1}} \end{equation} 
\end{cor} 

\begin{proof} The first inequality is an immediate consequence of 
\eqref{dist}. The second one follows by duality. 
\end{proof} 

We conclude with a density statement. 

\begin{theorem} Step functions are dense in $U^p$ but not in $V^p$. \end{theorem} 

\begin{proof} The first statement is an immediate consequence of the atomic definition. Let $\eta \in C^{\infty}_0(a,b)$. We define 
\[ v  = \eta \sum_{j=0}^\infty 2^{-\frac{j}p}  \sin(2^j x). \]
It is not hard to see that $v \in C^{\frac1p}$ and hence $v \in V^p$, but it is not close to any step function.  
\end{proof}

We conclude the section with the proof of Lemma~\ref{characterization}.

\begin{proof}[Proof of Lemma~\ref{characterization} ]
First 
\[\Vert u \Vert_{l^2_1 DU^2} \le \Vert u \Vert_{l^2L^2} = \Vert u \Vert_{L^2} \]
hence $L^2 \subset l^2_1 DU^2$. 
Now let $u \in DU^2$ and define 
\[   
a = \sum_{j \in \Z} a_j \delta_j (t), \qquad a_j = \int_{j}^{j+1} u dt
\] 
where $\delta_j$ is the Dirac measure at the point $j$. This is well defined 
since $\chi_{[j,j+1)} \in V^2$. Moreover, since $U^2\subset V^2$,  we have the  
bound
\[ 
\sum_j  a_j^2  \le \Vert u \Vert_{DU^2}^2. 
\] 
In particular $a \in l^2_1 DU^2$. On the other hand, 
we can write 
\[
u - a = \sum_j \chi_{[j,j+1)} - a_j \delta_j
\]
where each term is supported in $[j,j+1]$ and has integral zero.
Thus by Lemma~\ref{updecomp} we  can bound   
\[ \Vert u-a \Vert_{l^2_1 DU^2} \le  \Vert u \Vert_{DU^2}. \] 
Hence  $DU^2 \subset l^2_1 DU^2$. 

Now suppose that $ u \in l^2 DU^2$ and define 
\[   a(t) =   \int_{j}^{j+1} u ds \] 
if $t \in [j+j+1)$. Then $a \in L^2$ and it remains to verify that 
\[  \Vert u-a \Vert_{DU^2} \le c \Vert u \Vert_{l^2(DU^2)} \] 
But this is easily verified for atoms.  
\end{proof}

\newsection{The NLS hierarchy} 

\label{appNLS} 

In this section we collect material on the NLS hierarchy. It well
known in the inverse scattering community and it is logically
independent of the rest of the paper, but we include it to give a more
complete picture. We refer to \cite{MR1462745} for more detail.
                             
The central element of the hierarchies is the Lax operator or, equivalently, the spatial part of the zero curvature condition, which for the focusing NLS is
\begin{equation}   \Psi_x = \left(\begin{matrix}  -iz & u \\ -\bar u & i z \end{matrix} \right)\label{spatial}  \end{equation}  
Here $z$ is a complex number. The matrix is affine in $z$, 
\[  \left(\begin{matrix}  -iz & u \\ -\bar u & i z \end{matrix} \right)= 
z \left(\begin{matrix}  -i & 0 \\ 0  & i  \end{matrix} \right) +  \left(\begin{matrix}  0 & u \\ -\bar u & 0 \end{matrix} \right) \] 
and with matrix coefficients in $su(2)$. We search for equations 
\begin{equation}\label{temporal} \partial_t \psi = B(z,u)
  \psi \end{equation}
so that the compatibility condition between \eqref{spatial} and \eqref{temporal} is an equation for $u$. More precisely we will choose $B$ so that 
\[ \left[\partial_t - B(z,u) ,    \partial_x - \left(\begin{matrix}  -iz & u \\ -\bar u & i z \end{matrix} \right)\right] 
= \left( \begin{matrix} 0 & \partial_t u - F \\ -\overline{ \partial_t u -F}  & 0 \end{matrix} \right) 
\] 
where $F$ is a polynomial in $u$, $\bar u$ and their derivatives. It is important that $F$ is  independent of $z$.

We make the Ansatz
\[  B_k(z,u) = \sum_{j=0}^k  z^{k-j} Q_j(u) \] 
with $Q_j\in su(2)$ where $k \ge 0$. Equation \eqref{spatial} and \eqref{temporal} are compatible if 
\[  
\partial_t  \left(\begin{matrix}  -iz & u \\ -\bar u & i z \end{matrix} \right) - \partial_x  \sum_{j=0}^k  z^{k-j} Q_j(u) + \left[ \left(\begin{matrix}  -iz & u \\ -\bar u & i z \end{matrix} \right),  \sum_{j=0}^k  z^{k-j} Q_j(u)  \right]=0  \] 
It reduces to an equation for $u$ if the $Q_j$ satisfy 

\begin{equation} 
\begin{split} Q_0 = &  \left( \begin{matrix} -i & 0 \\ 0 & i \end{matrix} \right) \\  
  \left[ \left( \begin{matrix} -i & 0 \\ 0 & i \end{matrix} \right), Q_{k+1} \right] = & \partial_x Q_k + \left[ Q_k, \left( \begin{matrix} 0 & u\\-\bar u & 0 \end{matrix} \right) \right] \\ 
  \partial_x Q_k + &\, \left[ Q_k, \left( \begin{matrix} 0 & u\\-\bar u & 0 \end{matrix} \right) \right] \qquad \text{ is off diagonal. } 
\end{split} 
\end{equation} 
 We denote  $ Q_k =  \left( \begin{matrix} -i r_k & p_k \\ -\bar p_k & ir_k \end{matrix} \right)$ and obtain the recursive formulas 
\begin{equation} p_{k+1} = \frac{i}2   p'_k  + r_k u \end{equation}  
\begin{equation} r_k' = i ( p_k \bar u - \bar p_k u).  \end{equation}
We hope to obtain polynomial expressions for $r_k$ in $u$ and its
derivatives. This is indeed true, but not so easy to prove. On the
other hand if we recursively solve the equations and if we obtain
polynomials expressions then the zero curvature conditions yields
equations for $u$. It is not difficult to calculate some of the $Q_k$.
By assumption  $r_0=1$ and  $p_0=0$. Next   $p_1= u $ and  $r_1=0$. 
It is not much harder to obtain   $p_2= \frac{i}2 u'$, $r_2= -\frac{1}2 |u|^2 $ and   $p_3=-\frac{1}4 (u'' + 2|u|^2 u)$,     $r_3=\frac{i}4( u'\bar u - u \bar u') $. 
  The operators are 
\begin{eqnarray*}   \psi_t =& \left( \begin{matrix} -i & 0 \\ 0 & i \end{matrix} \right)\psi , \qquad & \text{ for } k=0  \\
\psi_t =& \left( \begin{matrix} -iz & u \\ -\bar u  & iz \end{matrix} \right)\psi, \qquad & \text{ for } k=1 \\
\psi_t =& 2\left( \begin{matrix} -iz^2 +\frac{i}2 |u|^2   & zu+\frac{i}2 u'  \\ -z\bar u+\frac{i}2 \bar u'   & iz^2-\frac{i}2 |u|^2  \end{matrix} \right)\psi, \qquad & \text{ for } k=2 \\ 
\psi_t =& 4\left( \begin{matrix} -iz^3 +\frac{iz}2 |u|^2+\frac{1}4 (u'\bar u -u \bar u' )    & z^2u+\frac{iz}2 u' -\frac{1}4 (u'' +2|u|^2 u)   \\[2mm] -z^2\bar u+\frac{iz}2 \bar u'+\frac{1}4 (\bar u'' +2|u|^2 \bar u)   & iz^3-i z|u|^2  -\frac{1}4 (u'\bar u -u \bar u' )  \end{matrix} \right)\psi , \qquad & \text{ for } k=3
\end{eqnarray*} 
where we have chosen factors so that we arrive at normalizations we had chosen before.

We obtain the equations for the defocusing hierarchy by replacing $\bar u$ by $-\bar u$ and the KdV hierarchy by replacing $\bar u $ by $1$. 
The status of wellposedness for the full hierarchy seems not well
studied: Gr\"unrock \cite{MR2653659} considered the KdV and mKdV hierarchy (which is the hierarchy of real equations given as half of the NLS hierarchy). Most likely his results extend with marginal changes to the full NLS hierarchy.
It is true though not obvious that the compability conditions are Hamiltonian equations with the symplectic form  
\[ (u,v) \to \real \int i \bar u v dx \] 
and Hamiltonians given by 
\[ -\frac1{k+1}  \int_{-\infty}^\infty \trace \left\{Q_{k+2}\left( \begin{matrix} -i & 0 \\ 0 & i \end{matrix} \right)\right\} dx       ,\] 
see \cite{MR1462745}. These flows
commute, which  can be seen by the action on the reflection coefficient,
or by the fact that $T(z)$ and $T(z')$ are in involution (see Faddeev and Takhtajan
\cite{MR905674}). This last formulation has the advantage that it is well
defined even without decay assumptions on the initial data.

\bigskip{ \bf Acknowledgements:} This research was carried out in part
while the authors were visiting the Hausdorff Research Institute for
Mathematics in Bonn in summer 2014, and in part while the authors were
visiting MSRI in Berkeley in Fall 2015.  The first author was supported by the DFG through SFB 1060. The second author was
supported by the NSF grant DMS-1266182 and by a Simons Investigator
grant from the Simons Foundation.

 % \printbibliography
\bibliography{nls} 
\bibliographystyle{abbrv} 
\end{document}